\documentclass[11pt,a4paper]{article}
\usepackage[utf8]{inputenc}
\usepackage{amsmath}
\usepackage{amsfonts}
\usepackage{amssymb}
\usepackage{graphicx}
\usepackage{color}
\usepackage[left=2cm,right=2cm,top=2cm,bottom=2cm]{geometry}
\usepackage{placeins}
\usepackage{amsthm}
\usepackage[makeroom]{cancel}
\RequirePackage[numbers]{natbib}
\usepackage{hyperref}
\usepackage{ulem}
\usepackage{stackrel}

\usepackage{pgffor}
\usepackage{multicol,multirow}
\usepackage{algorithm}
\usepackage[noend]{algpseudocode}
\makeatletter
\def\BState{\State\hskip-\ALG@thistlm}
\def\BStatex{\Statex\hskip-\parindent \hskip2ex}

\usepackage{tikz}
\usepackage{pgfplots}

\newtheorem{theorem}{Theorem}
\newtheorem{definition}{Definition}
\newtheorem{proposition}{Proposition}
\newtheorem{remark}{Remark}
\newtheorem{condition}{Condition}
\newtheorem{lemma}{Lemma}
\newtheorem{corollary}{Corollary}

\newcommand{\subdiv}{S}
\newcommand{\subdivset}{\mathcal{S}}
\newcommand{\clos}{F}
\newcommand{\epi}{\text{epi}}
\newcommand{\hull}{\text{hull}}

\newcommand{\interpset}{\mathcal{I}}
\newcommand{\ineqset}{\mathcal{C}}
\newcommand{\strictineqset}{\overset{\bullet}{\ineqset}}
\newcommand{\Hilb}{\mathcal{H}}
\newcommand{\extfun}{P_{\clos \to [0,1]^d}}
\newcommand{\multiset}{\mathcal{L}_\subdiv}
\newcommand{\multi}{\underline{\ell}}

\newcommand{\multiknot}{t^{(\subdiv)}_{\multi}}
\newcommand{\multiknotp}{t^{(\subdiv)}_{\multi'}}
\newcommand{\R}{\mathbb{R}}
\newcommand{\domain}{[0, 1]^d}
\newcommand{\domainD}{[0, 1]^D}
\newcommand{\prodF}[1]{\clos^{1:#1}}

\newcommand{\activeset}{\mathcal{J}}
\newcommand{\addknot}[1]{\, \cup_{#1} \,}
\newcommand{\addvar}[1]{\, + \,}

\author{}

\title{Convergence}

\definecolor{green}{rgb}{0, 0.5, 0}

\begin{document}

\title{Sequential construction and dimension reduction of Gaussian processes under constraints}

\author{Fran\c{c}ois Bachoc$^\ast$, Andr\'es F. L\'opez Lopera$^{\dagger}$ and Olivier Roustant$^\ddagger$\\
${}^{\ast}$Inst. de Math\'ematiques de Toulouse (IMT), UMR5219 CNRS, 31062 Toulouse, France \\
${}^{\dagger}$Univ. Polytechnique Hauts-de-France, C\'eramaths, F-59313 Valenciennes, France \\
${}^{\ddagger}$IMT, UMR5219 CNRS, INSA, 31077 Toulouse c\'edex 4, France
}

\maketitle

\abstract{
		Accounting for inequality constraints, such as boundedness, monotonicity or convexity, is challenging when modeling costly-to-evaluate black box functions. In this regard, finite-dimensional Gaussian process (GP) regression models bring a valuable solution, as they guarantee that the inequality constraints are satisfied everywhere. Nevertheless, these models are currently restricted to small dimensional situations (up to dimension $5$). Addressing this issue, we introduce the MaxMod algorithm that sequentially inserts one-dimensional knots or adds active variables, thereby performing at the same time dimension reduction and efficient knot allocation. We prove the convergence of this algorithm. In intermediary steps of the proof, we propose the notion of multi-affine extension and study its properties. We also prove the convergence of finite-dimensional GPs, when the knots are not dense in the input space, extending the recent literature. With simulated and real data, we demonstrate that the MaxMod algorithm remains efficient in higher dimension (at least in dimension $20$), and needs fewer knots than other constrained GP models from the state-of-the-art, to reach a given approximation error.
}

\section{Introduction} \label{section:intro}

Gaussian processes (GPs) are widely used to address diverse applications since they form a flexible prior over functions \cite{Rasmussen2005GP,stein1999interpolation}. They have been successfully applied in research fields such as numerical code approximations \cite{Sacks89Design}, global optimization \cite{Jones1998EGO}, model calibration \cite{kennedy2001bayesian},
geostatistics \cite{chiles2009geostatistics,mu2019intrinsic} and machine learning \cite{Rasmussen2005GP}. 

It is known that accounting for inequality constraints (e.g. positivity, monotonicity, convexity) in GPs leads to smaller prediction errors and to more realistic uncertainties \cite{DaVeiga2012GPineqconst,DaVeiga2020GPineqconst,Golchi2015MonotoneEmulation,LopezLopera2017FiniteGPlinear,maatouk2017gaussian,Pallavi2019BayesianShapeGPs,Riihimaki2010GPwithMonotonicity}. 
These inequality constraints correspond to available information on functions over which GP priors are considered. They are encountered in diverse research fields such as social system analysis \cite{Riihimaki2010GPwithMonotonicity}, computer networking \cite{Golchi2015MonotoneEmulation}, econometrics \cite{Cousin2016KrigingFinancial}, geostatistics \cite{maatouk2017gaussian}, nuclear safety criticality assessment \cite{LopezLopera2017FiniteGPlinear}, tree distributions \cite{LopezLopera2019GPCox}, coastal flooding \cite{LopezLopera2019lineqGPNoise}, and nuclear physics \cite{Zhou2019ProtonConstrGPs}. Note also that, beyond GPs, regression methods that account for these inequality constraints, also called shape constraints, are acknowledged as an important need in statistics and machine learning \cite{bellec2018sharp,durot2002sharp,durot2018limit,groeneboom2014nonparametric,groeneboom2001estimation,hornung1978monotone,lin2014bayesian}.

Among the existing approaches which enable us to impose inequality constraints to GP models, we focus on those based on the approximation of GP samples in finite-dimensional spaces of functions such as piecewise linear functions \cite{Bachoc2010cMLE,Cousin2016KrigingFinancial,LopezLopera2017FiniteGPlinear,maatouk2017gaussian,Zhou2019ProtonConstrGPs}.
Indeed, the main benefit of these approaches is that they guarantee the inequality constraints to be satisfied everywhere in the input space. For instance, in the case of boundedness constraints, this means that the realizations from the posterior distribution of a constrained GP model obtained from these approaches are above or below the prescribed bounds everywhere in the input space. In contrast, realizations from the posterior distributions obtained from several other GP-based approaches are guaranteed to be above or below the prescribed bounds only at a limited number of selected input points (see, e.g., \cite{DaVeiga2012GPineqconst,DaVeiga2020GPineqconst} and the work in \cite{Riihimaki2010GPwithMonotonicity} accounting for monotonicity constraints). Hence, the approaches based on the approximation of GP samples in finite-dimensional spaces take into account the full information of the inequality constraints. 
Furthermore, in practice, domain experts may consider that a statistical model on a black box function is more trustworthy and physically interpretable if it respects known inequality constraints everywhere in the input space. For instance, in the application of Section \ref{section:numerical:experiments:subsec:BRGM}, a prediction of the flooded area is expected to be non-negative and non-decreasing with respect to the tide and surge inputs.
From this point of view, the approximation of GP samples in finite-dimensional spaces is one of the few admissible methods.

Nevertheless, the main drawback of approximating GP samples in finite-dimen\-sional spaces is the scalability to high-dimensional input spaces. Indeed, the finite-dimensional spaces rely on basis functions, each of them being centered at a $D$-dimensional knot, with $D$ being the input space dimension.
These $D$-dimensional knots need to be obtained from the tensorization of $D$ sets of one-dimensional knots, in order to satisfy the constraints everywhere. According to the state-of-the-art, for instance \cite{LopezLopera2019lineqGPNoise,LopezLopera2017FiniteGPlinear,maatouk2017gaussian}, the sets of knots are fixed a priori for each of the $D$ inputs. This limits the applicability to small dimension, say, $3$ to $5$. 

In this paper, we overcome this limitation, in situations where the dimension $D$ is allowed to be significantly larger (for instance $D = 20$ in Section \ref{section:numerical:experiments}), but where there are many irrelevant input variables, or, in other words, the effective dimension is small. We suggest a sequential procedure for knot insertion and variable selection that is scalable to these situations of higher dimensional input spaces with irrelevant variables. The procedure leverages three important intuitive principles. First, there should be a higher concentration of knots in input regions where the function is varying most. Second, the most influential variables should be allocated the most knots. Third, weakly influential variables should be allocated one-dimensional knots with the least priority. 

Let us now describe the sequential procedure.
Consider a set of $n$ input points and $n$ corresponding observations of the function of interest, to be interpolated, as well as given inequality constraints.
We start with a coarse finite-dimensional GP model based on few active variables and small sets of one-dimensional knots for them. Then, at each step, we either add a new active variable or insert a one-dimensional knot to a variable that is already active. The new variable or new knot is the one that corresponds to the largest modification, in $L^2$ norm, of the maximum a posteriori (MAP) function, also called the mode, of the constrained GP model. For this reason the suggested sequential procedure is called the MaxMod (maximum modification) algorithm, hereafter called MaxMod. 
The MAP is the most probable function that interpolates the observations and that satisfies the inequality constraints everywhere, according to the finite-dimensional constrained GP model. As shown in \cite{maatouk2017gaussian}, computing the MAP yields a convex optimization problem of moderate complexity. Here, we provide a computationally simple expression of the subsequent $L^2$ norm (in Appendix \ref{subsection:computing:Ldeux:difference:modes}), resulting in a computational complexity that is linear in the number of multi-dimensional knots. We allow for free locations of one-dimensional knots by using asymmetric hat basis functions, instead of the symmetric ones investigated in \cite{LopezLopera2017FiniteGPlinear,maatouk2017gaussian}.
The sequential procedure also naturally incorporates a penalization for adding new variables, or for inserting one-dimensional knots, that overly increase the total number of multi-dimensional knots.

From the point of view of free knot insertion in spline approximation, MaxMod differs from many existing references
\cite{Creutzig2007FreeKnotSplineStoch,DeBoor2001GuideSplines,DeBoor2002SplineBasics,Goldman2003BSplineApprox,Yingkang1993AlgoSplinesFreeKnots,Jupp1978ApproxDataSplinesFreeKnots,Kobbelt2002MultiResTechniques,Slassi2014OptFreeKnotSplineStochDiffEq}
(that typically address spline approximation independently of GPs and inequality constraints).
Indeed, these references are based on directly evaluating and minimizing the approximation error of a target function and thus rely on multiple evaluations of this function. In contrast, MaxMod is adapted to the situation where evaluations of the target function are scarce, and it simply maximizes the difference between successive spline approximations.

We provide a convergence guarantee for MaxMod. We consider the set of $n$ input points and function observations to be fixed and we let the number of iterations go to infinity. This corresponds to increasing the computational budget, as measured by the number of multi-dimensional knots. Then, we show that all the variables are eventually activated and that the set of $D$-dimensional knots becomes dense in the input space. This implies, based on \cite{bay2016generalization,bay2017new}, that the MAP function obtained from MaxMod converges to the optimal constrained interpolant function in the reproducing kernel Hilbert space (RKHS, see for instance \cite{berlinet2011reproducing}) of the covariance function of the GP model. Hence, the convergence result states that MaxMod, despite being a sequential procedure, becomes globally efficient as the number of iterations increases. In particular, loosely speaking, the procedure does not fall into an undesirable local pattern, where the inserted knots would cluster and would not eventually cover the whole input space. It is, in general, important to ensure that sequential procedures avoid undesirable local patterns \cite{bect2019supermartingale,ben2017universal}.

In order to obtain the convergence result, we extend the results of \cite{bay2016generalization,bay2017new}. These results show that, given a dense sequence of multi-dimensional knots,
the MAP of a constrained finite-dimensional GP converges to the above discussed optimal constrained interpolant function in the RKHS of the GP covariance function. The extension tackles the case where the multi-dimensional knots are not dense. Based on the subset $\clos$ of the input space corresponding to their closure, we define a transformation that we call the multiaffine extension, that extends a function defined on $\clos$ to the entire input space. This extension enables us to define an optimal constrained interpolant function based on a new RKHS restricted to $\clos$. Then, we show that the MAP converges to this optimal constrained interpolant function, thus extending \cite{bay2016generalization,bay2017new} to any sequence of multi-dimensional knots, not necessarily dense on the entire input space. The construction and properties of the multiaffine extension and this extension of \cite{bay2016generalization,bay2017new} may be of independent interest. Also, this general proof scheme for the convergence of MaxMod may be adapted to the convergence of other algorithms based on hat basis functions.

The benefit of the suggested sequential procedure is shown in a series of numerical experiments,  provided in Section \ref{section:numerical:experiments}, with simulated and real-world data. For the latter, data come from a coastal flooding application (see \cite{Azzimonti2018CoastalFlooding,LopezLopera2019lineqGPNoise}) satisfying both positivity and monotonicity constraints. We test the versatility of MaxMod for efficiently inserting knots or adding active dimensions while reducing the approximation error of the resulting constrained GP. We demonstrate that MaxMod remains tractable and yields a constrained GP model with accurate predictions, even up to the dimension $D=20$, for which the state-of-the-art procedures either are intractable \cite{LopezLopera2017FiniteGPlinear,maatouk2017gaussian} or do not satisfy the constraints everywhere \cite{DaVeiga2020GPineqconst,Riihimaki2010GPwithMonotonicity}. Even in smaller dimension, when the procedures of \cite{LopezLopera2017FiniteGPlinear,maatouk2017gaussian} are tractable, MaxMod typically needs fewer knots to achieve a comparable approximation error. We also show the benefit of having fewer knots when subsequently computing confidence intervals from constrained GP models.

This paper is organized as follows. In Section \ref{sec:finite:dimensional:constrained:GP}, we describe the finite-dimensional GP approach proposed in \cite{LopezLopera2017FiniteGPlinear,maatouk2017gaussian}, that we adapt to the case where only a subset of the $D$ variables is active. In Section  \ref{section:MaxMod}, we introduce MaxMod. In Section \ref{section:convergence}, we present the multiaffine extension and establish the various convergence results. 
The numerical
experiments are carried out in Section \ref{section:numerical:experiments}.
Section \ref{section:conclusion} concludes the paper. 
Finally, in the appendix, we provide technical developments, some of the technical conditions, and all the proofs of the paper.

\section{Finite-dimensional constrained Gaussian processes} \label{sec:finite:dimensional:constrained:GP}

For convenience, we have summarized the notations of Sections   \ref{sec:finite:dimensional:constrained:GP} and \ref{section:MaxMod} in Table \ref{table:list:of:symbols}, located at the end of Section \ref{section:MaxMod}.

\subsection{Basis function decomposition}

The principle of finite-dimensional constrained GPs is to consider linear combinations of basis functions which are tensorizations of one-dimensional asymmetric hat basis functions. These one-dimensional basis functions are parametrized by $-\infty < u <v <w < +\infty$, and are written $\phi_{u,v,w} : \R \to \R$, defined by,  \label{page:phi:u:v:w}
\[
\phi_{u,v,w}(t)
= 
\begin{cases}
	\frac{1}{v-u}( t-u  )
	& \text{for  $u \leq t \leq v$} \\
	\frac{1}{w-v}( w-t  )
	& \text{for  $v \leq t \leq w$} \\
	0 & \text{for $t \not \in [u,w]$},
\end{cases}
\]
for $t \in \R$.
Clearly, $\phi_{u,v,w} $ is a `hat' function centered at $v$ and with support $[u,w]$.

Still in dimension one, we now explain how a set of basis functions can be defined from the notion of subdivision.
A (one-dimensional) subdivision is a set of (one-dimensional) knots $\subdiv = \left\lbrace t^{(\subdiv)}_0, \ldots, t^{(\subdiv)}_{m_{\subdiv}+1} \right\rbrace$. \label{page:one:dim:subdiv}
We let $t^{(\subdiv)}_{(0)} \leq \dots \leq t^{(\subdiv)}_{(m_{\subdiv}+1)} $ be the corresponding ordered knots. \label{page:ordered:knots}
For any subdivision, we assume that
$ 0 = t^{(\subdiv)}_{(1)} < \dots < t^{(\subdiv)}_{(m_{\subdiv})} = 1$ and,
to deal with boundary issues, we set $t^{(\subdiv)}_{(0)} = -1$  and $t^{(\subdiv)}_{(m_{\subdiv}+1)} = 2$. 
With this convention, the smallest subdivision containing $0$ and $1$, denoted as $S^0$, is $\{-1, 0, 1, 2 \}$.

Now, let $D \in \mathbb{N}$ be the ambient (potentially large) dimension and consider a given set of active variables $\activeset = (a_1, \dots, a_d) \subseteq \{1 , \ldots , D\}$, with $1 \leq a_1 < \dots <a_d \leq D$, of size $d \leq D$. \label{page:D:d:active:set}
We now explain how to construct $d$-dimensional basis functions of the $d$ active variables. We define a $d$-dimensional subdivision (indexed by $\activeset$) as a vector of 
one-dimensional subdivisions $\subdiv = (\subdiv_{a_1}, \ldots, \subdiv_{a_d} )$. \label{page:vector:subdivisions}
For convenience, we may identify $S$ with the set of $d$-dimensional knots $\prod_{j=1}^d S_{a_j} \subseteq \mathbb{R}^d$.
We denote by $\subdivset_\activeset$ the set of $d$-dimensional subdivisions indexed by $\activeset$.

For conciseness, we use tensor notation. 
Thus, the notation for a multi-index is $\multi = (\ell_{a_1}, \dots, \ell_{a_d}) \in \mathbb{N}^d$.
For a subdivision $\subdiv \in \subdivset_\activeset$, the associated set of multi-indices is denoted  \label{page:set:multi:indices}
\[
\multiset = \{ \multi ; \, \ell_{i} \in \{1, \ldots, m_{\subdiv_{i}}\}, i \in \activeset \},
\]
and $A_\subdiv = \{ \alpha_{\multi} \in \mathbb{R} ; \multi \in \multiset \}$ is the set of associated real-valued sequences.
\label{page:A:subdiv}
For a given multi-index $\multi \in \multiset$, 
the associated vector of knots is denoted $\multiknot = \left( t^{(\subdiv_{a_1})}_{(\ell_{a_1})},\ldots,t^{(\subdiv_{a_d})}_{(\ell_{a_d})} \right)$. We call this vector a $d$-dimensional knot. \label{page:d:dim:knot}
Then, to any $\multi \in \multiset $ we associate the $d$-dimensional basis function defined by tensorization, for $t = (t_{a_1},\ldots,t_{a_d}) \in \domain$, \label{page:phi:multi}
\[
\phi_{\multi}^{(\subdiv)}(t) = 
\prod_{i=1}^d
\phi_{t^{(\subdiv_{a_i})}_{(\ell_{a_i}-1)},t^{(\subdiv_{a_i})}_{(\ell_{a_i})},t^{(\subdiv_{a_i})}_{(\ell_{a_i}+1)}}(t_{a_i}).
\]
This basis function is $1$ at the knot $( t^{(\subdiv_{a_1})}_{(\ell_{a_1})}, \ldots , t^{(\subdiv_{a_d})}_{(\ell_{a_d})} )$ and has support equal to the hypercube $\prod_{i=1}^d
[t^{(\subdiv_{a_i})}_{(\ell_{a_i}-1)},t^{(\subdiv_{a_i})}_{(\ell_{a_i}+1)}]$.

In Figure \ref{figure:illustration:knots}, we now provide an illustrative example. Consider a total of $D=4$ input variables, among which the $d=2$ variables $1$ and $3$ are active. We thus have the active set $\activeset = (a_1,a_2) = (1,3)$. Consider that for the first active variable, the subdivision is $\subdiv_1 = \{ -1,0,1/3,1,2 \}$. In $\subdiv_1$, the knots $-1$ and $2$ are only present to deal with boundary issues and the knots $t^{(\subdiv_1)}_{(1)} = 0$, $t^{(\subdiv_1)}_{(2)} = 1/3$ and $t^{(\subdiv_1)}_{(3)} = 1$ are displayed as two blue bullets and one blue triangle on the left panel. The basis functions  $\phi_{-1,0,1/3}$, $\phi_{0,1/3,1}$ and $\phi_{1/3,1,2}$ centered at these latter three knots are also displayed. 
Similarly for the second active variable, the subdivision is $\subdiv_3 = \{ -1,0,1/4,1/2,\linebreak[1]1,2 \}$ (this subdivision is written $\subdiv_3$ because the second active variable is the variable $3$). The four knots $t^{(\subdiv_3)}_{(1)} = 0$, $t^{(\subdiv_3)}_{(2)} = 1/4$, $t^{(\subdiv_3)}_{(3)} = 1/2$ and $t^{(\subdiv_3)}_{(4)} = 1$ are displayed at the middle panel, as well as the basis functions  $\phi_{-1,0,1/4}$, $\phi_{0,1/4,1/2}$, $\phi_{1/4,1/2,1}$ and $\phi_{1/2,1,2}$ centered at these knots.

Then, consider  the two-dimensional space formed by the two active variables with subdivision $\subdiv = (\subdiv_1 , \subdiv_3 )$. Its associated set of multi-indices $\multiset$ contains the $3 \times 4 = 12$ indices $(1,1) , (1,2) , \ldots , (3,4)$. The corresponding $12$ bi-dimensional knots are displayed on the right. The bi-dimensional knot represented as a triangle is $t ^{(\subdiv)}_{(2,3)} = (1/3,1/2)$ with multi-index $\multi=(2,3)$. It corresponds to the triangle one-dimensional knots $t^{(\subdiv_1)}_{(2)} = 1/3$ and $t^{(\subdiv_3)}_{(3)} = 1/2$. The bi-dimensional basis function $\phi_{\multi}^{(\subdiv)}$ on the right panel is a pyramid which top is located at the blue triangle and which vertices are the dashed lines.

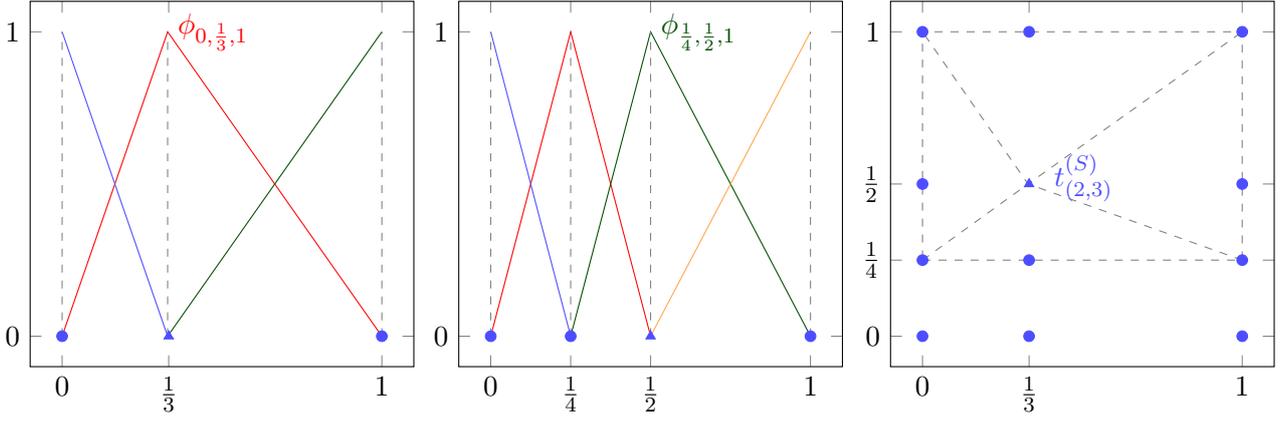
\begin{figure} 
	\centering
	\begin{tikzpicture}
		\begin{axis}[scale=1,%
			width=0.39\textwidth,
			height=0.25\textheight,
			xtick = {0, 1/3, 1},
			xticklabels = {$0$, $\frac{1}{3}$, $1$},
			ytick = {0, 1},
			scatter/classes={%
				a={mark={.},draw=black}}]
			\addplot[red, scatter,
			scatter src=explicit symbolic]%
			table[meta=label] {
				x y label
				0 0 a
				0.33 1 a
				1 0 a
			};
			\addplot[blue!70!white,
			scatter src=explicit symbolic]%
			table[meta=label] {
				x y label
				0 1 b
				0.33 0 b
			};
			\addplot[green!60!black,
			scatter src=explicit symbolic]%
			table[meta=label] {
				x y label
				0.33 0 c
				1 1 c
			};
			\addplot[gray, dashed,
			mark = ,%
			scatter src=explicit symbolic]%
			table[meta=label] {
				x y label
				0 0 d
				0 1 d
			};
			\addplot[gray, dashed,
			mark = ,%
			scatter src=explicit symbolic]%
			table[meta=label] {
				x y label
				0.33 0 e
				0.33 1 e
			};
			\addplot[gray, dashed,
			mark = ,%
			scatter src=explicit symbolic]%
			table[meta=label] {
				x y label
				1 0 f
				1 1 f
			};
			\addplot[blue!70!white,mark=*] coordinates {(0,0)};
			\addplot[blue!70!white,mark=triangle*] coordinates {(1/3,0)};
			\addplot[blue!70!white,mark=*] coordinates {(1,0)};
			\addplot[red,mark=] coordinates {(1/3,1)};
			\node[right, red] at
			(33,100) {$\phi_{0,\frac{1}{3}, 1}$};
		\end{axis}
	\end{tikzpicture}
	\begin{tikzpicture}
		\begin{axis}[scale=1,%
			width=0.39\textwidth,
			height=0.25\textheight,
			xtick = {0, 1/4, 1/2, 1},
			xticklabels = {$0$, $\frac{1}{4}$, $\frac{1}{2}$, $1$},
			ytick = {0, 1},
			scatter/classes={%
				a={mark={.},draw=black}}]
			\addplot[red, scatter,
			scatter src=explicit symbolic]%
			table[meta=label] {
				x y label
				0 0 a
				0.25 1 a
				0.5 0 a
			};
			\addplot[blue!70!white,
			scatter src=explicit symbolic]%
			table[meta=label] {
				x y label
				0 1 b
				0.25 0 b
			};
			\addplot[green!60!black,
			scatter src=explicit symbolic]%
			table[meta=label] {
				x y label
				0.25 0 c
				0.5 1 c
				1 0 c
			};
			\addplot[orange!70!white,
			scatter src=explicit symbolic]%
			table[meta=label] {
				x y label
				0.5 0 d
				1 1 d
			};
			\addplot[gray, dashed,
			mark = ,%
			scatter src=explicit symbolic]%
			table[meta=label] {
				x y label
				0 0 e
				0 1 e
			};
			\addplot[gray, dashed,
			mark = ,%
			scatter src=explicit symbolic]%
			table[meta=label] {
				x y label
				0.25 0 f
				0.25 1 f
			};
			\addplot[gray, dashed,
			mark = ,%
			scatter src=explicit symbolic]%
			table[meta=label] {
				x y label
				0.5 0 g
				0.5 1 g
			};
			\addplot[gray, dashed,
			mark = ,%
			scatter src=explicit symbolic]%
			table[meta=label] {
				x y label
				1 0 h
				1 1 h
			};
			\addplot[blue!70!white,mark=*] coordinates {(0,0)};
			\addplot[blue!70!white,mark=*] coordinates {(1/4,0)};
			\addplot[blue!70!white,mark=triangle*] coordinates {(1/2,0)};	
			\addplot[blue!70!white,mark=*] coordinates {(1,0)};
			\addplot[blue!70!white,mark=] coordinates {(1/3,1)};
			\node[right, green!60!black] at
			(50,100) {$\phi_{\frac{1}{4},\frac{1}{2}, 1}$};
		\end{axis}
	\end{tikzpicture}
	\begin{tikzpicture}
		\begin{axis}[scale=1,%
			width=0.39\textwidth,
			height=0.25\textheight,
			xtick = {0, 1/3, 1},
			xticklabels = {$0$, $\frac{1}{3}$, $1$},
			ytick = {0, 1/4, 1/2, 1},
			yticklabels = {$0$, $\frac{1}{4}$, $\frac{1}{2}$, $1$},	
			scatter/classes={%
				a={mark={.},draw=black}}]
			\addplot[dashed, gray,
			mark = ,%
			scatter src=explicit symbolic]%
			table[meta=label] {
				x y label
				0 0.25 a
				1 0.25 a
				1 1 a
				0 1 a
				0 0.25 a
			};
			\addplot[dashed, gray,
			mark = ,%
			scatter src=explicit symbolic]%
			table[meta=label] {
				x y label
				0 0.25 b
				0.33 0.5 b
				0 1 b
			};
			\addplot[dashed, gray,
			mark = ,%
			scatter src=explicit symbolic]%
			table[meta=label] {
				x y label
				1 0.25 c
				0.33 0.5 c
				1 1 c
			};
			\addplot[blue!70!white,mark=*] coordinates {(0,0)};
			\addplot[blue!70!white,mark=*] coordinates {(1/3,0)};
			\addplot[blue!70!white,mark=*] coordinates {(1,0)};	
			\addplot[blue!70!white,mark=*] coordinates {(0,1/4)};
			\addplot[blue!70!white,mark=*] coordinates {(1/3,1/4)};	
			\addplot[blue!70!white,mark=*] coordinates {(1,1/4)};
			\addplot[blue!70!white,mark=*] coordinates {(0,1/2)};
			\addplot[blue!70!white,mark=*] coordinates {(1,1/2)};
			\addplot[blue!70!white,mark=*] coordinates {(0,1)};
			\addplot[blue!70!white,mark=*] coordinates {(1/3,1)};	
			\addplot[blue!70!white,mark=*] coordinates {(1,1)};
			\addplot[blue!70!white,mark=triangle*] coordinates {(1/3,0.5)};	
			\node[right, blue!70!white] at
			(38,52) {$t_{(2,3)}^{(S)}$};
		\end{axis}
	\end{tikzpicture}
	\vspace{-2ex}
	\caption{Illustration of the knots and basis functions.}
	\label{figure:illustration:knots}
\end{figure}

Next, let us construct the finite-dimensional space of functions $E_\subdiv$ that contains finite-dimensional constrained GPs. We let $E_\subdiv$ be the linear space of functions $\domain \to \R$, spanned by the tensor basis functions $\phi_{\multi}^{(\subdiv)}$ with $\multi \in \multiset $. \label{page:E:subdiv}
Equivalently, $E_\subdiv$ is the space of multivariate splines of degree $1$,
constituted of componentwise piecewise linear functions, 
with knots defined by $S$.
For $\alpha \in A_\subdiv$, we let $Y_{\subdiv,\alpha}$ be the element of $E_\subdiv$ with coefficients $\alpha$: \label{page:Y:Subdiv:alpha}
\begin{equation} \label{eq:Y:S:alpha}
	Y_{\subdiv,\alpha} = \sum_{\multi \in \multiset} \alpha_{\multi} \phi_{\multi}^{(\subdiv)}.
\end{equation}
Note that $Y_{\subdiv,\alpha}(t_{\multi}^{(\subdiv)}) = \alpha_{\multi}$ for all $\multi \in \multiset$.
Finally, for a function $f: \domain \to \R$, 
we denote by $\pi_\subdiv(f)$ the projection of $f$ onto $E_\subdiv$: \label{page:pi:subdiv:f}
\begin{equation} \label{eq:piSm}
	\pi_\subdiv (f) = 
	\sum_{\multi \in \multiset} f(\multiknot) \phi_{\multi}^{(\subdiv)}.
\end{equation}
This projection is the componentwise piecewise linear function that coincides with $f$ at the knots of $\subdiv$.

\subsection{Function spaces for interpolation and inequality constraints}

We let $\mathcal{F}(A,\mathbb{R})$ (resp. $\mathcal{C}(A,\mathbb{R})$) be the set of functions (resp. continuous functions) from a subset $A$ of a finite-dimensional vector space to $\mathbb{R}$. 
For $d \in \{1,\ldots,D\}$, $U = (u_1,\ldots,u_n) \in ([0,1]^d)^n$ 
and $v^{(n)} = (v_1,\ldots,v_n) \in \R^n$, we write \label{page:I:U:vn}
\[
\interpset_{U,v^{(n)}} = \left\{ f: [0,1]^d \to \R ; f(u_i) = v_i ~ \text{for} ~ i=1,\ldots,n \right\},
\]
the set of $d$-dimensional functions which interpolate $v^{(n)}$ at $U$. Typically, a $D$-dimensional function of interest is observed at observation points $\delta_1,\ldots,\delta_n \in \domainD$, with values $v_1,\ldots,v_n$ and we let $u_i = (\delta_i)_{\activeset}$. Throughout the paper, we use the following notation: for $\activeset \subseteq \{1,\ldots,D\}$ and $x \in \mathbb{R}^D$, we let $x_\activeset$ be the vector extracted from $x$ by keeping only the components with indices in $\activeset$.

We treat the inequality constraints as a subset $\ineqset_D$ of $\mathcal{C}(\domainD , \mathbb{R})$. \label{page:C:D}
Note that, even if a constrained GP model has a subset $\activeset$ of active variables that is strictly smaller than $\{1 , \ldots , D \}$, the functions of this model are also functions of the full $D$ variables, and considered as such, they are required to belong to $\ineqset_D$.

Three classical examples of inequality constraints are given below,
corresponding respectively to boundedness, (componentwise) monotonicity and componentwise convexity:
\begin{eqnarray}
	\ineqset_D &=& \{ f \in \mathcal{C}(\domainD , \mathbb{R}) ; \, 
	a \leq f(x) \leq  b ~ \text{for all} ~ x \in \domainD \},  \label{eq:boundedness:D} \\
	\ineqset_D &=& \{ f \in \mathcal{C}(\domainD , \mathbb{R}) ; \,
	f(u) \leq f(v) ~ \text{for all} ~ u,v \in \domainD, u \leq v \}, \label{eq:monotonicity:D} \\
	\ineqset_D &=& \{ f \in \mathcal{C}(\domainD , \mathbb{R}) ; \,
	\text{for all $i \in \{1,\ldots,D\}$,} ~ \text{for all $x_{\sim i} \in [0,1]^{D-1}$,} \label{eq:convexity:D} \\
	& & \hspace{4cm} \text{the function} ~ u_i \mapsto f(u_i , x_{\sim i}) ~ \text{is convex}  \}. \nonumber
\end{eqnarray}
For monotonicity, the notation used above, $u \leq v$, means $u_1 \leq v_1 , \ldots , u_D \leq v_D$. For componentwise convexity, 
for $t \in \domainD$, $i \in \{1, \dots, D \}$ and $u_i \in [0, 1]$, 
we denote $t_{\sim i}$ the vector obtained from $t$ by removing the coordinate $i$, 
and $(u_i, t_{\sim i})$ the vector obtained from $t$ by replacing the $i^{\text{th}}$ coordinate by $u_i$.

The inequality constraint set $\ineqset_D$ of functions in $ \mathcal{C}(\domainD , \mathbb{R})$ naturally yields a corresponding inequality constraint set for functions in $ \mathcal{C}([0,1]^{|\activeset|} , \mathbb{R})$, where for a finite set $\Theta$, we write $|\Theta|$  for its cardinality.  The inequality constraint set in $ \mathcal{C}([0,1]^{|\activeset|} , \mathbb{R})$ corresponding to $\ineqset_D$ is \label{page:C:activeset}
\[
\ineqset_{\activeset} = \left\{ g: [0,1]^{|\activeset|} \to \mathbb{R}; 
\left(
x \in \domainD \mapsto g(x_{\activeset}) 
\right)
\in \ineqset_D \right\}.
\]
The set $\ineqset_{\activeset}$ is the set of functions of the variables in $\activeset$, that yield functions in $\ineqset_D$ when extended with $D - |\activeset|$ inactive variables.

We finally introduce the next basic condition on the inequality constraint set.

\begin{condition}[constraint set topology] \label{cond:C:convex:closed}
	The set $\ineqset_{D}$ is convex and closed with the topology of uniform convergence.
\end{condition}

Note that Condition \ref{cond:C:convex:closed} holds when $\ineqset_{D}$ is given by
one of \eqref{eq:boundedness:D}, \eqref{eq:monotonicity:D} or \eqref{eq:convexity:D}. 

\subsection{Finite-dimensional Gaussian processes under interpolation and inequality constraints}
\label{subsection:finite:dim:GP}

\subsubsection{Finite-dimensional Gaussian processes}

We consider a continuous GP $\xi_D$ indexed by $\domainD$, with mean zero and continuous covariance function $k_D$ on $\domainD \times \domainD$. The covariance function $k_D$ can be restricted to an active set of variables $\activeset \subseteq \{1 , \ldots , D\}$ as follows.
We let $ - \activeset = \{1 , \ldots , D\} \backslash \activeset$. We let $k_\activeset$ be the covariance function on $[0,1]^{|\activeset|} \times [0,1]^{|\activeset|}$ defined by, for $u,v \in [0,1]^{|\activeset|}$, $k_\activeset (u,v) = k_D( \widetilde{u} , \widetilde{v} )$, where $\widetilde{u}_{\activeset} = u$,  $\widetilde{u}_{ - \activeset} = 0$, $\widetilde{v}_{\activeset} = v$ and $\widetilde{v}_{ - \activeset} = 0$.
\label{page:k:D:K:J}Note that the choice of the value $0$ for the inactive variables, i.e. that are not in $\activeset$, is arbitrary.

We consider the continuous GP $\xi_\activeset$ on $[0,1]^{|\activeset|}$ defined by, for $u \in [0,1]^{|\activeset|}$, $\xi_\activeset (u) = \xi_D( \widetilde{u} )$, where $\widetilde{u} \in \domainD$ is defined by $\widetilde{u}_{\activeset} = u$ and  $\widetilde{u}_{ - \activeset} = 0$. \label{xi:active:set}
Then $\xi_\activeset$ has mean zero and covariance function $k_\activeset$.
Remark that $\xi_\activeset$ is obtained by freezing to $0$ the inputs of $\xi_D$ that are not in $\activeset$, which enables to model functions of the variables in $\activeset$.
Given a multi-dimensional subdivision, $\subdiv \in \subdivset_{\activeset}$, we consider the finite-dimensional GP 
\begin{equation} \label{eq:finite:dimensional:GP}
	\pi_\subdiv ( \xi_\activeset)
	= 
	\sum_{\multi \in \multiset}
	\xi_{\activeset}(\multiknot) \phi_{\multi}^{(\subdiv)}.
\end{equation}
This finite-dimensional GP only depends on the vector of values at the 
knots, $(\xi_{\activeset}(\multiknot))_{\multi \in \multiset}$.

Let us consider the covariance matrix of this vector, that we write $k_{\activeset}(\subdiv,\subdiv)$. \label{page:k:J:S:S}
Writing $\activeset = (a_1 , \ldots , a_d)$ with $1 \leq a_1 < \cdots < a_d \leq D$,
$k_{\activeset}(\subdiv,\subdiv)$ is the matrix of size 
$m_{\subdiv_{a_1}} \times \cdots \times m_{\subdiv_{a_d}}$, 
that we write in a multi-index way as, for $\multi, \multi' \in \multiset$, 
$k_{\activeset}(\subdiv,\subdiv)_{\multi, \multi'} 
= k_{\activeset} \left( \multiknot, \multiknotp \right).
$

Remark that with this multi-index writing, matrix products of the form 
$k_{\activeset}(\subdiv,\subdiv)k_{\activeset}(\subdiv,\subdiv)$, 
matrix inverses of the form $k_{\activeset}(\subdiv,\subdiv)^{-1}$ 
and matrix vector products of the form $k_{\activeset}(\subdiv,\subdiv) \alpha$ 
for $\alpha \in A_S$ can be defined by a straightforward extension 
of the corresponding operations for standard (single indexed) matrices and vectors.  Let us explain the matrix vector product case. 
Since $\multiset$ corresponds to the set $\prod_{i=1}^d \{1,\ldots,m_{\subdiv_{a_i}}\} $ and $\alpha = (\alpha_{\multi})_{\multi \in \multiset}$, the values of $\alpha$ can be indexed (arbitrarily) as $\alpha_1,\ldots,\alpha_N$, with $N = m_{\subdiv_{a_1}} \times \cdots \times m_{\subdiv_{a_d}}$. With the same indexing, the values in $k_{\activeset}(\subdiv,\subdiv)$ can be written as $(k_{i,j})_{i,j=1,\ldots,N}$. The standard matrix vector product between $(k_{i,j})_{i,j=1,\ldots,N}$ and $(\alpha_1,\ldots,\alpha_N)$ provides the vector $(\beta_1,\ldots,\beta_N)$. Using then the reverse indexing, we obtain the values $(\beta_{\multi})_{\multi \in \multiset}$ which are exactly the values of the (multi-indexed) matrix vector product $k_{\activeset}(\subdiv,\subdiv) \alpha$.

We assume that $k_{\activeset}(\subdiv,\subdiv)$ is invertible for all 
$\activeset \subseteq \{1,\ldots,D\}$ and $\subdiv \in \subdivset_{\activeset}$. 
This holds in particular when the matrix $[k_{D} (\delta_i,\delta_j)]_{i,j=1,\ldots,q}$ is invertible for any $q \in \mathbb{N}$ and $\delta_1,\ldots , \delta_q \in [0,1]^D$, two-by-two distinct. This is verified by most common covariance functions, for instance the squared exponential ones and those from the Mat\'ern class \cite{Rasmussen2005GP,stein1999interpolation}.

\subsubsection{Obtaining a finite number of linear inequality constraints for finite-dimensional Gaussian processes}
\label{subsubsection:obtaining:linear:inequality:constraints}

The main benefit of the finite-dimensional GP  $\pi_\subdiv ( \xi_\activeset)$ is that, for many classical inequality sets $\ineqset_{\activeset}$, obtained from $\activeset = (a_1,\ldots,a_d) \subseteq \{1 , \ldots , D\}$ and $\ineqset_D$, there exists an explicit (multi-indexed) matrix  $ M (\ineqset_{\activeset}) = (M (\ineqset_{\activeset})_{b,\multi})_{b=1,\ldots,B,\multi \in \multiset}$ and an explicit vector $v (\ineqset_{\activeset})  = (v( \ineqset_{\activeset} )_b)_{b=1,\ldots,B}$ such that \label{page:linear:constraints}
\begin{equation} \label{eq:equivalence:constraint}
	\pi_\subdiv ( \xi_\activeset)
	\in \ineqset_{\activeset}
	\Longleftrightarrow
	M (\ineqset_{\activeset} )
	(\xi_{\activeset}(\multiknot))_{\multi \in \multiset}
	\leq v( \ineqset_{\activeset} ),
\end{equation}
where again the definition of multi-indexed matrix-vector products is straightforward. 

Equation \eqref{eq:equivalence:constraint} provides $B$ linear inequality constraints on the vector of values of $\xi_\activeset$ at the knots $(\multiknot)_{\multi \in \multiset}$. Hence, the constraint $\pi_\subdiv ( \xi_\activeset)
\in \ineqset_{\activeset}$, that is a priori infinite-dimensional and intractable, boils down to simple linear inequality constraints. 

Consider for illustration the one-dimensional case, i.e. $d=1$, and let $\subdiv = \subdiv_{a_1}$. Then, in the case of boundedness constraints, it is shown in \cite{maatouk2017gaussian} that $\pi_\subdiv ( \xi_\activeset) \in \ineqset_{\activeset}$ if and only if
\[
\xi_{\activeset}(t_{(\ell)}^{(\subdiv)}) \in [a,b],
~  ~
\ell = 1 , \ldots , m_{\subdiv}.
\]
Hence, for boundedness constraints, \eqref{eq:equivalence:constraint} holds when $M (\ineqset_{\activeset})$ is the (single-indexed) $2 m_{\subdiv} \times m_{\subdiv}$ matrix $(-I_{m_{\subdiv}},I_{m_{\subdiv}})^\top$ and $v( \ineqset_{\activeset} )$ is $(-a,\ldots,-a,b,\ldots,b)^\top$ of size $2 m_{\subdiv} \times 1$.

In the case of monotonicity constraints, it is shown in \cite{maatouk2017gaussian} that $\pi_\subdiv ( \xi_\activeset) \in \ineqset_{\activeset}$ if and only if
\[
\xi_{\activeset}(t_{(\ell)}^{(\subdiv)})
\geq  
\xi_{\activeset}(t_{(\ell-1)}^{(\subdiv)}),
~  ~
\ell = 2 , \ldots , m_{\subdiv}.
\]

Hence under monotonicity constraints, \eqref{eq:equivalence:constraint} holds when
$v( \ineqset_{\activeset} )$ is the zero vector
and
$M (\ineqset_{\activeset})$ is the (single-indexed) $(m_{\subdiv}-1) \times m_{\subdiv}$ banded matrix
\[
\begin{pmatrix}
	1 & -1 & 0 &  \cdots & & 0 \\
	0 & 1 & -1 &  0  & \cdots & \vdots \\
	\vdots & \ddots & \ddots & \ddots & & \vdots \\
	\vdots &  & \ddots & \ddots & \ddots & 0 \\
	0 & \cdots & \cdots & 0 & 1 & -1 \\
\end{pmatrix}.
\]

In dimension $ |\activeset| = d > 1$, for boundedness and monotonicity constraints, the principle is the same but the notations become more cumbersome. In Appendix \ref{appendix:obtaining:linear:inequality:constraints}, we provide the expressions of $M (\ineqset_{\activeset})$ and $v( \ineqset_{\activeset} )$ for which \eqref{eq:equivalence:constraint} holds, for boundedness, monotonicity and componentwise convexity, in any dimension. These expressions follow from \cite{LopezLopera2017FiniteGPlinear,maatouk2017gaussian}, except for componentwise convexity in dimension larger than one, when they are not available in earlier references, to the best of our knowledge.

\subsubsection{The maximum a posteriori function}
\label{subsubsection:MAP}

Consider now some input locations $U = (u_1,\ldots,u_n) \in ([0,1]^d)^n$ and some corresponding observations $v^{(n)} = (v_1,\ldots,v_n) \in \R^n$. 
We are interested in the conditional distribution of $\pi_\subdiv ( \xi_\activeset)$ given $\pi_\subdiv ( \xi_\activeset) \in \interpset_{U,v^{(n)}}$ (interpolation constraints) and given $\pi_\subdiv ( \xi_\activeset) \in \ineqset_{\activeset}$ (inequality constraints). 

Let us define the following central characteristics of this conditional distribution:\label{page:hat:alpha}
\begin{equation} \label{eq:def:hat:alpha}
	\widehat{\alpha}_{\activeset,\subdiv, U, v^{(n)}}
	= \underset{ \substack{\alpha \in A_{\subdiv}
			\\
			\textrm{s.t. }\, Y_{\subdiv,\alpha} \, \in  \, \interpset_{U, v^{(n)}} \cap \, \ineqset_{\activeset}
	} }{\mathrm{argmin}}
	\alpha^\top
	k_{\activeset}(\subdiv,\subdiv)^{-1}
	\alpha.
\end{equation}

\label{page:map:function}
The mathematical definitions underlying \eqref{eq:def:hat:alpha} are recalled in Table \ref{table:list:of:symbols}. In words,
the quantity in \eqref{eq:def:hat:alpha} is the mode of the density of the Gaussian vector extracted from the finite-dimensional GP at the knots, conditionally on the equality and inequality constraints. The linear combination of the basis functions based on these knot values leads to the MAP function
\begin{equation} \label{eq:the:mode}
	Y_{\subdiv, \widehat{\alpha}_{\activeset,\subdiv,U,v^{(n)}} },
\end{equation}
which is used by \cite{bay2016generalization,bay2017new,LopezLopera2017FiniteGPlinear,maatouk2017gaussian} and is also called the mode of a finite-dimensional GP given interpolation and inequality constraints.

Note that the interpolation constraint $Y_{\subdiv,\alpha} \in \interpset_{U,v^{(n)}}$ can be expressed as a set of explicit linear equations for $\alpha$, see for instance \cite{LopezLopera2017FiniteGPlinear,maatouk2017gaussian}. 
When the equivalence \eqref{eq:equivalence:constraint} holds, the constraint $Y_{\subdiv,\alpha} \in \ineqset_{\activeset}$ is expressed as a set of linear inequality constraints. In this case,
the optimization problem \eqref{eq:def:hat:alpha} is a quadratic optimization problem with linear inequality constraints and efficient optimization procedures are available \cite{boyd2004convex,goldfarb1983numerically}, see also  \cite{LopezLopera2017FiniteGPlinear,maatouk2017gaussian}.

Furthermore, under Condition \ref{cond:C:convex:closed} and when $u_1,\ldots,u_n$ are two-by-two distinct, 
$\widehat{\alpha}_{\activeset,\subdiv, U, v^{(n)}}$ in  \eqref{eq:def:hat:alpha} is well-defined when the minimization set is non-empty, because this set is closed and convex, and the function to be minimized is continuous and strictly convex, and goes to infinity as $||\alpha||$ goes to infinity.

Finally, beside the MAP function $Y_{\subdiv, \widehat{\alpha}_{\activeset,\subdiv,U,v^{(n)}} }$, obtaining conditional realizations of $\pi_\subdiv ( \xi_\activeset) $ given $\pi_\subdiv ( \xi_\activeset) \in \interpset_{U,v^{(n)}}$ and $\pi_\subdiv ( \xi_\activeset) \in \ineqset_{\activeset}$ is also of high interest, for uncertainty quantification. When \eqref{eq:equivalence:constraint} holds, these conditional realizations can be approximately sampled by Monte Carlo and Markov chain Monte Carlo (MCMC) procedures \cite{LopezLopera2017FiniteGPlinear,maatouk2017gaussian} (see also Section \ref{section:numerical:experiments}).

\section{The MaxMod algorithm} \label{section:MaxMod}

\subsection{Initialization}

We consider the set of $n$ observation points and observed values to be fixed in Section \ref{section:MaxMod}. 
We let  $x_D^{(1)},\ldots,x_D^{(n)} \in \domainD$ be the $n$
two-by-two distinct observation points, and $y_1 ,\ldots,y_n \in \R $ be the corresponding observations.  Typically, $y_i = f( x_D^{(i)} )$ where $f : \domainD \to \mathbb{R}$ is the function of interest that is modeled by a GP realization. 
We write $X_D = (x_D^{(1)},\ldots,x_D^{(n)})$ and $y^{(n)} = (y_1,\ldots,y_n)$.
For $i=1,\ldots,n$ and $\activeset \subseteq \{1 , \ldots , D\}$ 
we write $x_\activeset^{(i)}$ for the vector extracted from $x_D^{(i)}$ 
by keeping the components with indices in $\activeset$. 
We write $X_{\activeset} = (x_\activeset^{(1)},\ldots,x_\activeset^{(n)})$.  \label{page:observation:points:values}

The sequential procedure we suggest is based on updating the set of active input variables and the (multidimensional) subdivision on the space of active input variables. 
It is initialized by a non-empty set $\activeset_0 \subseteq \{1,\ldots,D\}$ of active variables, 
with $| \activeset_0 | = d_0$ and by an initial subdivision $\subdiv^{(0)} \in \subdivset_{\activeset_0}$.

We also assume that $x_{\activeset_0}^{(1)},\ldots,x_{\activeset_0}^{(n)}$ are two-by-two distinct. This condition is not restrictive and typically holds when selecting a few active variables for initialization (for instance the most correlated with the outputs, empirically).
We also assume that the set 
$ \{\alpha \in A_{\subdiv^{(0)}} ; 
Y_{\subdiv^{(0)},\alpha} \in \interpset_{X_{\activeset_0},y^{(n)}} \, \cap \, \ineqset_{\activeset_0} \} $
is non-empty, which means that there exists a finite-dimensional interpolating function satisfying the inequality constraints at the initialization of the sequential procedure. This holds for general choices of a sufficient number of multidimensional knots.

Then, $\widehat{\alpha}_{\activeset_0,\subdiv^{(0)},X_{\activeset_0},y^{(n)}}$ is well-defined when Condition \ref{cond:C:convex:closed} holds, 
since the minimization set in \eqref{eq:def:hat:alpha} is non-empty, as discussed above.

\subsection{The $L^2$ difference between modes}

For a non-empty set $\activeset \subseteq \{1 , \ldots , D\}$, 
and for $\subdiv \in \subdivset_\activeset$, define \label{page:hat:alpha:Y:maxmod}
\begin{equation*}
	\widehat{\alpha}_{\activeset,\subdiv} = \widehat{\alpha}_{\activeset,\subdiv,X_{\activeset},y^{(n)}} 
	\qquad \textrm{and}
	\qquad \widehat{Y}_{\activeset,\subdiv} = Y_{\subdiv, \widehat{\alpha}_{\activeset,\subdiv}}.
\end{equation*}
By definition of $\widehat{\alpha}_{\activeset,\subdiv}$, 
the function $\widehat{Y}_{\activeset,\subdiv}$ 
satisfies both interpolation and inequality constraints,
i.e. belongs to $\interpset_{X_{\activeset},y^{(n)}} \, \cap \, \ineqset_{\activeset}$. Recall that, by definition of $\widehat{\alpha}_{\activeset,\subdiv}$, this function is the most likely constrained interpolator of the observations, given the active variables $\activeset$ and the knots in $\subdiv$.
Let us now introduce two simple operations on a given subdivision $\subdiv \in \subdiv_\activeset$.
\begin{itemize}
	\item Insertion of a new knot: for $i \in \activeset$ and for $t \in [0,1] \setminus \subdiv_i$,
	we denote by $\subdiv \addknot{i} t$ the subdivision $\subdiv' \in \subdivset_\activeset$ 
	defined by $\subdiv'_j = \subdiv_j$ for $j \in \activeset \setminus \{i\}$, 
	and by $\subdiv'_i = \subdiv_i \cup \{t\}$. \label{page:S:union:i:t}
	\item Addition of a new variable: for $i \not \in \activeset$, 
	we denote by $\subdiv \addvar{\activeset} i$ the subdivision $\subdiv' \in \subdivset_{\activeset \cup \{i \}}$ 
	defined by $\subdiv'_j = \subdiv_j$ for $j \in \activeset $, 
	and by $\subdiv'_i = S^0$, the minimal subdivision defined in Section \ref{sec:finite:dimensional:constrained:GP} that corresponds to functions that are linear with respect to variable $i$. \label{page:S:plus:i}
\end{itemize}

For a non-empty $\activeset \subset \{ 1 , \ldots , D\}$, for $S \in \subdivset_{\activeset}$ and for $i \in \{ 1 , \ldots , D\}$, let \label{page:N:S:J:i}
\[
N_{S,\activeset,i} = 
\begin{cases}
	\prod_{\substack{j \in \activeset \\ j \neq i }} \left( |S_j| -2 \right)
	& \text{if} ~ ~ i \in \activeset
	\\ 
	\prod_{\substack{j \in \activeset}}  \left( |S_j| -2 \right)
	& \text{if} ~ ~ i \not \in \activeset,
\end{cases}
\]
with the convention $\prod_{\substack{j \in \activeset \\ j \neq i }} \left( |S_j| -2 \right) = 1$ if $\activeset = \{ i \}$.

The quantity $N_{S,\activeset,i}$ is the increase of the number of basis functions if, starting from a subdivision $S$ and a set of active variables $\activeset$, a knot was inserted for the active variable $i$ or the inactive variable $i$ was made active. Recall that for $j \in \activeset$, $|S_j|-2$ is the number of one-dimensional basis functions corresponding to the variable $j$.

Now, we introduce the criterion used by MaxMod.
The idea is to measure the $L^2$ difference between the mode
of the finite-dimensional process obtained at a potential next step of the algorithm, and the current one, characterized by $\activeset$ and $\subdiv$.
At a new iteration, two choices are possible. Either a new knot $t$ is inserted for an active variable $i \in \activeset$, 
or a new active variable $i \not \in \activeset$ is added. 
In this case, the minimal subdivision is added for the dimension $i$
(corresponding to a linear function with respect to the dimension $i$).
Formally,  for $t \in [0,1]$ and $i \in \{1,\ldots,D\}$, the criterion is written: \label{page:normalized:Ldeux:difference}
\begin{equation} \label{eq:criterion:maxmod}
	I_{\activeset,\subdiv}(i,t) =
	\begin{cases}
		\frac{ 1
		}{
			N_{S,\activeset,i}
		}
		\displaystyle \int_{[0,1]^d}
		\left(
		\widehat{Y}_{\activeset, \, \subdiv \addknot{i} t}(x)
		-
		\widehat{Y}_{\activeset, \, \subdiv}(x)
		\right)^2
		d x
		&
		~ ~ \text{if} ~ ~ i \in \activeset,
		\vspace{0.1cm}\\
		\frac{ 1
		}{
			N_{S,\activeset,i}
		}
		\displaystyle \int_{[0,1]^{d+1}}
		\left(
		\widehat{Y}_{\activeset \cup \{i\}, \, \subdiv \addvar{\activeset} i }(x)
		-
		\widehat{Y}_{\activeset , \, \subdiv}(x)
		\right)^2
		d x
		&
		~ ~ \text{if} ~ ~ i \not \in \activeset.
	\end{cases}
\end{equation}
Note that in \eqref{eq:criterion:maxmod}, in the case $i \not \in \activeset$, we have made the slight abuse of notation of treating $\widehat{Y}_{\activeset , \, \subdiv}$ as a function of the $|\activeset|+1$ variables $(x_j)_{j \in \activeset \cup \{i\}}$, that does not use the variable $x_i$. We also remark that, for $i \not \in \activeset$, $I_{\activeset,\subdiv}(i,t)$ does not depend on $t$. We explain how to compute \eqref{eq:criterion:maxmod} efficiently in practice in Appendix \ref{subsection:computing:Ldeux:difference:modes}, resulting in a computational complexity that is linear in the number of multi-dimensional knots.

The criterion \eqref{eq:criterion:maxmod} also penalizes insertions of knots at active variables or additions of new active variables, that increase the number of basis functions (that is the computational complexity) significantly. More precisely, the $L^2$ difference between modes is divided by the number of additional basis functions.

\subsection{The MaxMod algorithm}

For $\Delta, \Delta' >0$, we introduce the reward 
\label{page:reward}
\begin{equation}
	R_{\activeset, \subdiv}(i, t) = 
	\begin{cases} 
		\Delta d(t, \subdiv_{ i }) & \text{ if } i \in \activeset \\
		\Delta' & \text{ otherwise }
	\end{cases}
	\label{eq:criterion:reward}
\end{equation}
that promotes the addition of new variables or the insertion of one-dimensional knots not too close to existing ones. 
Here, for $q \in \mathbb{N}$, $x \in \R^q$ and $B \subseteq \R^q$, we denote by $d(x,B) = \inf_{u \in B} ||x - u||$, the distance between $x$ and the set $B$.  \label{page:d:t:S}
The reward for adding a new active variable is $\Delta'$ and 
the reward for inserting a knot for an existing variable is $\Delta$ times the distance to the closest existing knot.

For technical reasons, we need to prevent inserting one-dimensional knots that coincide with existing ones. We thus introduce a sequence of separation distances $(b_m)_{m \in \mathbb{N}}$, with $b_m>0$, such that, at step $m$ of MaxMod, no one-dimensional knot should be inserted at distance less than $b_m$ to an existing knot. In practice, $b_m$ can be taken as small as desired, even equal to the machine precision. 
We assume that the $(b_m)_{m \in \mathbb{N}}$ are small enough such that
\begin{equation} \label{eq:nonCoverageCondition}
	\forall m \in \mathbb{N}, \qquad 2 b_m \left( \max_{i \in \activeset_0} \left \vert \subdiv^{(0)}_i  \right \vert + m \right) < 1.
\end{equation}
As will be explained below, \eqref{eq:nonCoverageCondition} simply guarantees that the separation distance does not prevent MaxMod to insert knots.

We also introduce a sequence $(a_m)_{m \in \mathbb{N}}$ of strictly positive numbers that correspond to distances to optimality. The principle is that a knot or a variable added by MaxMod at step $m$ should maximize a quality criterion and that we allow for a distance to the global maximum that is not exactly zero but is simply bounded by $a_m$.
The MaxMod sequential procedure can now be written as in Algorithm \ref{alg:iterative}.


\begin{algorithm}[t!]
	\caption{MaxMod (maximum modification of the MAP)}\label{alg:iterative}
	\begin{algorithmic}[1]
		\BStatex  \textbf{Input parameters:} $\Delta>0$, $\Delta' >0$, two sequence of strictly positive numbers $(a_m)_{m \in \mathbb{N}}$ and $(b_m)_{m \in \mathbb{N}}$, the initial set of active variables $\activeset_0 \subseteq \{1,\ldots,D\}$ and the initial subdivision $\subdiv^{(0)} \in \subdivset_{\activeset_0}$. 
		\BStatex \textbf{Sequential procedure:} For $m \in \mathbb{N}$,  $m \geq 0$, do the following.
		\BState Set $i^\star_{m+1} \in \{1,\ldots, D\}$, $t^{\star}_{m+1} \in  [0,1] $
		such that
		$d \big( t^{\star}_{m+1} , \subdiv^{(m)}_{i^\star_{m+1}} \big) \geq b_m$ 
		if $i^\star_{m+1} \in \activeset_m$, and such that
		\begin{align}
			\label{eq:max:mod}
			& I_{\activeset_m, S^{(m)}}(i^\star_{m+1} ,t^\star_{m+1}) 
			+ R_{\activeset_m, \subdiv^{(m)}}(i^\star_{m+1}, t^\star_{m+1}) + a_m
			\notag
			\\
			&			\geq 
			\underset{
				\substack{
					i \in \{1 , \ldots , D\}, \, t \in [0,1], \\
					\text{s.t. } d(t , \subdiv^{(m)}_{ i })\, \geq \, b_m \\
					\text{if $i \in \activeset_m$}
				}
			}{\sup}
			\big(
			I_{\activeset_m, S^{(m)}}(i ,t) 
			+ R_{\activeset_m, \subdiv^{(m)}}(i, t)
			\big).
		\end{align}
		\If {$i^\star_{m+1} \in \activeset_m$} $\activeset_{m+1} = \activeset_{m} $ 
		and $\subdiv^{(m+1)} = \subdiv^{(m)} \, \addknot{ i^\star_{m+1} } \, t^\star_{m+1}$, that is we add the knot $t^\star_{m+1}$ to the current subdivision $\subdiv^{(m)}$ for the variable $i^\star_{m+1} $.
		\Else~{$\activeset_{m+1} = \activeset_{m} \cup \{ i^\star_{m+1} \}$ 
			and $\subdiv^{(m+1)} = \subdiv^{(m)} \addvar{} i^\star_{m+1}$, that is we add the variable $i^\star_{m+1} $ to the current subdivision $\subdiv^{(m)}$.}
		\EndIf
	\end{algorithmic}
\end{algorithm}

MaxMod maximizes a quality criterion, over the added variable or the inserted one-dimensional knot. This quality criterion is the sum of the reward and the $L^2$ distance between the current mode function and the next one (divided by the number of additional knots). Note that if a new variable is added, there is no new one-dimensional knot to select. Indeed, the new one-dimensional subdivision of the new variable is always composed of the knots $\{-1,0,1,2\}$. If a knot is inserted to an existing variable, its location in $[0,1]$ is optimized continuously. As discussed above, we allow for an approximate maximization of the quality criterion with a gap $a_m$ to the maximum at step $m$. Again as discussed above, no knot should be inserted to an active variable at distance less than $b_m$ from an already existing knot.

Notice that Algorithm \ref{alg:iterative} is well-defined in the sense that $i^\star_{m+1}, \, t^\star_{m+1}$ 
can be chosen at each step. Indeed, first the supremum in \eqref{eq:max:mod} is over a non-empty set.
This is because \eqref{eq:nonCoverageCondition} implies that for all $m \in \mathbb{N}$, 
one can take any $i \in \activeset_0 \subseteq \activeset_m$ and find $t \in [0,1]$ s.t. $d \big(t , \subdiv^{(m)}_{i} \big) \geq b_m$.
Indeed, consider the intervals of length $2b_m$ centered at $t^{(S^{(m)}_i)}_j$, $j=1, \dots, m_{S^{(m)}_i}$.
At step $m$ of Algorithm \ref{alg:iterative}, the number of these intervals is less than 
$\max_{i \in \activeset_0} \big \vert \subdiv^{(0)}_i  \big \vert + m$, 
since the initial number of knots is less than 
$\max_{i \in \activeset_0} \big \vert \subdiv^{(0)}_i  \big \vert$ 
and the algorithm inserted at most one knot for coordinate $i$ at each previous step. 
Thus, the union of these intervals does not cover $[0, 1]$.\\ 
Secondly, the $\sup$ is finite in \eqref{eq:max:mod} as proved in the following lemma. 

\begin{lemma} \label{lem:sup:finite}
	Let Condition \ref{cond:C:convex:closed} hold. 
	Then, for each $m \in \mathbb{N}$, the $\sup$ in \eqref{eq:max:mod} is finite.
\end{lemma}

\begin{table}[t!]
	\caption{List of symbols for Sections \ref{sec:finite:dimensional:constrained:GP} and \ref{section:MaxMod}, with the page numbers where the symbols are introduced.}
 	\label{table:list:of:symbols}
	\centering
	\small
	\begin{tabular}{p{2.9cm}|p{10.3cm}|p{0.9cm}} 
		\hline 
		{\bf Symbol} & {\bf Description}  &	 {\bf Page} \\
		\hline
		$\phi_{u,v,w}$ & One-dimensional hat basis function with knots $u,v,w$ &  \pageref{page:phi:u:v:w} \\
		\hline 		
		$t^{(\subdiv)}_{(0)}, \ldots, t^{(\subdiv)}_{(m_{\subdiv}+1)}$ & Ordered knots of a one-dimensional subdivision $\subdiv$ &  \pageref{page:ordered:knots} \\
		\hline 		
		$D,d$ & Input space dimension and  number of active variables&  \pageref{page:D:d:active:set} \\
		\hline 
		$\activeset = (a_1,\ldots,a_d)$ & Set of active variables &  \pageref{page:D:d:active:set} \\
		\hline 
		$\subdiv = (\subdiv_{a_1}, \ldots, \subdiv_{a_d} )$ & (Multi-dimensional) subdivision (vector of one-dim. subdivisions) &  \pageref{page:vector:subdivisions} \\
		\hline 
		$\multi  \in \multiset $ & 
		Multi-index for a subdivision $\subdiv$ (indexation of the knots) &  \pageref{page:set:multi:indices} \\
		\hline 
		$\alpha \in A_\subdiv$ & 
		Sequence of coefficients for a subdivision $\subdiv$ (values at knots) &  \pageref{page:A:subdiv} \\
		\hline 
		$\multiknot$ & 
		($d$-dimensional) knot for a subdivision $\subdiv$ &  \pageref{page:d:dim:knot} \\
		\hline 
		$\phi_{\multi}^{(\subdiv)}$ & 
		Basis function for a subdivision $\subdiv$ (tensorized hat basis functions) &  \pageref{page:phi:multi} \\
		\hline 
		$ E_\subdiv$ & 
		Linear space of the componentwise, piecewise linear functions (equivalently multivariate splines of degree 1) with knots in $\subdiv$  &  \pageref{page:Y:Subdiv:alpha} \\
		\hline 
		$Y_{\subdiv,\alpha} $ & 
		Function in $E_\subdiv$ with values at knots $\alpha$
		&  \pageref{page:Y:Subdiv:alpha} \\
		\hline 
		$\pi_\subdiv(f)$ & 
		Function in $E_\subdiv$ interpolating $f$ at knots in $\subdiv$ &  \pageref{page:pi:subdiv:f} \\
		\hline 
		$\interpset_{U,v^{(n)}}$ & 
		Set of $d$-dimensional functions whose values at inputs $U$ are $v^{(n)}$ &  \pageref{page:I:U:vn} \\
		\hline 
		$\ineqset_D$ & 
		Set of $D$-dimensional functions satisfying the inequality constraints &  \pageref{page:C:D} \\
		\hline 
		$\ineqset_{\activeset}$ & 
		Set of $d$-dimensional functions satisfying the inequality constraints, when seen as functions of $D$ variables &  \pageref{page:C:activeset} \\
		\hline 
		$k_D, k_\activeset$ & 
		Covariance functions with input dimension $D$ and $d$ &  \pageref{page:k:D:K:J} \\
		\hline 
		$\xi_\activeset$ & 
		$d$-dimensional GP with covariance function $k_\activeset$
		&  \pageref{xi:active:set} \\
		\hline 
		$k_{\activeset}(\subdiv,\subdiv)$ & 
		Covariance matrix of the values of $\xi_\activeset$ at the knots in $\subdiv$
		&  \pageref{page:k:J:S:S} \\
		\hline 
		$	M (\ineqset_{\activeset} ) , v( \ineqset_{\activeset} )$ & 
		Matrix and vector for the inequality constraints
		&  \pageref{page:linear:constraints} \\
		\hline 
		$\widehat{\alpha}_{\activeset,\subdiv, U, v^{(n)}}$ & 
		Values at knots of $\xi_\activeset$ with highest density conditionally on interpolation and inequality constraints
		&  \pageref{page:hat:alpha} \\
		\hline 
		$Y_{\subdiv, \widehat{\alpha}_{\activeset,\subdiv,U,v^{(n)}} }$ & 
		$d$-dim. function in $E_{\subdiv}$ built from $\widehat{\alpha}_{\activeset,\subdiv, U, v^{(n)}}$ (MAP or mode)
		&  \pageref{page:map:function} \\
		\hline 
		$X_D,X_{\activeset},y^{(n)} $ & 
		Observation points and associated observed values for MaxMod
		&  \pageref{page:observation:points:values} \\
		\hline 
		$\widehat{\alpha}_{\activeset,\subdiv} ,  \widehat{Y}_{\activeset,\subdiv}$ & 
		$\widehat{\alpha}_{\activeset,\subdiv, U, v^{(n)}}$  and $Y_{\subdiv, \widehat{\alpha}_{\activeset,\subdiv,U,v^{(n)}} }$  for MaxMod
		&  \pageref{page:hat:alpha:Y:maxmod} \\
		\hline 
		$\subdiv \addknot{i} t $ & 
		Insertion of the knot $t$ to the subdivision $\subdiv$ for the variable $i$
		&  \pageref{page:S:union:i:t} \\
		\hline 
		$\subdiv \addvar{\activeset} i$ & 
		Addition of the variable $i$ to the subdivision $\subdiv$
		&  \pageref{page:S:plus:i} \\
		\hline 
		$N_{S,\activeset,i} $ & 
		Number of additional basis functions for a new knot/variable
		&  \pageref{page:N:S:J:i} \\
		\hline 
		$I_{\activeset,\subdiv}(i,t) $ & 
		Normalized $L^2$ difference between modes for a new knot/variable
		&  \pageref{page:normalized:Ldeux:difference} \\
		\hline 
		$R_{\activeset, \subdiv}(i, t) $ & 
		Reward promoting new variables and distance between knots
		&  \pageref{page:reward} \\
		\hline 
		$d(x,B)$ & 
		Distance between a point $x$ and a set $B$
		&  \pageref{page:d:t:S} \\
		\hline 
	\end{tabular}
\end{table}

\section{Convergence results} \label{section:convergence}

In Section \ref{subsection:convergence:finite:dimensional:mode}, we first provide an intermediary result that may be considered of independent interest. We show that for any sequence of multi-dimensional knots, the MAP function converges uniformly to a limit function which we define. Then, in Section \ref{subsection:convergence:maxmod} we use this intermediary result to prove the convergence of MaxMod.

\subsection{Convergence of the finite-dimensional MAP function with a general sequence of multi-dimensional knots} \label{subsection:convergence:finite:dimensional:mode}

In this section, we consider a fixed set of active variables $\activeset$ with $|\activeset| = d$, 
and a corresponding sequence of subdivisions, not necessarily obtained from MaxMod. 
Without loss of generality, we set $\activeset = \{1, \dots, d \} \subseteq \{1 , \ldots , D\}$ for $d \leq D$.
Thus we shall consider functions from $\domain$ to $\mathbb{R}$, with the fixed dimension $d$. 
This allows us to remove the dependence on $\activeset$ in the notations,
and to write more simply: 
$X = (x^{(1)},\ldots,x^{(n)}) = (x_\activeset^{(1)},\ldots,x_\activeset^{(n)})$,
$\subdivset = \subdivset_{\activeset}$, 
$\widehat{\alpha}_{\subdiv,U,v^{(n)}} = \widehat{\alpha}_{\activeset,\subdiv,U,v^{(n)}}$,
$\widehat{\alpha}_{\subdiv} = \widehat{\alpha}_{\activeset,\subdiv}$,
$\widehat{Y}_{\subdiv} = \widehat{Y}_{\activeset,\subdiv}$,
$\ineqset = \ineqset_{\activeset}$ 
and $k = k_{\activeset}$.
In particular, the sets $\ineqset$ corresponding to \eqref{eq:boundedness:D}, \eqref{eq:monotonicity:D} and \eqref{eq:convexity:D} are written:
\begin{eqnarray} 
	\ineqset &=& \{ f \in \mathcal{C}(\domain, \mathbb{R}) ; \,
	a \leq f(x) \leq  b ~ \text{for all} ~ x \in \domain \}, \label{eq:boundedness} \\
	\ineqset &=& \{ f \in \mathcal{C}(\domain, \mathbb{R}) ;  \,
	f(u) \leq f(v) ~ \text{for all} ~ u,v \in \domain, u \leq v \}  \label{eq:monotonicity} \\
	\ineqset &=& \{ f \in \mathcal{C}(\domain, \mathbb{R}) ; \, 
	\text{for all $i \in \{1,\ldots,d\}$,} ~ \text{for all $x_{\sim i} \in [0,1]^{d-1}$,} ~ \label{eq:convexity} \\
	& & \hspace{4cm} \text{the function} ~ u_i \mapsto f(u_i , x_{\sim i}) ~ \text{is convex} \}.  \nonumber
\end{eqnarray}

\subsubsection{The multiaffine extension of a multivariate function}
\label{susubsection:multiaffine:extension}


For a univariate continuous function $g$ defined on a closed subset $B$ of $[0, 1]$ containing $0$ and $1$, 
let us denote by $L_B (g)$, or simply $L_B \, g$, the \emph{affine extension} of $g$ on $[0, 1]$: 
\begin{equation} \label{eq:piecewiseLinearExtension}
	L_B \, g(t) = 
	\begin{cases}
		g(t) & \textrm{ if } t \in B \\ 
		\omega_{-}(t) g(t^-) + \omega_+(t) g(t^+) & \textrm{ if } t \notin B
	\end{cases} 
\end{equation}
where $t^- = \max \{u \in B; u \leq t \}$ is the closest left neighbor of $t$ in $B$, 
$t^+ = \min \{u \in B; u \geq t \}$ is the closest right neighbor of $t$ in $B$, 
$\omega_+(t) = \frac{t - t^-}{t^+ - t^-}$ and $\omega_-(t) = 1 - \omega_+(t)$.
Notice that, thanks to the assumptions on $B$, if $t \in [0,1] \setminus B$ then 
$t^-$ and $t^+$ are well-defined and distinct and belong to $B$.
Hence $\omega_-(t), \omega_+(t) \in (0, 1)$ are well-defined, and $L_B$ is well-defined by (\ref{eq:piecewiseLinearExtension}).

One can show that $L_B \, g$ is the unique continuous function $h$ equal to $g$ on $B$ 
such that $h$ is affine on all intervals of $[0, 1] \setminus B$ (see the proof of Proposition \ref{prop:multiaffine:extension}). This is clear when $B$ is a finite union of intervals, 
but it is still valid for more complex closed sets, such as an infinite sequence with an accumulation point.
Starting from this property, we now introduce the multiaffine extension of a continuous multivariate function, as defined below.
\begin{definition}[multiaffine function]
	Let $f$ be a function defined on a product set $G = G_1 \times \dots \times G_d$. 
	We say that $f$ is \emph{d-affine} or simply \emph{multiaffine} 
	if it is componentwise affine:
	for all $i=1, \dots, d$, for all $u = (u_1, \dots, u_d) \in G$, the function 
	$u_i \mapsto f(u)$ is affine on $G_i$.
\end{definition}

\begin{proposition}[definition of the multiaffine extension] \label{prop:multiaffine:extension}
	Consider a continuous multivariate function $f$ on $\clos = \clos_1 \times \dots \times \clos_d$, 
	where each $\clos_i$ is a closed subset of $[0, 1]$ containing $0$ and $1$.\\
	Then, there exists a unique continuous function $g$ defined on $\domain$, equal to $f$ on $\clos$,  
	such that, for all $i=1, \dots, d$, for all $t \in [0, 1]^d$, the univariate cut function
	$g(., t_{\sim i}): u_i \mapsto g(u_i, t_{\sim i})$ is affine on each interval of $[0, 1] \setminus \clos_i$.
	Furthermore, $g$ is obtained sequentially from $f$ by linearly extending univariate cuts. 
	More precisely, for $i \in \{1,\ldots,d\}$, denote $\prodF{i} = \clos_1 \times \dots \times \clos_i, \prodF{0} = \emptyset$, and
	let $L_i$ be defined from $\mathcal{F}(\prodF{i} \times [0,1]^{d-i},\mathbb{R})$ to $\mathcal{F}(\prodF{(i-1)}\times [0,1]^{d-i+1},\mathbb{R})$ by 
	\[
	L_i(h) (t) = L_i \, h(t) = (L_{\clos_i} [h(., t_{\sim i})])(t_i)
	\]
	for $h \in \mathcal{F}(\prodF{i}\times [0,1]^{d-i},\mathbb{R})$ and $t \in \prodF{(i-1)}\times [0,1]^{d-i+1}$.
	Then
	$$g = L_1 \dots L_d \, f.$$
	We call $g$ the multiaffine extension of $f$, and denote it by $\extfun (f)$ (or simply $\extfun f$). Furthermore, $\extfun (f)$ is given explicitly by, for $t \in \domain$,
	\begin{equation} \label{eq:multiExtFormula}
		\extfun (f) (t)
		= \sum_{\epsilon_1, \dots, \epsilon_d \in \{-, +\}} 
		\Bigg(\prod_{j = 1}^d \omega_{\epsilon_j}(t_j) \Bigg)
		f(t_1^{\epsilon_1}, \dots, t_d^{\epsilon_d}),
	\end{equation}
	where the $2^d$ products are non-negative and sum to one.
\end{proposition}

A two-dimensional illustration of the multiaffine extension is provided in Figure \ref{fig:multiaffine:extension}. As observed on this figure, a feature of the multiaffine extension is the extension of the definition of a continuous function $f$ on hyper-rectangles where $f$ is defined only at the $2^d$ vertices (for instance the square $[0.1,0.4]^2$ in the figure). The connection with the hat basis functions of Section \ref{sec:finite:dimensional:constrained:GP} is that the extension at one of these hyper-rectangles coincides with the expression of  $Y_{\subdiv,\alpha}$ from \eqref{eq:Y:S:alpha}, when the closest knots in $\subdiv$ to the hyper-rectangle coincide with its vertices. We refer to the last item of Remark \ref{remark:multi:affine:extension} for a more formal statement.

\begin{figure}[t!]
	\centering
	\includegraphics[width=\linewidth, trim = 0cm 3cm 0cm 2cm]{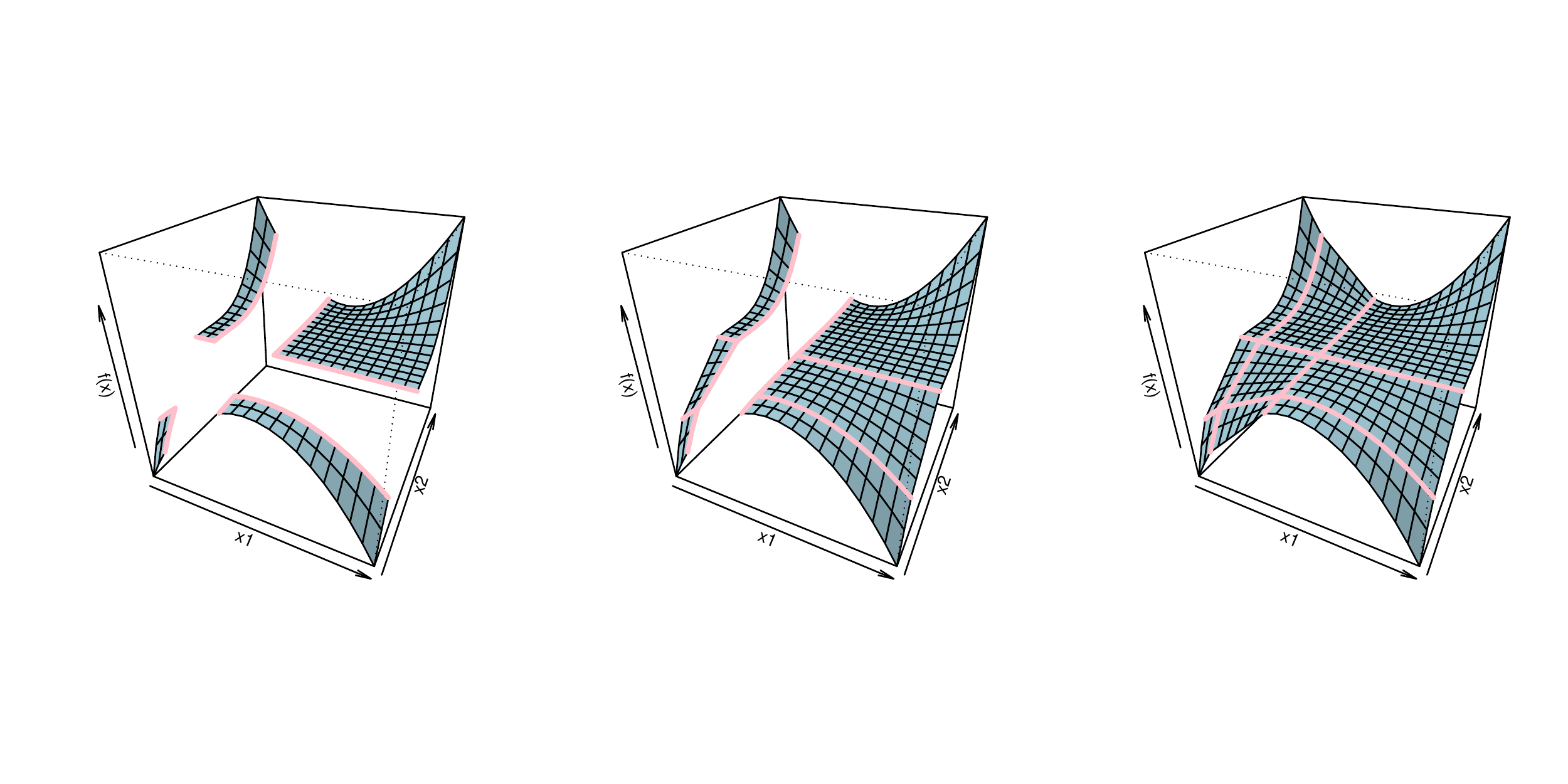}
	\caption{\small Sequential construction of the multiaffine extension. Illustration on the $2$-dimensional function  
		${f(x) = (x_1 - 0.5)^2 (x_2 - 0.5)^3}$, defined on $\clos_1 \times \clos_2$, with $\clos_1 = \clos_2 = [0, 0.1] \cup [0.4, 1]$. The three perspective plots represent, from left to right: $f$, $L_2 f$ and $\extfun f = L_1 L_2 f$.}
	\label{fig:multiaffine:extension}
\end{figure}

\clearpage
\begin{remark} \label{remark:multi:affine:extension}
	{~}
	\begin{itemize}
		\item As a direct consequence of Proposition~\ref{prop:multiaffine:extension},
		all $q$-dimensional cuts of the multiaffine extension, i.e. $(t_{i_1}, \dots, t_{i_q}) \mapsto \extfun f(t)$, are $q$-affine 
		on any hypercube in the product of complementary sets $\prod_{\ell=1}^q ([0,1] \setminus \clos_{i_\ell}).$
		\item In Proposition~\ref{prop:multiaffine:extension}, the multiaffine extension is obtained by composing 1-dimensional cuts in a specific order. By uniqueness, permuting the order will give the same result.
		\item In one dimension, extending linearly a function by using the closest neighbors as in \eqref{eq:piecewiseLinearExtension}
		can be viewed as computing the posterior mean of the Brownian motion \cite[preface, page xiv]{berlinet2011reproducing}.
		In general, one could define the multiaffine extension of Proposition \ref{prop:multiaffine:extension} with the Brownian sheet.
		However, this would require to condition the Brownian sheet on a continuous set and would add technicalities (see, e.g., \cite{bay2010continuousBoundaries} for more details).
		This is why we have chosen to introduce the multiaffine extension with basic tools.
		\item Consider $\subdiv \in \subdivset$ and the set $ \clos_{\subdiv} =  \prod_{j=1}^d ( \subdiv_j \cap [0,1] )$, where each $ \subdiv_j \cap [0,1]$ is a closed subset of $[0,1]$ containing $0$ and $1$.
		Then, by Proposition~\ref{prop:extension:properties}, 
		for all $f \in \mathcal{C}( \domain , \mathbb{R} )$,
		\begin{equation} \label{eq:pi:S:f:equal:multiaffine:extension}
			\pi_{\subdiv} (f)
			=
			P_{ \clos_{\subdiv} \to [0,1]^d} (f_{| \clos_{\subdiv}}).
		\end{equation}
		Hence, we can also interpret the multiaffine extension as a generalization of the projection $\pi_{\subdiv}$ to an infinite number of knots.
	\end{itemize}
\end{remark}
Now, we gather below some properties of the multiaffine extension.

\begin{proposition} \label{prop:extension:properties}
	{~}
	\begin{enumerate}
		\item The map $f \mapsto P_{\clos \to \domain}(f)$ is affine and $1$-Lipschitz 
		from $\mathcal{C}(\clos, \R)$
		to $\mathcal{C}([0, 1]^d, \R)$,
		equipped with the $L^\infty$ norm. In particular, it preserves uniform convergence.
		\item If $\subdiv$ is a $d$-dimensional subdivision such that 
		$ S \cap [0, 1]^d \subseteq \clos$,
		then any piecewise multilinear function constructed from $\subdiv$
		coincides with the multiaffine extension of its restriction to $\clos$:
		$$ \forall f \in E_\subdiv: \qquad \extfun \left( f_{\vert \clos} \right) = f .$$
	\end{enumerate}
\end{proposition}

\begin{corollary} \label{cor:vertice:to:hypercube}
	For $x \in \domain$, let $f,g$ be $d$-affine on the hypercube $\Delta = \prod_{j=1}^d [x_j^- , x_j^+]$, and coincide on its $2^d$ vertices. Then they are equal on $\Delta$. 
\end{corollary}

\subsubsection{Technical conditions}

We  let $m_0 \in \mathbb{N}$ and consider a sequence of subdivisions $(\subdiv^{(m)})_{m \geq m_0}$ with $\subdiv^{(m)} \in \subdivset $ for $m \geq m_0$ (not necessarily the sequence obtained from MaxMod).
Let the subdivisions be nested (i.e. $S_i^{(m)} \subseteq S_i^{(m+1)}$ for $i \in \{1 , \ldots d\}$).
We will show the convergence of the sequence of mode functions obtained from the sequence of subdivisions $(\subdiv^{(m)})_{m \geq m_0}$. In that view, we now introduce a list of technical conditions.

To the set $\ineqset$, that is interpreted as the set of functions satisfying inequality constraints, we associate a set 
$\strictineqset \subseteq \mathcal{C}$, that we can choose and that is interpreted as a set of functions satisfying corresponding strict inequality constraints. Let us explicitly show how $\strictineqset$ is chosen when $\ineqset$ is given by one of \eqref{eq:boundedness}, \eqref{eq:monotonicity} or \eqref{eq:convexity}.

First, consider that $\ineqset$ is given by  \eqref{eq:boundedness}, 
with $- \infty <a < b < + \infty$
(the cases where either $a = - \infty$ or $b=+ \infty$ are similar). 
Then we choose $\strictineqset$ as the set of continuous functions that are strictly between $a$ and $b$.
Second, consider that $\ineqset$ is given by  \eqref{eq:monotonicity}. 
Let us write $u <v$ when $u \leq v$ and $u \neq v$.
We say that a function $f$ from $\domain$ to $\R$ is
strictly increasing if $f(u) < f(v)$ for all
$u,v \in \domain$, $u < v$.
Then we chose $\strictineqset$ as the set of continuous strictly increasing functions. 
Finally, consider that $\ineqset$ is given by  \eqref{eq:convexity}. Then we choose $\strictineqset$ as
\begin{eqnarray*}
	\{ f \in \mathcal{C}(\domain , \mathbb{R}),
	&& ~ \text{for all $i \in \{1,\ldots,d\}$,} ~ \text{for all $x_{\sim i} \in [0,1]^{d-1}$},  \\
	&& \text{the function} ~ u_i \mapsto f(u_i , x_{\sim i}) ~ \text{is strictly convex} \}.
\end{eqnarray*}
For other inequality sets $\ineqset$, we emphasize that we are free to choose the set $\strictineqset \subseteq \ineqset$ for which the technical conditions given below (Conditions \ref{cond:non:empty} and \ref{cond:interior:non:empty}) hold. 

\begin{condition}[initial knots flexibility, 1] \label{cond:non:empty}
	The set 
	$\{ \alpha \in A_{\subdiv^{(m_0)}};
	Y_{\subdiv^{(m_0)},\alpha} \in \interpset_{X,y^{(n)}}
	\cap \strictineqset
	\}$ is non-empty.
\end{condition}

The previous condition means that one can construct, from the initial subdivision, a finite-dimensional function that satisfies the interpolation and strict inequality constraints. 
With the choice of $\strictineqset$ for boundedness, Condition \ref{cond:non:empty} implies that $a < y_1 <b, \ldots ,a < y_n <b$ and that there are 
enough
knots in the initial subdivision to generate an interpolating function that is strictly between $a$ and $b$.

With the choice of $\strictineqset$ for monotonicity, Condition \ref{cond:non:empty} implies $y_i < y_j$ for $x_i < x_j$ ($i,j=1,\ldots,n$) and that there are sufficiently many knots in the initial subdivision to generate a strictly increasing interpolating function.

Consider finally componentwise convexity. For $U = (u_1,\ldots,u_r) \in (\domain)^r$ and $v^{(r)} = (v_1,\ldots,v_r) \in \R^r$, let us say that $U$ and $v^{(r)}$ are compatible with strict convexity if there exists a function $g$ in $\strictineqset$ such that $g(u_i) = v_i$, $i=1,\ldots,r$. 
Then, Condition \ref{cond:non:empty} implies that $X$ and $y^{(n)}$ are compatible with strict convexity and that there are sufficiently many knots in the initial subdivision to generate a componentwise strictly convex interpolating function.

\begin{condition}[initial knots flexibility, 2] \label{cond:spans:Rn}
	$\big\{
	\left(
	Y_{\subdiv^{(m_0)} , \alpha} (x^{(i)})
	\right)_{i=1,\ldots,n}
	;\linebreak[1]
	\alpha \in \subdiv^{(m_0)} 
	\big\}
	=
	\mathbb{R}^n.$
\end{condition}
This condition means that for any possible $n$-dimensional observation vector (not only $y^{(n)}$), there exists a combination of basis functions based only on the initial knots that interpolates this observation vector. This condition requires that there are enough initial knots (at least $n$) and that these knots are located adequately. This condition is mild, since the number of knots increases to infinity. For instance, in dimension $d=1$, it is sufficient that there is one initial knot in between each pair of observation points.

For a subset $B$ in a metric space with distance $\mathrm{dist}$, its interior is written $\mathrm{int}_{\mathrm{dist}}(B)$ and its closure is written $\overline{B}$. 
In particular, for a finite dimensional subspace $A$, and for $ \mathcal{B} \subseteq  \mathcal{C}( A,\mathbb{R} ) $, we shall consider $\mathrm{int}_{||.||_{\infty}}( \mathcal{B} ) $ relatively to the $L^{\infty}$ norm.
We let $\Hilb$ be the RKHS (see for instance \cite{berlinet2011reproducing}) of $k$ and we write $||.||_{\Hilb}$ for the RKHS norm on $\Hilb$. 
We shall then consider $\mathrm{int}_{||.||_{\Hilb}}(\Hilb \cap \ineqset)$, where the interior is defined w.r.t. the RKHS norm of $k$ on $\Hilb$.

\begin{condition}[RKHS interior] \label{cond:interior:non:empty}
	For all $h \in \strictineqset$, for all $r \in \mathbb{N}$, $U = (u_1,\ldots,u_r) \in (\domain)^r$, 
	with $u_1,\ldots,u_r$ two-by-two distinct, 
	letting $v^{(r)} = (h(u_1), \ldots, h(u_r)) \in \R^r$, 
	the set $\mathrm{int}_{||.||_{\Hilb}}(\Hilb \cap \ineqset) \cap \interpset_{U,v^{(r)}}$ is non-empty.
\end{condition}

With the choice of $\strictineqset$ for boundedness, Condition \ref{cond:interior:non:empty} holds when
$\mathrm{int}_{||.||_{\Hilb}}(\Hilb \cap \ineqset) \cap \interpset_{U,v^{(r)}}$ is non-empty
for all $U = (u_1,\ldots,u_r) \in (\domain)^r$ and $v^{(r)} = (v_1,\ldots,v_r) \in \R^r$, with $u_1,\ldots,u_r$ two-by-two distinct and $a < v_1 < b, \ldots,a < v_n < b$. 
We also have the following lemma.

\begin{lemma} \label{lem:non-empty:contains}
	With $\ineqset$ given by \eqref{eq:boundedness} with $- \infty <a < b < + \infty$ and when $\strictineqset$ is the set of continuous functions that are strictly between $a$ and $b$, we have $\Hilb \cap \strictineqset \subseteq \mathrm{int}_{||.||_{\Hilb}} (\Hilb \cap \ineqset) $.
\end{lemma}

From Lemma \ref{lem:non-empty:contains},  the set $\mathrm{int}_{||.||_{\Hilb}}(\Hilb \cap \ineqset) \cap \interpset_{U,v^{(r)}}$ is non-empty if there exists $h \in \Hilb$ that is strictly between $a$ and $b$ and interpolates $(u_1,v_1),\ldots,(u_v,v_n)$. This can be interpreted as requiring $\Hilb$ to be rich enough, and is not restrictive (it is also required in \cite{bay2016generalization,bay2017new}).

With the choice of $\strictineqset$ for monotonicity,
Condition \ref{cond:interior:non:empty} holds when
the set $\mathrm{int}_{||.||_{\Hilb}}(\Hilb \cap \ineqset) \cap \interpset_{U,v^{(r)}}$ is non-empty
for all $U = (u_1,\ldots,u_r) \in (\domain)^r$ and $v^{(r)} = (v_1,\ldots,v_r) \in \R^r$, with $u_1,\ldots,u_n$ two-by-two distinct, with $v_i < v_j$ for $u_i < u_j$ ($i,j=1,\ldots,n$).
Again, this is not restrictive, it is also required in \cite{bay2016generalization,bay2017new} and can be interpreted as $\Hilb$ being rich enough.

With the choice of $\strictineqset$ for componentwise convexity, Condition \ref{cond:interior:non:empty} holds when
$\mathrm{int}_{||.||_{\Hilb}}(\Hilb  \cap \ineqset)\linebreak[1]  \cap \interpset_{U,v^{(r)}}$ is non-empty
for all $U = (u_1,\ldots,u_r) \in (\domain)^r$ and $v^{(r)} = (v_1,\ldots,v_r) \in \R^r$ that are compatible with strict convexity. Again this is not restrictive and is also required in \cite{bay2016generalization,bay2017new}.

Next, we introduce two conditions related to the stability of the constraint set.

\begin{condition}[constraint set stability by projection] \label{cond:piSC:subset:C}
	For all subdivision $\subdiv \in \subdivset, \pi_{\subdiv} (\ineqset) \subseteq \ineqset $.
\end{condition}
\begin{condition}[constraint set stability by multiaffine extension] \label{cond:PfC:subset:C}
	For $f \in \ineqset$, with $f_{|\clos}$ the restriction of $f$ to $\clos$, we have $P_{\clos \to \domain} (f_{|F}) \in \ineqset$.  
\end{condition}

By Lemma~\ref{lem:cond:PfC:subset:C} below, these two conditions hold in the cases of boundedness, monotonicity and componentwise convexity.
About Condition \ref{cond:piSC:subset:C}, 
this is a consequence of Remark \ref{remark:multi:affine:extension} and Lemma~\ref{lem:cond:PfC:subset:C}, applied to the set $\clos_S = \subdiv \cap \domain$, that is closed, of product form, and contains $\{0,1\}^d$. Notice that the fact that Condition \ref{cond:piSC:subset:C} holds for boundedness and monotonicity constraints  has been already proved in \cite{maatouk2017gaussian}.
\begin{lemma} \label{lem:cond:PfC:subset:C}
When $\ineqset$ is given by one of \eqref{eq:boundedness}, \eqref{eq:monotonicity} or \eqref{eq:convexity}, Condition \ref{cond:PfC:subset:C} holds.
\end{lemma}

\subsubsection{Convergence of the finite-dimensional MAP function}

For a sequence of sets $(B_m)_{m \geq m_0}$, with $B_m \subseteq [0,1]$ for $m \geq m_0$, 
we say that $B_m$ is dense in $[0,1]$ as $m \to \infty$ if for every $x \in [0,1]$, 
we have $d(x,B_m) \to 0$ as $m \to \infty$. 
We first recall the convergence theorem (Theorem 3) in \cite{bay2017new} that shows that, when for $i=1,\ldots,d$ the knots in $S_i^{(m)}$ are dense in $[0,1]$, then the function $\widehat{Y}_{\subdiv^{(m)}}$ converges as $m \to \infty$ to a limit function $Y_{\text{opt}}$.

\begin{theorem}[Kimeldorf-Wahba correspondence under constraints, \cite{bay2017new}] \label{theorem:convergence:spline}
Consider a sequence of nested subdivisions $(\subdiv^{(m)})_{m \geq m_0}$ with $\subdiv^{(m)} \in \subdivset$ for $m \geq m_0$ (i.e. $S_i^{(m)} \subseteq S_i^{(m+1)}$ for $i \in \{1 , \ldots d\}$), such that for $i \in \{1 , \ldots d\}$, $S_i^{(m)}$ is dense in $[0, 1]$. Assume that Condition \ref{cond:C:convex:closed} holds.
Assume that the two following conditions hold.
\begin{itemize}
	\item[($H_1$)] $\mathrm{int}_{||.||_{\Hilb}}(\Hilb \, \cap \, \ineqset) \, \cap \, \interpset_{X,y^{(n)}}$ is non-empty,
	\item[($H_2$)] $\forall m \geq m_0, \, \pi_{S^{(m)}} (\ineqset) \subseteq \ineqset $.
\end{itemize}
Then, as $m \to \infty$, the function $\widehat{Y}_{\subdiv^{(m)}}$ converges uniformly on $\domain$
to $Y_{\text{opt}}$, with:
$$ 
Y_{\text{opt}} = \underset{f \, \in \, 
	\Hilb \, \cap \, \ineqset \, \cap \, \interpset_{X,y^{(n)}}}{\mathrm{argmin}}
||f||_{\Hilb}.
$$
\end{theorem}

In fact, in the above  theorem, ($H_1$) and  ($H_2$) hold in our framework.
Indeed, ($H_1$) is a consequence of Conditions \ref{cond:non:empty} and \ref{cond:interior:non:empty} and ($H_2$) is a consequence of Condition \ref{cond:piSC:subset:C}.

The interpretation is that $\widehat{Y}_{\subdiv^{(m)}}$ is a finite-dimensional GP mode, obtained by solving an optimization problem in finite dimension (see Section \ref{sec:finite:dimensional:constrained:GP}), while $Y_{\text{opt}}$ is the optimal constrained interpolator in $\Hilb$, obtained by solving an optimization problem in infinite dimension, see \cite{bay2017new}.

We now give an extension to Theorem \ref{theorem:convergence:spline} to the case where, for $i \in \{ 1 , \ldots , d\}$, the knots in $S_i^{(m)}$ are not necessarily dense in $[0, 1]$. This extension is the main result of this section. 
We define $\clos = \clos_1 \times \dots \times \clos_d$, 
where $\clos_i$ is the closure in $[0, 1]$ of all the knots at coordinate $i$:
\begin{equation*}
\clos_i = [0,1] \cap \overline{\bigcup_{m \geq m_0} \subdiv^{(m)}_i}.
\end{equation*}

We further denote by $\ineqset_\clos$ and $\interpset_{\clos,X,y^{(n)}}$
the pre-image sets of $\ineqset$ and $\interpset_{X, y^{(n)}}$ by the multiaffine extension $\extfun$
introduced in Section \ref{susubsection:multiaffine:extension}:
\begin{eqnarray*}
\ineqset_\clos &=& \{f : \clos \to \R, \text{ s.t. } \extfun f \in \ineqset \},  \\
\interpset_{\clos,X,y^{(n)}} &=& \{ f : \clos \to \R, \text{ s.t. } \extfun f \in \interpset_{X, y^{(n)}} \}.
\end{eqnarray*}
We also let $k_\clos$ be the restriction of the kernel $k$ on $\clos \times \clos$ 
and $\Hilb_\clos$ be the corresponding RKHS of functions from $\clos \to \R$ (see e.g. \cite{berlinet2011reproducing}).
We have $\Hilb_\clos = \{f : \clos \to \R, \, \exists h \in \Hilb \text{ s.t. } h_{\vert \clos} = f\}$,
and the RKHS norm in $\Hilb_\clos$ is $\Vert f \Vert_{\Hilb_\clos} = \underset{h_{\vert \clos} = f}{\text{inf}} \Vert h \Vert_\Hilb$.\\

Then we state the extension of Theorem~\ref{theorem:convergence:spline} to non-dense sequences.
The result is intuitive: on the closure set $\clos$,
the finite-dimensional GP mode converges uniformly to the constrained interpolator in $\Hilb_\clos$, with the equality and inequality constraints given by $\interpset_{\clos,X, y^{(n)}}$ and $\ineqset_{\clos}$. On the complement of $\clos$, these two functions are piecewise multilinear (the interpolator in $\Hilb_\clos$ being extended with the multiaffine extension), see Proposition \ref{prop:multiaffine:extension}.
The multiaffine extension enables to express the convergence, both on $\clos$ and its complement, simply.

\begin{theorem} \label{theorem:extension:convergence:spline}
Consider a sequence of nested subdivisions $\subdiv^{(m)}$, as well as $\clos$, $\Hilb_\clos$, $\ineqset_\clos$, $\interpset_{\clos,X,y^{(n)}}$, as defined in this section. 
Assume that Conditions \ref{cond:C:convex:closed} to \ref{cond:PfC:subset:C} hold.
Then, as $m \to \infty$, the function $\widehat{Y}_{\subdiv^{(m)}}$ converges uniformly on $\domain$
to $\extfun \left( Y_{\clos,\text{opt}} \right)$, with:
$$ 
Y_{\clos,\text{opt}} = \underset{f \, \in \, 
	\Hilb_\clos \, \cap \, \ineqset_\clos \, \cap \, \interpset_{\clos,X,y^{(n)}}}{\mathrm{argmin}}
||f||_{\Hilb_\clos}.
$$
\end{theorem}

Note that the two functions $Y_{\text{opt}}$ in Theorem \ref{theorem:convergence:spline}
and $Y_{\clos,\text{opt}}$ in Theorem \ref{theorem:extension:convergence:spline} are equal when $\clos = \domain$.
When $\clos \neq \domain$, these two functions need not coincide, even on $\clos$.

\subsection{Convergence of the MaxMod algorithm} \label{subsection:convergence:maxmod}

We can now apply Theorem \ref{theorem:extension:convergence:spline} to prove the convergence of MaxMod in Theorem \ref{theorem:convergence} below. 
The technical conditions for Theorem \ref{theorem:convergence} are adaptations of those of 
Theorem \ref{theorem:extension:convergence:spline} to the setting of MaxMod. They are stated and discussed in Appendix \ref{appendix:section:adaptation:condition}. There, we also show that these conditions hold in the cases of boundedness, monotonicity and componentwise convexity.

Theorem \ref{theorem:convergence} shows the consistency of MaxMod, which will asymptotically select all the variables and allocate a dense sequence of knots to each variable. As a consequence, the mode function obtained from MaxMod converges to the infinite-dimensional optimal function $Y_{\text{opt}}$ defined in Theorem \ref{theorem:convergence:spline}. In Theorem \ref{theorem:convergence}, we let $\Hilb_{D}$ be the RKHS of $k_D$.

\begin{theorem} \label{theorem:convergence}
Let $(\activeset_m)$ and $(\subdiv^{(m)})$ be the sequence of sets of active variables and of subdivisions obtained from Algorithm~\ref{alg:iterative}.
Assume that $a_m \to 0$ and that Conditions \ref{cond:C:convex:closed} and \ref{cond:non:empty:D}
to
\ref{cond:PfC:subset:C:D} (see Appendix \ref{appendix:section:adaptation:condition}) hold.
Then, for $m$ large enough, $\activeset_m = \{1 , \ldots , D\}$. Furthermore, for $j=1,\ldots,D$, the set $\subdiv^{(m)}_j$ (which becomes well-defined for $m$ large enough) is dense in $[0,1]$ as $m \to \infty$. Consequently, as $m \to \infty$, the mode $\widehat{Y}_{\subdiv^{(m)}}$ converges uniformly on $\domainD$ to the function $Y_{\text{opt}}$ defined by
$$ 
Y_{\text{opt}} = \underset{f \, \in \, 
	\Hilb_D \, \cap \, \ineqset_D \, \cap \, \interpset_{X_D,y^{(n)}}}{\mathrm{argmin}}
||f||_{\Hilb_D}.
$$
\end{theorem}

Theorem \ref{theorem:convergence} can be interpreted as stating that, for a fixed dataset, as the computational budget (quantified here by the number of multi-dimensional knots) goes to infinity, the finite-dimensional mode converges to $Y_{\text{opt}}$, which is optimal for the dataset, but requires, so to speak, an infinite computational budget. Remark that MaxMod is sequential and of the greedy type. Hence,
as discussed in Section \ref{section:intro}, it is important to guarantee that its one-step-ahead allocation of the knots does not prevent it to yield mode functions that are converging to the global optimum $Y_{\text{opt}}$.

\begin{remark}
From the proof of Theorem \ref{theorem:convergence}, one can see that the convergence of MaxMod still holds if the $L^2$ distance
in \eqref{eq:criterion:maxmod} is replaced by any discrepancy criterion $\Delta(
f_1,
f_2)$, that goes to zero when $|| f_1 - f_2 ||_{\infty} \to 0$. 
\end{remark}

\section{Numerical experiments} \label{section:numerical:experiments}

In this section, we aim at testing the performance of the constrained GP when the knots and active dimensions are sequentially added using MaxMod.  In practice, as shown in \cite{LopezLopera2019lineqGPNoise}, incorporating noise in the constrained GP model leads to significant computational improvements due to the ``relaxation'' of the interpolation conditions. The noise, parametrized by a variance $\tau^2 \geq 0$, leads to a new definition of the mode in \eqref{eq:the:mode}, that can be found in \cite{LopezLopera2019lineqGPNoise,maatouk:hal-01533356}, and that we call the noisy mode. This noisy mode can be computed even when the number of multi-dimensional knots $m = m_{\subdiv_{a_1}} \times \dots \times m_{\subdiv_{a_d}}$ is smaller than the number of observations $n$. MaxMod can be carried out exactly as in Section \ref{section:MaxMod} when working with the noisy mode. We will denote as $\widehat{Y}_{\text{MaxMod}}$ the noisy mode obtained from MaxMod.

Here, we shall work with the noisy mode, which allows us to always initialize MaxMod with only one active dimension, i.e. $|\activeset| = 1$, and to add new active ones according to Algorithm \ref{alg:iterative}. We take the first active dimension as the one resulting in an initial mode $\widehat{Y}_{\text{MaxMod},0}$ that differs the most from zero. Then, for each addition of a new dimension, an initial set of two knots is allocated at the boundaries of $[0,1]$.
\begin{algorithm}[t!]
\caption{Practical implementation of MaxMod using the noisy mode}\label{alg:iterative2}
\begin{algorithmic}[1]
	\BStatex  \textbf{Input parameters:} $\Delta>0$, $\Delta' >0$, $D$.
	\BState Initialize the algorithm by selecting as initial input dimension the variable that gives the largest norm of the initial noisy mode $\widehat{Y}_{\text{MaxMod},0}$.
	\BStatex \textbf{Sequential procedure:} For $m \in \mathbb{N}$,  $m \geq 0$, do the following.
	\For{$k = 1, \ldots, D$} 
	\If {the variable $k$ is already active} compute the optimal position of the new knot $t_k \in [0,1]$ according to the MaxMod criterion described in \eqref{eq:max:mod} (in the supremum), and denote the resulting noisy mode as $\widehat{Y}_{\text{MaxMod},m+1}^{(k)}$.
	\Else~{add two knots at the boundaries of the selected new active dimension, i.e. $(t_{k,1}, t_{k,2}) = (0,1)$, and denote the resulting noisy mode as $\widehat{Y}_{\text{MaxMod},m+1}^{(k)}$.}
	\EndIf
	\EndFor
	\BState Choose the optimal decision $k^{\ast} \in \{1, \ldots, D\}$ that maximizes the MaxMod criterion:
	\begin{equation*}
		k^\ast \in \underset{k \in \{1, \ldots, D\}}{\mathrm{argmax}} \mathrm{MaxMod}_{\Delta, \Delta'}(\widehat{Y}_{\text{MaxMod}, m}, \widehat{Y}_{\text{MaxMod},m+1}^{(k)}). 
	\end{equation*}
	\BState Update knots and active variables, set new noisy mode to $\widehat{Y}_{\text{MaxMod},m+1} = \widehat{Y}_{\text{MaxMod},m+1}^{(k^\ast)}$. 
\end{algorithmic}
\end{algorithm}

Algorithm \ref{alg:iterative2} summarizes the practical implementation of MaxMod that is used within this section.
We fix $\Delta = \Delta' = 1\times 10^{-9}$, see \eqref{eq:criterion:reward}, which means that there is a negligible reward for either inserting a new knot in an already active dimension or adding a new dimension. For other applications, one may be particularly interested in adding knots while preserving a tractable input dimension, in which case $\Delta$ must be larger than $\Delta'$.
Furthermore, we consider a squared exponential kernel for the covariance function $k_{\activeset}$ \cite{Genton2001Kernels,Rasmussen2005GP}. For each step of MaxMod, the covariance parameters of $k_{\activeset}$ and the noise variance parameter $\tau^2$ are estimated via maximum likelihood \cite{Rasmussen2005GP}.
The open source code of MaxMod is available in the R package lineqGPR (0.2.0) \cite{lineqgpr}.\footnote{The source code can be found at: \url{https://github.com/anfelopera/lineqGPR}. The analytic examples proposed here (Sections \ref{section:numerical:experiments:subsec:2Dtoy} and \ref{section:numerical:experiments:subsec:dimension:reduction}) can be reproduced using the Jupyter notebooks available in the GitHub repository (\url{https://github.com/anfelopera/lineqGPR/tree/master/notebooks}).
The numerical experiments were executed on an 11th Gen Intel(R) Core(TM) i5-1145G7 CPU@2.60GHz, 16 Gb RAM.
Computations from step 2 to 4 of Algorithm \ref{alg:iterative2} were parallelized using two independent cores.}

\subsection{2D illustration under monotonicity constraints}
\label{section:numerical:experiments:subsec:2Dtoy}
For illustration, we consider the 2D monotonic function $f(x) = \frac{1}{2} x_1 + \arctan(10x_2)$ on $[0,1]^2$. We evaluate $f$ at a maximin Latin hypercube design (LHD, \cite{Dupuy2015DiceDesign}) with $n = 40$ points $x_D^{(1)},\ldots,x_D^{(n)}$. The observations, with the notation of Section \ref{section:MaxMod}, are thus $y_1 = f(x_D^{(1)}),\ldots,y_n=f(x_D^{(n)})$.
For $\widehat{Y}_{\text{MaxMod}}$, we account for monotonicity constraints everywhere. 
As a stopping rule, we check that the criterion in \eqref{eq:criterion:maxmod}, plus the reward in \eqref{eq:criterion:reward}, is smaller than a tolerance fixed to $1\times 10^{-5}$, for all possible new knot or variable.

\subsubsection{Assessment of MaxMod in terms of bending energy}
From Figure \ref{fig:toy2D}, we can observe that MaxMod starts by adding the second dimension rather than the first one since $f$ is more variable across $x_2$. Note also that, before activating the first dimension, the algorithm refines the second one by placing a third knot around $x_2 = 0.29$. Then, after the second iteration, although the first dimension has been activated, the algorithm prefers enriching the quality of $\widehat{Y}_{\text{MaxMod}}$ across $x_2$ while staying linear across $x_1$. The convergence of the algorithm is obtained after four iterations, resulting in a total of $m = 10$ knots: 2 and 5 one-dimensional knots allocated across the first and the second dimension, respectively. We note that the final estimated noise variance parameter is negligible, $\widehat{\tau}_{\text{MaxMod}}^2 = 8.38 \times 10^{-5}$ (equivalent to $\approx 0.07\%$ of the variance of the observations), resulting in a GP model that almost interpolates all the observations (see Figure \ref{fig:toy2D}).
\begin{figure}[t!]
\centering
\begin{minipage}{0.35\columnwidth}
	\centering
	\vskip1ex
	\includegraphics[width=\columnwidth]{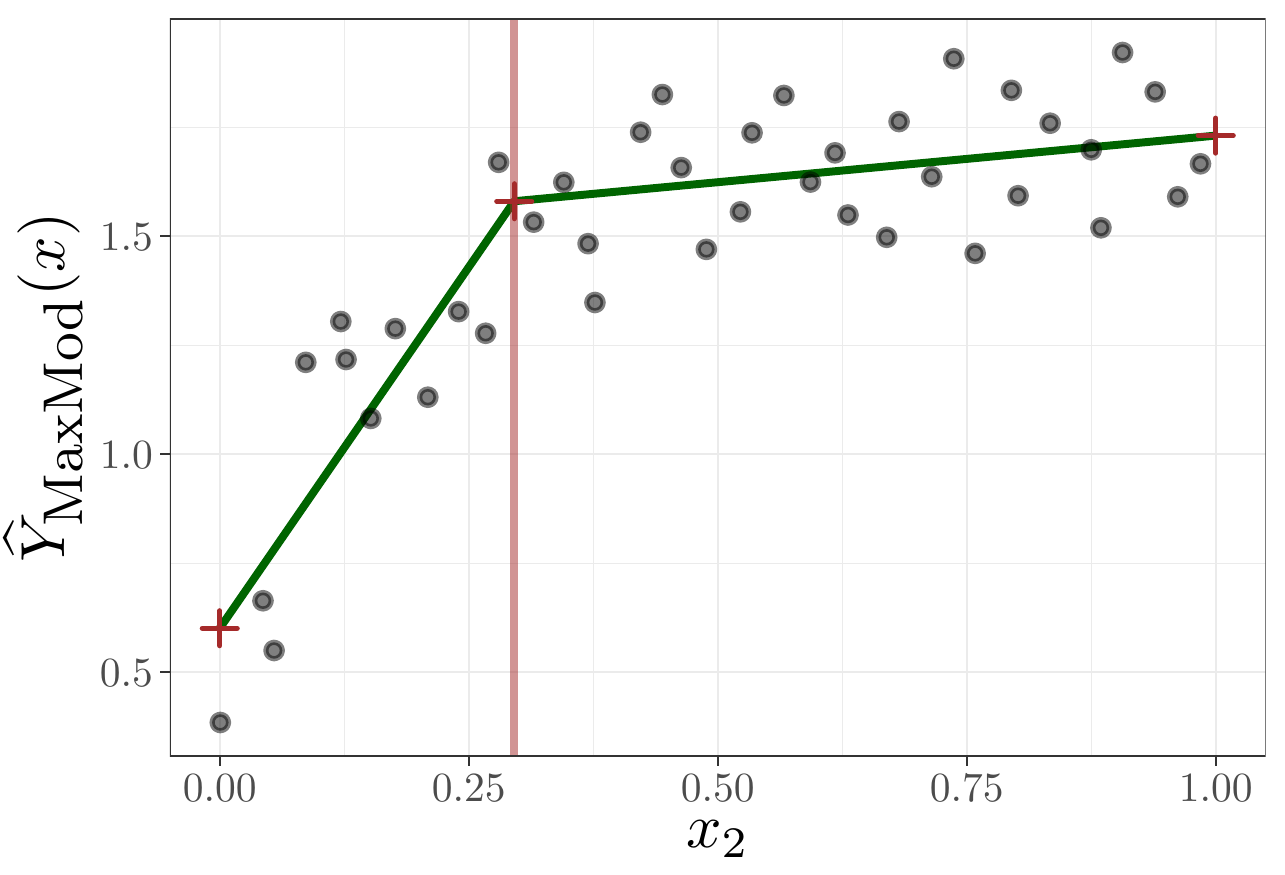}
\end{minipage}
\foreach \i\j in {4/3,6/5} {
	\begin{minipage}{0.27\columnwidth}
		\centering
		\includegraphics[width=\columnwidth]{demToyLineqGPMaxMode2Exp1SeqGrid2DMAPfig\i}
	\end{minipage}		
}
\begin{minipage}{0.35\columnwidth}
	\centering
	\hskip5ex iteration 1
\end{minipage}	
\foreach \i\j in {4/2,6/4} {
	\begin{minipage}{0.27\columnwidth}	
		\centering
		iteration \j
	\end{minipage}
}
\caption{Evolution of MaxMod using $f(x) = \frac{1}{2} x_1 + \arctan(10x_2)$ as target function. The mode $\widehat{Y}_{\text{MaxMod}}$ accounts for monotonicity constraints everywhere. The panels shows: the observations (black dots), the mode (1D: green solid line, 2D: solid surface) and the knots (red crosses). The set of added knots are highlighted by a vertical red line in 1D and a vertical red plane in 2D.}
\label{fig:toy2D}
\end{figure}

We now compare the mode $\widehat{Y}_{\text{MaxMod}}$ to modes resulting from equispaced designs of the knots (see \cite{LopezLopera2019lineqGPNoise,LopezLopera2017FiniteGPlinear,maatouk2017gaussian}). We consider either square or rectangular designs of the knots. This leads to three modes denoted as:
\begin{itemize}
\item $\widehat{Y}_{\text{MaxMod}}$: the mode resulting from MaxMod.
\item $\widehat{Y}_{\text{MaxMod, rect}}$: the mode resulting from equispaced one-dimensional knots but where the number of knots per dimension is the same as for $\widehat{Y}_{\text{MaxMod}}$.
\item $\widehat{Y}_{\text{square}}$: the mode resulting from equispaced one-dimensional knots with the same number of knots per dimension. 
\end{itemize}
We assess the quality of $\widehat{Y}_{\text{square}}, \widehat{Y}_{\text{MaxMod,rect}}$ and $\widehat{Y}_{\text{MaxMod}}$ in terms of the (normalized) bending energy \cite{DeBoor2001GuideSplines}: $E_n(f, \widehat{Y}) = \int_{[0,1]^{D}} (f(x) - \widehat{Y}(x) )^2dx/\int_{[0,1]^{D}} f^2(x) dx$. By comparing the $E_n$ values using $\widehat{Y}_{\text{MaxMod}}$ to the ones provided by $\widehat{Y}_{\text{square}}$ or $\widehat{Y}_{\text{MaxMod,rect}}$, we aim at showing that MaxMod not only adds active dimensions strategically but also places knots in regions leading to smaller errors.

Figure \ref{fig:toy2Derrors} shows the performance of $\widehat{Y}_{\text{square}}, \widehat{Y}_{\text{MaxMod,rect}}$ and $\widehat{Y}_{\text{MaxMod}}$ in terms of $E_n$. Observe that the mode $\widehat{Y}_{\text{MaxMod}}$ minimizes faster the $E_n$ criterion, leading to negligible values after adding $m = 8$ knots (iteration 4, see also Figure \ref{fig:toy2D}). For equispaced designs of knots, the $E_n$ values from $\widehat{Y}_{\text{MaxMod,rect}}$ outperformed those from $\widehat{Y}_{\text{square}}$. However, both $\widehat{Y}_{\text{square}}$ and $\widehat{Y}_{\text{MaxMod,rect}}$ led to suboptimal results due to the equispaced restriction.
\begin{figure}[t!]
\centering
\includegraphics[width=0.85\columnwidth]{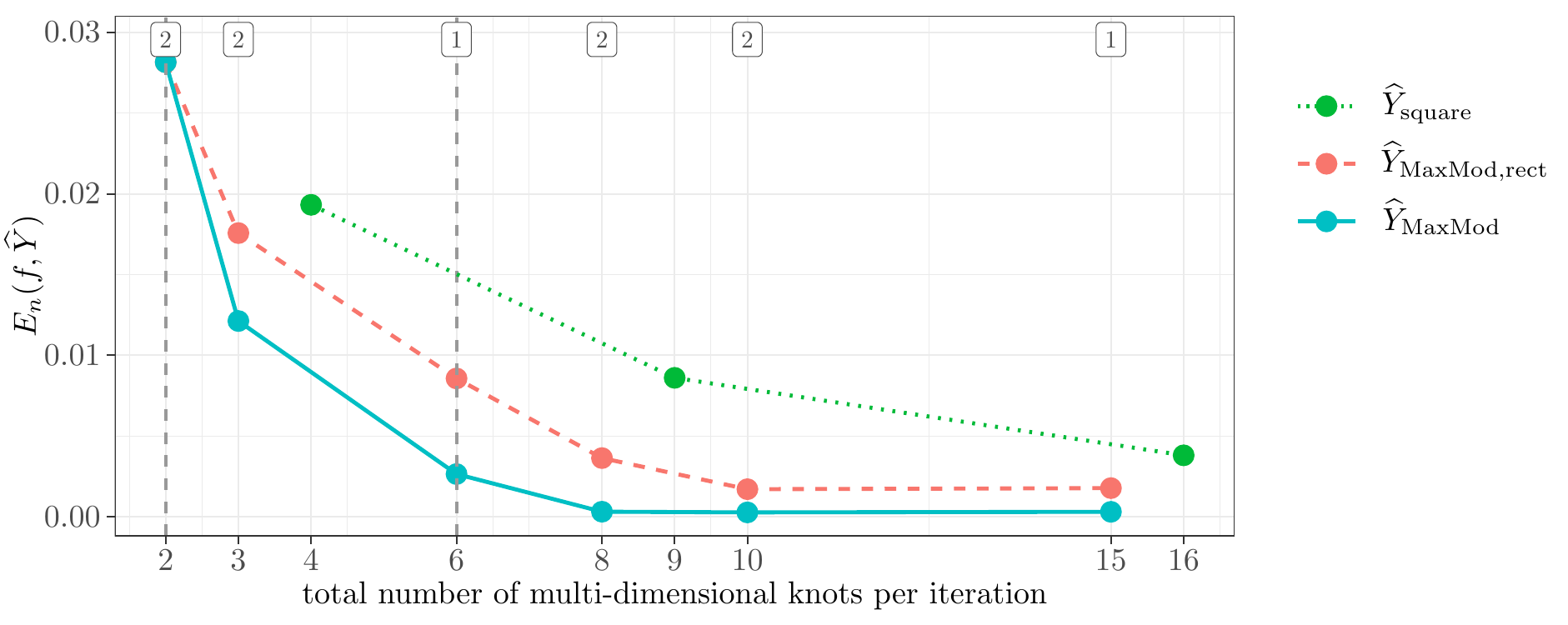}
\caption{Evolution of the (normalized) bending energy $E_n$ for the example in Figure \ref{fig:toy2D}. Results are shown for $\widehat{Y}_{\text{square}}$ (green dotted line), $\widehat{Y}_{\text{MaxMod,rect}}$ (red dashed line) and $\widehat{Y}_{\text{MaxMod}}$ (blue solid line). For $\widehat{Y}_{\text{MaxMod,rect}}$ and $\widehat{Y}_{\text{MaxMod}}$, the labels on top are given for each iteration. They denote which dimension has been refined by MaxMod and vertical dashed lines indicate when a new active dimension is added.}
\label{fig:toy2Derrors}
\end{figure}

\subsubsection{Computation time}
We report the computation time taken by MaxMod and compare it with equispaced knots evenly allocated to each variable (mode $\widehat{Y}_{\text{square}}$). We consider the execution of MaxMod up to $m=6$ and the equispaced approach with $m=16$.
Indeed,  from Figure \ref{fig:toy2Derrors}, both models lead to similar $E_n$ values. MaxMod led to a total lapse time of about 1.4s, while the equispaced approach led to 0.13s.
The computation time of MaxMod is mostly due to the update of the covariance parameters (maximum likelihood estimation) each time a new knot or active variable is tried.

Thus, the execution time of MaxMod is larger than for the equispaced approach, for the same accuracy. However, the benefit of MaxMod is that it reduces the number of knots. This is crucial for the subsequent exploitation of the Gaussian process model. 
More precisely, the cost of numerical sampling of constrained conditional simulations of $f$, that enables to compute confidence intervals \cite{LopezLopera2019lineqGPNoise,LopezLopera2017FiniteGPlinear,maatouk2017gaussian}, increases with the number of knots. Here (in two dimensions) this cost is mild: obtaining 1000 constrained GP simulations required 0.08s with the knots provided by MaxMod, compared to 0.14s for the equispaced approach.\footnote{In this paper, we considered the Hamiltonian Monte Carlo sampler proposed in \cite{Pakman2014Hamiltonian} for simulating constrained GP realizations.} However, the benefit of having less knots with MaxMod is key in higher dimension, as we illustrate below in dimension 5 (see Table \ref{tab:CPUtimesBRGMApp}). Finally, the equispaced approach becomes intractable when the dimension $D$ is larger than approximately $10$, as the available numerical routines are unable to compute a mode with more than about a thousand ($2^{10}$) knots.
This approach may be already inaccurate in the range $6-10$ for $D$, as 
the common number of knots for the variables cannot exceed $3$.  
In contrast, we will implement MaxMod up to dimension $20$.

\subsubsection{Impact of the sample size}
In the context of expensive black box objective functions, we are interested in having accurate predictions with small sample sizes (number of observations of $f$). 
We assess the quality of the three modes for varying sample sizes $n = 10, \ldots, 50$, considering again maximin LHDs.
We use the optimal configurations of knots obtained in the experiment in Figure \ref{fig:toy2Derrors}, that is $m = 15$ ($m_1 = 3, m_2 = 5$) for $\widehat{Y}_{\text{MaxMod,rect}}$ and $\widehat{Y}_{\text{MaxMod}}$ and $m=16$ ($m_1 = m_2 = 4$) for $\widehat{Y}_{\text{square}}$. 
Figure \ref{fig:toy2DerrorsvsSampleSize} shows that for small sample sizes, the three modes lead to small $E_n$ values, and that both $\widehat{Y}_{\text{MaxMod,rect}}$ and $\widehat{Y}_{\text{MaxMod}}$ outperform $\widehat{Y}_{\text{square}}$.
For this latter mode, we can see a local increase of the error from sample size $20$ to $30$, which may be explained by the randomness of the maximin LHDs. 
Hence, again, the prediction accuracy improves when considering the optimal design of knots of MaxMod.
\begin{figure}[t!]
\centering
%
\includegraphics[width=0.85\columnwidth]{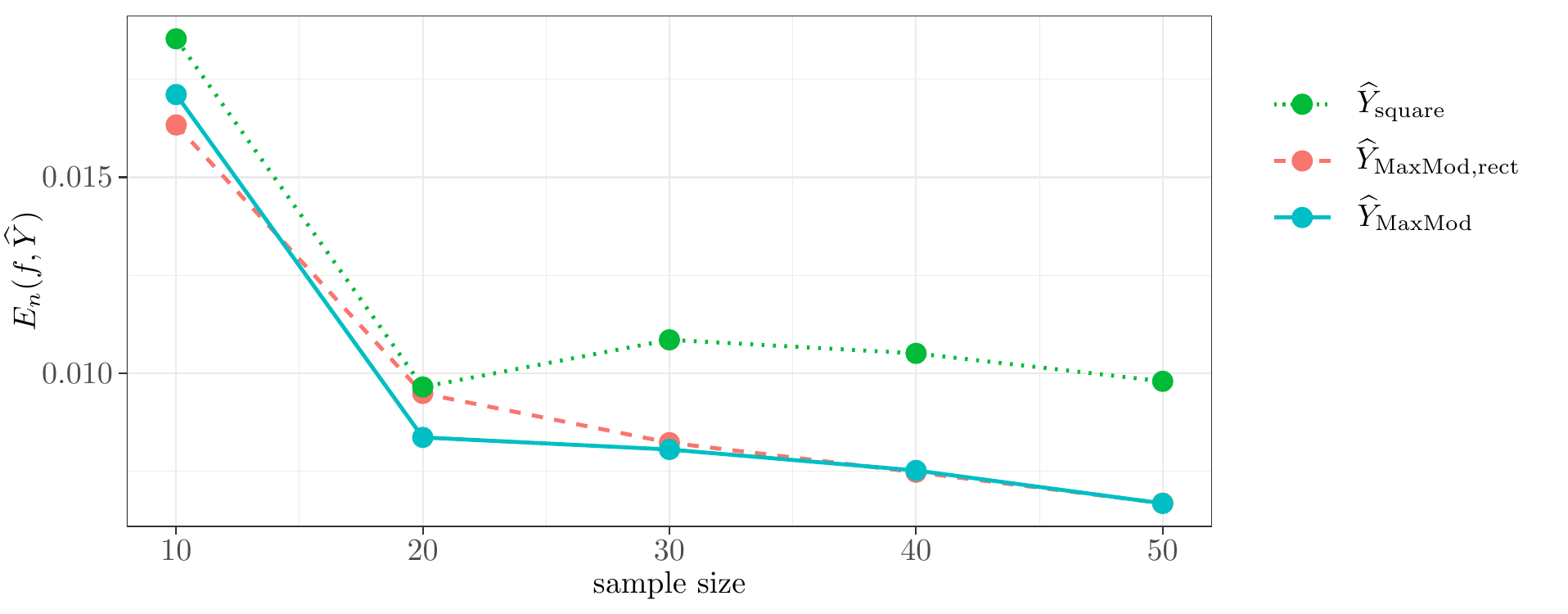}
\caption{Evolution of the (normalized) bending energy $E_n$ for the example in Figure \ref{fig:toy2D} with respect to the sample size $n = 10, \ldots, 50$. The $E_n$ criterion is computed over a $10 \times 10$ equispaced grid.}
\label{fig:toy2DerrorsvsSampleSize}
\end{figure}

\subsection{Dimension reduction illustration}
\label{section:numerical:experiments:subsec:dimension:reduction}
We now focus on the capability of MaxMod to perform dimension reduction. 
We apply the same stopping rule as in the previous test case, with a tolerance equal to $5\times 10^{-3}$.
We consider the target function: 
\begin{equation}
f(x) = \sum_{i=1}^{d} \arctan\left(5\bigg[1-\frac{i}{d+1}\bigg] x_i\right),
\label{eq:modatanFun}
\end{equation}
with $x \in [0,1]^d$. Note that $f$ is completely monotone exhibiting lesser growth rates as $i$ increases. In addition to $(x_1, \ldots, x_d)$, we include $D - d$ virtual variables, indexed as $(x_{d+1}, \ldots, x_{D})$, which will compose the subset of inactive dimensions since $f$ does not depend on them. We consider $D \in \{5,10,15,20\}$ and $d \in \{2,3,4,5\}$. We also analyze the case where $d=D$. For each value of $D$, we evaluate $f$ at a maximin LHD with $n = 10\times D$ points. For any possible combination of $D$ and $d$, we apply MaxMod expecting at properly finding the true $d$ active dimensions when $d < D$. When $d = D$, we expect that MaxMod will concentrate the computational budget on the most important input variables $x_1, \ldots, x_{d^\ast}$ with $d^\ast < d$.

From Table \ref{tab:ActInactivePerformances}, we observe that when $d<D$ MaxMod properly identifies the $d$ dimensions that are actually active, leading to small $E_n$ results when considering either $\widehat{Y}_{\text{MaxMod,rect}}$ or $\widehat{Y}_{\text{MaxMod}}$.\footnote{In all the replicates in Table \ref{tab:ActInactivePerformances}, MaxMod leads to small noise variance parameters $\hat{\tau}_{\text{MaxMod}}^2$ when $d < D$. The largest value, obtained for $D = 20$ and $d = 5$, is $\hat{\tau}_{\text{MaxMod}}^2 = 1.75 \times 10^{-3}$ ($\approx 0.4\%$ of the variance of the observations).} When $d = D$, MaxMod successfully detects the most important variables $ 1 , \ldots , d^\ast$ with $d^\ast < D$ for $D \geq 10$.
When the active dimension $d$ is larger or equal to $10$, the function approximation problem is intrinsically more difficult. Thus MaxMod stops before being able
to reach the stopping criterion, since the number of knots becomes too large for the numerical routines that compute the mode ($m > 1400$). In these cases, we fix the maximal number of iterations to $12$, which nevertheless yields small values of $E_n$.

Overall, $\widehat{Y}_{\text{MaxMod}}$ outperformed $\widehat{Y}_{\text{MaxMod,rect}}$ due to its flexibility to freely allocate knots without being limited to equispaced designs.
Finally, we remark that the mode $\widehat{Y}_{\text{square}}$ is intractable here due to the large dimension $D$.

\begin{table}
\caption{Performance of MaxMod for the example in Section \ref{section:numerical:experiments:subsec:dimension:reduction}. Results are shown for both $\widehat{Y}_{\text{MaxMod,rect}}$ and $\widehat{Y}_{\text{MaxMod}}$ considering $D \in \{5,10,15,20\}$ and $d \in \{2,3,4,5\}$. We also consider the case where $d=D$. The activated dimensions, number of one-dimensional knots per active dimension and $E_n$ results are displayed for any combination of $D$ and $d$.
	$^\dagger$The maximal number of iterations is fixed to $12$ for the experiments where $d \geq 10$. }
\label{tab:ActInactivePerformances}
\centering
\begin{tabular}{|c|c|c|c|c|c|}
	\hline
	$D$ & $d$ & active dim. & knots per dim. &  $E_n(f,\widehat{Y}_{\text{MaxMod,rect}})$ & $E_n(f, \widehat{Y}_{\text{MaxMod}})$ \\
	\hline		
	\multirow{4}{*}{5} & 2 & (1, 2) & (7, 4) & $\boldsymbol{2.59\times 10^{-5}}$ & $4.51\times 10^{-5}$  \\
	& 3 & $(1 ,2, 3)$ & $(6, 6, 4)$ & ${6.42\times 10^{-4}}$ & $\boldsymbol{4.09\times 10^{-4}}$ \\
	& 4 & $(1, \ldots, 4)$ & $(4, 4, 3, 2)$ & $\boldsymbol{9.02\times 10^{-4}}$ & ${9.05\times 10^{-4}}$ \\
	& 5 & $(1, \ldots, 5)$ & $(3, 4, 4, 3, 2)$ & $\boldsymbol{1.15\times 10^{-3}}$ & ${1.19\times 10^{-3}}$\\
	\hline
	\multirow{4}{*}{10} & 2 & $(1, 2)$ & $(5, 3)$ & $\boldsymbol{1.03\times 10^{-5}}$ & ${2.78\times 10^{-5}}$ \\ 
	& 3 & $(1, 2, 3)$ & $(5, 4, 3)$ & ${2.13\times 10^{-3}}$ & $\boldsymbol{1.79\times 10^{-3}}$ \\
	& 4 & $(1, \ldots, 4)$ & $(5, 3, 3, 2)$ & ${3.10\times 10^{-4}}$ & $\boldsymbol{2.89\times 10^{-4}}$ \\
	& 5 & $(1, \ldots, 5)$ & $(3, 4, 3, 3, 2)$ & ${7.40\times 10^{-4}}$ & $\boldsymbol{4.31\times 10^{-4}}$ \\
	& 10$^\dagger$ & $(1, \ldots, 7)$ & $(3, 3, 3, 2, 3, 3, 2)$ & ${4.44\times 10^{-3}}$ & $\boldsymbol{3.94\times 10^{-3}}$ \\
	\hline
	\multirow{4}{*}{15} & 2 & (1, 2) & (4, 3) & $\boldsymbol{1.73\times 10^{-4}}$ & ${1.85\times 10^{-4}}$ \\ 
	& 3 & $(1, 2, 3)$ & $(4, 3, 3)$ & ${2.13\times 10^{-4}}$ & $\boldsymbol{1.94\times 10^{-4}}$\\
	& 4 & $(1, \ldots, 4)$ & $(3, 3, 3, 2)$ & ${4.06\times 10^{-4}}$ & $\boldsymbol{1.94\times 10^{-4}}$ \\
	& 5 & $(1, \ldots, 5)$ & $(3, \ldots, 3, 2)$ & ${2.26\times 10^{-4}}$ & $\boldsymbol{9.29\times 10^{-5}}$ \\
	& 15$^\dagger$ & $(1, \ldots, 6)$ & $(3, \ldots, 3)$ & ${6.22 \times 10^{-3}}$ & $\boldsymbol{5.93 \times 10^{-3}}$ \\			
	\hline
	\multirow{4}{*}{20} & 2 & (1, 2) & (5, 3) & ${9.88\times 10^{-5}}$ & $\boldsymbol{9.37\times 10^{-5}}$\\ 
	& 3 & $(1, 2, 3)$ & $(4, 4, 3)$ & $\boldsymbol{1.40\times 10^{-4}}$ & $\boldsymbol{1.40\times 10^{-4}}$\\
	& 4 & $(1, \ldots, 4)$ & $(4, 3, 3, 3)$ & ${3.48\times 10^{-4}}$ & $\boldsymbol{1.97\times 10^{-4}}$\\
	& 5 & $(1, \ldots, 5)$ & $(3, \ldots, 3, 2)$ & ${5.60\times 10^{-4}}$ & $\boldsymbol{2.83\times 10^{-4}}$ \\
	& 20$^\dagger$ & $(1, \ldots, 7)$ & $(2, 3, \ldots, 3, 2)$ & ${5.58\times 10^{-3}}$ & $\boldsymbol{5.55\times 10^{-3}}$ \\
	\hline
\end{tabular}
\end{table}

\subsection{Coastal flooding application in 5D}
\label{section:numerical:experiments:subsec:BRGM}

We consider the 5D coastal flood application studied in \cite{Azzimonti2018CoastalFlooding,LopezLopera2019lineqGPNoise}, available in the R package \texttt{profExtrema} \cite{Azzimonti2018profExtrema}. In the past, a flood event at the \textit{Boucholeurs} area (\textit{La Rochelle}, France) was induced by an overflow on the Atlantic ocean caused by the Xynthia storm in 2010. To prevent adverse coastal flood events, such as the one led by the Xynthia storm, accurate forecast and early-warning systems (see, e.g., \cite{Azzimonti2018CoastalFlooding,LopezLopera2019lineqGPNoise,Rohmer2018RandomForest} for GP-based ones) are required.


The dataset contains 200 observations of the flooded area ($A_{\text{flood}} [m^2]$) driven by five offshore forcing conditions (inputs) at the \textit{Boucholeurs} area: tide ($T [m]$), surge ($S [m]$), the phase difference ($\phi$, hours) between the surge peak and the high tide, the time duration of the raising part ($t_{-}$, hours) and the falling part ($t_{+}$, hours) of the (triangular) surge signal. In particular, it is known that $A_{\text{flood}}$ increases as $T$ and $S$ increase. According to \cite{Azzimonti2018CoastalFlooding,LopezLopera2019lineqGPNoise}, while the contribution of $T$, $S$, $t_{-}$ and $t_{+}$ are almost linear, $A_{\text{flood}}$ exhibits a higher variation across $\phi$. Thus, in our experiments, we may expect that MaxMod properly concentrate one-dimensional knots across $\phi$ rather than across the other dimensions. As suggested in \cite{LopezLopera2019lineqGPNoise}, we consider here $Y := \log_{10}(A_{\text{flood}})$ as the output variable.

As shown in \cite{LopezLopera2019lineqGPNoise}, enforcing a GP-based coastal emulator to both positivity and monotonicity (with respect to $T$ and $S$) constraints leads to a more reliable prediction. There, the number of one-dimensional knots per dimension has been manually fixed looking for a trade-off between the computational cost and the quality of resolution of the constrained GP. Here, we aim at applying MaxMod and comparing the $E_n$ results to those from the knots in \cite{LopezLopera2019lineqGPNoise}. As in Sections \ref{section:numerical:experiments:subsec:2Dtoy} and \ref{section:numerical:experiments:subsec:dimension:reduction}, we compute the modes $\widehat{Y}_{\text{square}}$, $\widehat{Y}_{\text{MaxMod,rect}}$ and $Y_{\text{MaxMod}}$. In addition, we compute the mode resulting from the knots in \cite{LopezLopera2019lineqGPNoise}, denoted as $\widehat{Y}_{\ast}$. Since the target function $Y$ is actually unknown, the bending energy $E_n$ is computed over the available 200 observations: $E_n(Y, \widehat{Y}) = \sum_{i = 1}^{200} (Y_{i} - \widehat{Y}_i)^2/\sum_{i=1}^{200} Y_{i}^2$. Figure \ref{fig:toyBRGMerrors} shows that MaxMod results in a total of $m = 4\times3\times6\times3\times2 = 432$ knots, leading to $E_n(Y, \widehat{Y}_{\text{MaxMod,rect}}) = 9.17 \times10^{-3}$ and $E_n(Y, \widehat{Y}_{\text{MaxMod}}) = 8.81 \times10^{-3}$. These results are comparable to those from $\widehat{Y}_{\ast}$, $E_n(Y, \widehat{Y}_{\ast}) = 8.72 \times10^{-3}$ with $m = 4\times4\times5\times3\times3 = 720$, and from $\widehat{Y}_{\text{square}}$, $E_n(Y, \widehat{Y}_{\text{square}}) = 8.72 \times10^{-3}$ with $m = 4^5 = 1024$. Moreover, as expected, MaxMod concentrated the computational budget on the input $\phi$ ($6$ one-dimensional knots) rather than the other ones ($4,3,3,2$) since $Y$ varies the most across $\phi$. In Figure \ref{fig:toyBRGMerrors}, it appears that $\widehat{Y}_{\text{MaxMod}}$ is the most efficient, needing a smaller number of knots than $\widehat{Y}_{\text{square}}$, $\widehat{Y}_{\text{MaxMod,rect}}$ and $\widehat{Y}_{\ast}$ to reach a given value of $E_n$. In terms of noise variance, MaxMod estimates $\widehat{\tau}_{\text{MaxMod}}^2 = 6.16 \times 10^{-2}$, a small (but non negligible) value equivalent to $\approx 7.6\%$ of the variance of the observations.
Our interpretation for this small value is that it accounts for possible numerical instabilities of the computer code. Furthermore it improves the accuracy of the mode function and speeds up its computation by making the inequality constraints easier to satisfy (see \cite{LopezLopera2019lineqGPNoise} for further discussions).
\begin{figure}[t!]
\centering
\includegraphics[width=1\columnwidth]{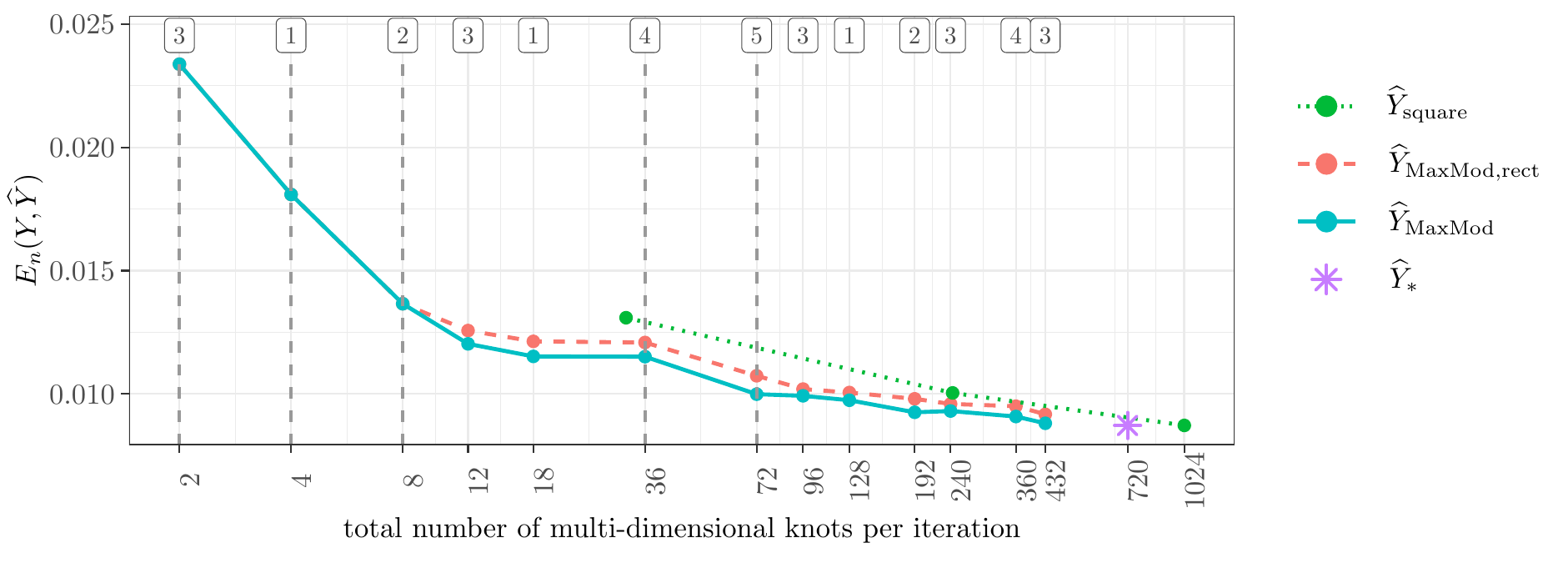}
\caption{Evolution of the bending energy $E_n$ criterion for the coastal flooding application in Section \ref{section:numerical:experiments:subsec:BRGM}. The panel description is the same as in Figure \ref{fig:toy2Derrors}. The $E_n$ value associated to the mode $\widehat{Y}_{\ast}$, yielded by using the configuration of the knots suggested in \cite{LopezLopera2019lineqGPNoise}, is displayed by a purple asterisks at $m = 720$.}
\label{fig:toyBRGMerrors}
\end{figure}

Table \ref{tab:CPUtimesBRGMApp} shows the CPU times required
for training the GP models from $\widehat{Y}_{\text{square}}$ (covariance parameter estimation) and $\widehat{Y}_{\text{MaxMod}}$ (Algorithm \ref{alg:iterative2} including covariance parameter estimation). It also shows the elapsed times related to the computation of the modes $\widehat{Y}_{\text{square}}$ and $\widehat{Y}_{\text{MaxMod}}$ (with given covariance parameters), and to the sampling of 100 conditional realizations of the resulting constrained processes.
The training time is smaller for $\widehat{Y}_{\text{square}}$.
The computation of $\widehat{Y}$ is faster for MaxMod, but for both MaxMod and the equispaced approach, the computation time for $\widehat{Y}$ is small compared to the training time. With the $432$ knots from MaxMod, the $100$ conditional realizations  are obtained in less than 2 minutes. 
In contrast, with the $1024$ knots from $\widehat{Y}_{\text{square}}$, the Hamiltonian Monte Carlo routine for sampling the conditional distribution proposed in \cite{Pakman2014Hamiltonian} does not converge, because of the overly high dimension. This highlights the benefit of MaxMod in real applications: it requires an acceptable increase of the (offline) training time in order to reduce the constrained GP model dimension, which enhances its subsequent exploitation (conditional sampling).
\begin{table}
\centering
\caption{CPU lapse times for the coastal flooding application in Section \ref{section:numerical:experiments:subsec:BRGM}.} 
\label{tab:CPUtimesBRGMApp}
\begin{tabular}{|c|c|c|c|c|c|}
	\hline
	& & \multirow{2}{*}{$E_n(Y, \widehat{Y})$} & \multicolumn{3}{c|}{CPU time [s]} \\ \cline{4-6}
	Approach & $m$ & \multirow{2}{*}{[$1\times 10^{-3}$]} & \multirow{2}{*}{Training step} & Computation & Sampling step \\
	& & & & of $\widehat{Y}$ & with 100 realizations \\
	\hline
	$\widehat{Y}_{\text{square}}$ & $1024$ &  8.72 & 49.1 & 8.03 & non converged after 1 day \\			
	$\widehat{Y}_{\text{MaxMod}}$ & $432$ & 8.81 & 949.5 & 0.58 & 108.72 \\
	\hline
\end{tabular}
\end{table}


\section{Conclusion} \label{section:conclusion}

This paper introduces MaxMod, that sequentially inserts one-dimen\-sional knots or adds active variables to a constrained GP regression model. This algorithm results in the first constrained GP model that at the same time satisfies the constraints everywhere and that is not restricted to small dimensional cases in practice. 

A proof of convergence, for a fixed dataset and as the number of iterations goes to infinity, guarantees that, despite its sequential nature, MaxMod globally converges to an optimal infinite dimensional model. In establishing this convergence, the notion of a multi-affine extension is constructed. Furthermore, the convergence of finite-dimensional GP models is shown in settings where the multi-dimensional knots are not dense in the input domain, thereby extending the recent literature. This construction, and this latter convergence result, may be of special and independent interest, together with the corresponding proof techniques. 

In Section \ref{section:numerical:experiments}, with simulated and real data, it is confirmed numerically that MaxMod is tractable and accurate (at least in dimension $D=20$), and typically needs less multi-dimensional knots than the other state-of-the-art constrained GP models. 
We demonstrate the strong benefit of having fewer knots when subsequently computing confidence intervals from constrained GP models.	
These numerical examples also indicate that MaxMod successfully detects the active variables when the effective dimension $d$ is small. 
Even when $d = D$, MaxMod was able to detect the most important variables.

\section*{Acknowledgments}
We are grateful to N. Durrande for constructive discussions. This research was supported by the Chair in Applied Mathematics OQUAIDO\footnote{OQUAIDO gathers partners in technological research (BRGM, CEA, IFPEN, IRSN, Safran, Storengy) and academia (CNRS, Ecole Centrale de Lyon, Mines Saint-Etienne, Univ. of Grenoble, Univ. of Nice, Univ. of Toulouse) around advanced methods for Computer Experiments.}, the ANR\footnote{French National Research Agency, under the RISCOPE project (ANR-16-CE04-0011).}, and the Isaac Newton Institute for Mathematical Sciences, Cambridge.\footnote{UNQ program (EPSRC grantEP/K032208/1).} Part of this research was carried out when the second author was affiliated to IMT and BRGM. We are indebted to the referees, for their constructive suggestions that greatly improved the manuscript. 

\appendix

\section{Expression of the linear inequality constraints in Section \ref{subsubsection:obtaining:linear:inequality:constraints} for boundedness, monotonicity and componentwise convexity} \label{appendix:obtaining:linear:inequality:constraints}

Let us provide expressions of $M (\ineqset_{\activeset} )$ and of $v( \ineqset_{\activeset} )$ such that \eqref{eq:equivalence:constraint} holds when $\ineqset_D$ is given by \eqref{eq:boundedness:D}, \eqref{eq:monotonicity:D} and \eqref{eq:convexity:D}. When $\ineqset_D$ is given by \eqref{eq:boundedness:D}, we define $M (\ineqset_{\activeset} )$ and $v( \ineqset_{\activeset} )$ as follows. The matrix $M (\ineqset_{\activeset} )$ has $2 |\multiset|$ rows and thus $v( \ineqset_{\activeset} )$ is of size $2 |\multiset|$. The first $|\multiset|$ rows and values of $v( \ineqset_{\activeset} )$ are indexed by $\multi \in \multiset$ and defined by, for $\multi' \in \multiset$,
\begin{equation} \label{eq:equivalence:constraint:boundedness:one}
	M (\ineqset_{\activeset})_{\multi,\multi'}
	= - \mathbf{1}_{\multi = \multi'}
	~ ~
	\text{and}
	~ ~
	v( \ineqset_{\activeset} )_{\multi} = -a.
\end{equation}

The last $|\multiset|$ rows and values of $v( \ineqset_{\activeset} )$ are indexed by $\multi \in \multiset$ and defined by, for $\multi' \in \multiset$,
\begin{equation} \label{eq:equivalence:constraint:boundedness:two}
	M (\ineqset_{\activeset})_{ |\multiset| + \multi,\multi'}
	=  \mathbf{1}_{\multi = \multi'}
	~ ~
	\text{and}
	~ ~
	v( \ineqset_{\activeset} )_{|\multiset|+ \multi} = b,
\end{equation}

using the slight abuse of notation $|\multiset|+ \multi$ to denote the last $|\multiset|$ rows and values. The fact that \eqref{eq:equivalence:constraint} holds with \eqref{eq:equivalence:constraint:boundedness:one} and \eqref{eq:equivalence:constraint:boundedness:two} can be simply shown to follow from the proof of Lemma \ref{lem:cond:PfC:subset:C} and from \eqref{eq:pi:S:f:equal:multiaffine:extension}. Similar equations were also stated, in the previous references
\cite{LopezLopera2017FiniteGPlinear,maatouk2017gaussian}.

When $\ineqset$ is given by \eqref{eq:monotonicity:D}, we define $M (\ineqset_{\activeset} )$ and $v( \ineqset_{\activeset} )$ as follows. The matrix
$M (\ineqset_{\activeset} )$ is composed of $d = | \activeset |$ vertically stacked matrices $M^{(1)} (\ineqset_{\activeset} ),\ldots,M^{(d)} (\ineqset_{\activeset} )$ and $v( \ineqset_{\activeset} )$ is composed of $d$ vertically stacked column vectors $v^{(1)}( \ineqset_{\activeset} ), \ldots , v^{(d)}( \ineqset_{\activeset} )$.
For $i = 1 , \ldots , d$, the matrix $M^{(i)} (\ineqset_{\activeset} )$
has $ (m_{\subdiv_{a_i}}-1)  \prod_{j=1,\ldots,d,j \neq i}  m_{\subdiv_{a_j}}$ rows and $v^{(i)}( \ineqset_{\activeset} )$ has the same number of components. The rows and values of $v^{(i)}( \ineqset_{\activeset} )$ are indexed by the multi-indices $\multi = (\ell_{a_1},\ldots,\ell_{a_d}) 
\in
\big(  \prod_{j=1}^{i-1} \{ 1 , \ldots , m_{\subdiv_{a_j}}\} \big)
\times
\{ 2 , \ldots , m_{\subdiv_{a_i}}\} 
\times
\big(  \prod_{j=i+1}^d \{ 1 , \ldots , m_{\subdiv_{a_j}}\} \big) 
$ 
and defined by 
\begin{align} \label{eq:equivalence:constraint:monotonicity}
	& M^{(i)}(\ineqset_{\activeset})_{(\ell_{a_1},\ldots,\ell_{a_d}),(\ell'_{a_1},\ldots,\ell'_{a_d})}
	= - \mathbf{1}_{\ell_{a_j} = \ell'_{a_j}~ \text{for} ~ j \neq i , \ell_{a_i} = \ell'_{a_i}}
	+
	\mathbf{1}_{\ell_{a_j} = \ell'_{a_j}~ \text{for} ~ j \neq i , \ell_{a_i} - 1 = \ell'_{a_i}}
	\\
	&
	~ ~
	\text{and}
	~ ~
	v^{(i)}( \ineqset_{\activeset} )_{(\ell_{a_1},\ldots,\ell_{a_d})} = 0. \notag
\end{align}

Equation \eqref{eq:equivalence:constraint:monotonicity} means that the set of values on the tensorized grid of $d$-dimensional knots is component-wise non-decreasing. It can be shown from the proof of Lemma \ref{lem:cond:PfC:subset:C} and from \eqref{eq:pi:S:f:equal:multiaffine:extension} that $M (\ineqset_{\activeset} )$ and $v( \ineqset_{\activeset} )$ given by \eqref{eq:equivalence:constraint:monotonicity}
imply \eqref{eq:equivalence:constraint}. This was noticed in the previous references \cite{LopezLopera2017FiniteGPlinear,maatouk2017gaussian}.

Finally, when $\ineqset$ is given by \eqref{eq:convexity:D},
we define $M (\ineqset_{\activeset} )$ and $v( \ineqset_{\activeset} )$ as follows.
The matrix $M (\ineqset_{\activeset} )$ is composed of $d$ vertically stacked matrices $M^{(1)} (\ineqset_{\activeset} ),\ldots,M^{(d)} (\ineqset_{\activeset} )$ and $v( \ineqset_{\activeset} )$ is composed of $d$ vertically stacked column vectors $v^{(1)}( \ineqset_{\activeset} ), \ldots , v^{(d)}( \ineqset_{\activeset} )$.
For $i = 1 , \ldots , d$, the matrix $M^{(i)} (\ineqset_{\activeset} )$
has $ (m_{\subdiv_{a_i}}-2)  \prod_{j=1,\ldots,d,j \neq i}  m_{\subdiv_{a_j}}$ rows and $v^{(i)}( \ineqset_{\activeset} )$ has the same number of components. The rows and values of $v^{(i)}( \ineqset_{\activeset} )$ are indexed by the multi-indices $\multi = (\ell_{a_1},\ldots,\ell_{a_d}) \in
\big(  \prod_{j=1}^{i-1} \{ 1 , \ldots , m_{\subdiv_{a_j}}\} \big)
\times
\{ 3 , \ldots , m_{\subdiv_{a_i}}\} 
\times
\big(  \prod_{j=i+1}^d \{ 1 , \ldots , m_{\subdiv_{a_j}}\} \big) 
$  and defined by 
\begin{align} \label{eq:equivalence:constraint:convexity}
	& M^{(i)}(\ineqset_{\activeset})_{(\ell_{a_1},\ldots,\ell_{a_d}),(\ell'_{a_1},\ldots,\ell'_{a_d})}
	\\
	& = - \mathbf{1}_{\ell_{a_j} = \ell'_{a_j}~ \text{for} ~ j \neq i , \ell_{a_i} = \ell'_{a_i}}
	+
	2 \mathbf{1}_{\ell_{a_j} = \ell'_{a_j}~ \text{for} ~ j \neq i , \ell_{a_i} - 1 = \ell'_{a_i}}
	- \mathbf{1}_{\ell_{a_j} = \ell'_{a_j}~ \text{for} ~ j \neq i , \ell_{a_i}-2 = \ell'_{a_i}}
	\notag
	\\ 
	\text{and} ~ ~ & 
	v^{(i)}( \ineqset_{\activeset} )_{(\ell_{a_1},\ldots,\ell_{a_d})} = 0. \notag
\end{align}

Equation \eqref{eq:equivalence:constraint:convexity} means that the set of values on the tensorized grid of $d$-dimensional knots is component-wise convex. It follows from the proof of Lemma \ref{lem:cond:PfC:subset:C} and from \eqref{eq:pi:S:f:equal:multiaffine:extension} that $M (\ineqset_{\activeset} )$ and $v( \ineqset_{\activeset} )$ given by \eqref{eq:equivalence:constraint:convexity}
imply \eqref{eq:equivalence:constraint}. To the best of our knowledge, this was only shown in dimension one, in earlier references. 

\begin{remark}
	The equation \eqref{eq:equivalence:constraint:convexity} is defined only for $i \in \{ 1 , \ldots,d\}$ such that $m_{\subdiv_{a_i}} \geq 3$. If $i$ is such that $m_{\subdiv_{a_i}} = 2$, then the matrix $M^{(i)} (\ineqset_{\activeset} )$ and the vector $v^{(i)}( \ineqset_{\activeset} )$ can be removed from $M (\ineqset_{\activeset} )$ and $v( \ineqset_{\activeset} )$. Indeed, when $m_{\subdiv_{a_i}} = 2$, the function $ \pi_\subdiv ( \xi_\activeset) $ is linear (thus convex) along the dimension $a_i$.
\end{remark} 

\section{Computing the $L^2$ difference between modes in practice}
\label{subsection:computing:Ldeux:difference:modes}

At step $m$ of the MaxMod algorithm, the current mode function is $\widehat{Y}_{\activeset_{m} , \, \subdiv^{(m)}} : [0,1]^{d_{m}} \to \mathbb{R}$, with $d_{m} =|\activeset_{m}| $ and with
\begin{equation} \label{eq:hatY:step:m:minus:one}
	\widehat{Y}_{\activeset_{m} , \, \subdiv^{(m)}}
	=
	\sum_{\multi \in \mathcal{L}_{\subdiv^{(m)}}} 
	(\widehat{\alpha}_{\activeset_{m},\subdiv^{(m)}})_{\multi}
	\phi_{\multi}^{(\subdiv^{(m)})}.
\end{equation}

To simplify the exposition in Section \ref{subsection:computing:Ldeux:difference:modes}, we let $\activeset_{m} = \activeset$, and we consider the case where $\activeset = \{ 1 , \ldots, d\}$. The results of Section \ref{subsection:computing:Ldeux:difference:modes} can then be immediately extended to a general set $\activeset$, by permuting indices. We also let $\subdiv^{(m)} = \subdiv = (\subdiv_{1},\ldots,\subdiv_{d})$.

\subsection{$L^2$ difference for a knot insertion to an active variable} 
\label{subsubsection:knot:added:to:active:variable}

Consider the active variable $1$, and a new knot $t \in [0,1] \backslash \subdiv^{(m)}_{1}$ (for the variable $1$). Again, the choice of the active variable  $1$ enables to simplify the exposition and the results of Section \ref{subsubsection:knot:added:to:active:variable} can be immediately extended to a general active variable $i \in \{ 1 , \ldots , d\}$.

The candidate mode function is
\begin{equation} \label{eq:hatY:step:m}
	\widehat{Y}_{\activeset  , \, \subdiv \cup_{1} t}
	=
	\sum_{\multi \in \mathcal{L}_{\subdiv\cup_{1} t }} 
	(\widehat{\alpha}_{\activeset,\subdiv\cup_{1} t})_{\multi}
	\phi_{\multi}^{(\subdiv\cup_{1} t)}.
\end{equation}
For $j=1,\ldots,d$, write $m_{j} = m_{\subdiv_{j}}$. With this notation, $\subdiv_{j} = \{  t^{(\subdiv_{j})}_{(0)},\ldots, t^{(\subdiv_{j})}_{(m_{j}+1)} \}$ with the ordered knots $t^{(\subdiv_{j})}_{(0)} < \dots < t^{(\subdiv_{j})}_{(m_{j}+1)}$.

Let $\subdiv' = \subdiv \cup_{1} t = (\subdiv'_{1} , \ldots , \subdiv'_{d})$. We have $\subdiv'_{j} = \subdiv_{j}$ for $j \in \{ 2 , \ldots, d\}$. Let $\nu \in \{ 1 , \ldots , m_{1} - 1\}$ be such that $t^{\subdiv_{1}}_{( \nu )} < t < t^{\subdiv_{1}}_{( \nu + 1 )}$. Then we have $\subdiv'_{1} = \{  t^{(\subdiv'_{1})}_{(0)},\ldots, t^{(\subdiv'_{1})}_{(m_{1}+2)} \}$ with the ordered knots $t^{(\subdiv'_{1})}_{(0)} < \dots < t^{(\subdiv'_{1})}_{(m_{1}+2)}$ with
\[
(t^{(\subdiv'_{1})}_{(0)} , \ldots , t^{(\subdiv'_{1})}_{(m_{1}+2)})
=
(t^{(\subdiv_{1})}_{(0)} , \ldots ,  t^{(\subdiv_{1})}_{(\nu)}
,
t
,
t^{(\subdiv_{1})}_{(\nu+1)}
, \ldots , 
t^{(\subdiv_{1})}_{(m_{1}+1)}
).
\]

Then the next proposition provides a computationally efficient formula for the $L^2$ difference between modes. To understand Proposition \ref{prop:efficient:computation:new:knot}, note that the difference between the current mode function \eqref{eq:hatY:step:m:minus:one} and the new mode function \eqref{eq:hatY:step:m} can be expressed as a weighted sum of the basis functions of the new (refined) subdivision $\subdiv'$.
In Proposition \ref{prop:efficient:computation:new:knot}, \eqref{eq:expression:coeff:in:new:basis:new:knot:one}, \eqref{eq:expression:coeff:in:new:basis:new:knot:two} and \eqref{eq:expression:coeff:in:new:basis:new:knot:three} first provide the expressions of the corresponding coefficients. Then, it just suffices to compute the matrix of $L^2$ inner products between the $d$-dimensional basis functions of $\subdiv'$.
This matrix is given in \eqref{eq:matrix:Psi}. It is the tensor product of the $d$
matrices of $L^2$ inner products between the one-dimensional basis functions of the $d$ one-dimensional subdivisions $\subdiv'_1,\ldots,\subdiv'_d$. These latter matrices are provided in \eqref{eq:matrix:psi}.

\begin{proposition}[$L^2$ difference for knot insertion to an active variable] \label{prop:efficient:computation:new:knot}
	Define for $\multi = (\ell_{1},\ldots,\ell_{d}) \in \mathcal{L}_{\subdiv'}$, 
	\begin{equation} \label{eq:expression:coeff:in:new:basis:new:knot:one}
		\beta_{\multi} = 
		(\widehat{\alpha}_{\activeset,\subdiv})_{\multi}
		-
		(\widehat{\alpha}_{\activeset,\subdiv'})_{\multi}
	\end{equation}
	if $\ell_{1} \in \{1 , \ldots , \nu \}$,
	\begin{equation} \label{eq:expression:coeff:in:new:basis:new:knot:two}
		\beta_{\multi} = 
		(\widehat{\alpha}_{\activeset,\subdiv})_{(\nu,\ell_2\ldots,\ell_{d})}
		\frac{ t^{(\subdiv_{1})}_{(\nu+1)} - t }{t^{(\subdiv_{1})}_{(\nu+1)} - t^{(\subdiv_{1})}_{(\nu)}}
		+
		(\widehat{\alpha}_{\activeset,\subdiv})_{(\nu+1,\ell_2,\ldots,\ell_{d})}
		\frac{ t - t^{(\subdiv_{1})}_{(\nu)}  }{t^{(\subdiv_{1})}_{(\nu+1)} - t^{(\subdiv_{1})}_{(\nu)}}
		- (\widehat{\alpha}_{\activeset,\subdiv'})_{\multi}
		,
	\end{equation}
	if $\ell_{1} = \nu+1$ and
	\begin{equation} \label{eq:expression:coeff:in:new:basis:new:knot:three}
		\beta_{\multi} = 
		(\widehat{\alpha}_{\activeset,\subdiv})_{(\ell_1-1,\ell_2,\ldots,\ell_{d})}
		-
		(\widehat{\alpha}_{\activeset,\subdiv'})_{\multi}
	\end{equation}
	if $\ell_{1} \in \{ \nu +2 , \ldots , m_{1} +1 \} $.

	Define for $\multi = (\ell_{1},\ldots,\ell_{d}) \in \mathcal{L}_{\subdiv}$, $\multi' = (\ell'_{1},\ldots,\ell'_{d}) \in \mathcal{L}_{\subdiv'}$,
	\begin{equation} \label{eq:matrix:Psi}
		\Psi_{\multi,\multi'}
		=
		\begin{cases}
			\prod_{j=1,\ldots,d}
			\psi^{(j)}_{\ell_{j},\ell'_{j}}
			& ~ \text{if} ~  ~  ~
			| \ell_{j} - \ell'_{j} | \leq 1
			~ \text{for} ~ j=1,\ldots,d
			\\
			0
			& ~ \text{else} 
		\end{cases},
	\end{equation}
	with, when $| \ell_{j} - \ell'_{j} | \leq 1 $,
	\begin{equation} \label{eq:matrix:psi}
		\psi^{(j)}_{\ell_{j},\ell'_{j}}
		=
		\begin{cases}
			\frac{
				t^{(\subdiv')}_{(\ell_{j}+1)}
				-
				t^{(\subdiv')}_{(\ell_{j})}
			}{
				3
			}
			~ ~ \text{if} ~ ~
			\ell_{j} = \ell'_{j} = 1
			\\
			\frac{
				t^{(\subdiv')}_{(\ell_{j})}
				-
				t^{(\subdiv')}_{(\ell_{j}-1)}
			}{
				3
			}
			~ ~ \text{if} ~ ~
			\ell_{j} = \ell'_{j} = m_{j} +1
			\\
			\frac{
				t^{(\subdiv')}_{(\ell_{j}+1)}
				-
				t^{(\subdiv')}_{(\ell_{j}-1)}
			}{
				3
			}
			~ ~ \text{if} ~ ~
			\ell_{j} = \ell'_{j} \in \{2 , \ldots ,  m_{j} \}
			\\
			\frac{
				t^{(\subdiv')}_{(\ell_{j}+1)}
				-
				t^{(\subdiv')}_{(\ell_{j})}
			}{
				6
			}
			~ ~ \text{if} ~ ~
			\ell'_{j} = \ell_{j} +1 
			\\
			\frac{
				t^{(\subdiv')}_{(\ell_{j})}
				-
				t^{(\subdiv')}_{(\ell_{j}-1)}
			}{
				6
			}
			~ ~ \text{if} ~ ~
			\ell'_{j} = \ell_{j} -1 
			\\
		\end{cases}.
	\end{equation}
	Then, we have
	\begin{equation} \label{eq:quadratic:form:for:Ltwo:difference}
		\int_{[0,1]^d}
		\left(
		\widehat{Y}_{\activeset , \, \subdiv}(x)
		-
		\widehat{Y}_{\activeset  , \, \subdiv \cup_{1} t}(x)
		\right)^2
		dx
		=
		\sum_{\multi \in \mathcal{L}_{\subdiv'}, \multi' \in \mathcal{L}_{\subdiv'}}
		\beta_{\multi}
		\Psi_{\multi,\multi'}
		\beta_{\multi'}.
	\end{equation}
\end{proposition}

Consider that the coefficients in \eqref{eq:hatY:step:m:minus:one} and \eqref{eq:hatY:step:m} have been computed by optimization of a quadratic function with linear inequality constraints, see Section \ref{subsection:finite:dim:GP}. Then, from the above proposition, the $L^2$ difference between modes can be obtained by the explicit quadratic form \eqref{eq:quadratic:form:for:Ltwo:difference}. Furthermore, the matrix $\Psi$ defining this quadratic form is the tensor product of $d$ banded matrices $\psi^{(1)},\ldots,\psi^{(d)}$.
Hence, the computational cost is eventually linear in the number of multi-dimensional knots.

\subsection{$L^2$ difference when a new active variable is added}
\label{subsubsection:new:variable}

Consider the new variable $d+1$. Again, the choice of the new variable $d+1$ enables to simplify the exposition and the results of Section \ref{subsubsection:new:variable} can be immediately extended to a general new variable $i \in \{ d+1 , \ldots , D\}$.
The candidate mode function is
\begin{equation} \label{eq:hatY:step:m:new:variable}
	\widehat{Y}_{\activeset \cup \{d+1\}  , \, \subdiv + (d+1)}
	=
	\sum_{\multi \in \mathcal{L}_{\subdiv + (d+1)  }} 
	(\widehat{\alpha}_{\activeset \cup \{d+1\}, \subdiv + (d+1) })_{\multi}
	\phi_{\multi}^{(\subdiv + (d+1) )}.
\end{equation}

Then the next proposition provides a computationally efficient formula for the $L^2$ difference between modes.

\begin{proposition}[$L^2$ difference when a new active variable is added] \label{prop:efficient:computation:new:variable}
	Let $\multi  \in \mathcal{L}_{\subdiv + (d+1)}$ and write $\multi$ of the form $(\widetilde{\multi},\ell_{d+1})$ with $\widetilde{\multi} \in \mathcal{L}_{\subdiv}$ and $\ell_{d+1} \in \{1 , 2 \}$. Then let
	\[
	\beta_{\multi} = 
	(\widehat{\alpha}_{\activeset,\subdiv})_{\widetilde{\multi}}
	-
	(\widehat{\alpha}_{\activeset \cup \{d+1\},\subdiv+(d+1)})_{\multi}
	.
	\]
	For $\multi, \multi'  \in \mathcal{L}_{\subdiv + (d+1)}$, define $\Psi_{\multi,\multi'}$ as in Proposition \ref{prop:efficient:computation:new:knot}, but with $\subdiv'$ there replaced by $\subdiv + (d+1)$. Then we have
	\begin{equation} \label{eq:quadratic:form:for:Ltwo:difference:new:variable}
		\int_{[0,1]^{d+1}}
		\left(
		\widehat{Y}_{\activeset , \, \subdiv}(x)
		-
		\widehat{Y}_{\activeset \cup \{d+1\} , \, \subdiv + (d+1)}(x)
		\right)^2
		dx
		=
		\sum_{\multi \in \mathcal{L}_{\subdiv+(d+1)}, \multi' \in \mathcal{L}_{\subdiv+(d+1)}}
		\beta_{\multi}
		\Psi_{\multi,\multi'}
		\beta_{\multi'}.
	\end{equation}
\end{proposition}

The discussion of Proposition \ref{prop:efficient:computation:new:variable} is the same as for Proposition \ref{prop:efficient:computation:new:knot}.

\section{Technical conditions for Theorem \ref{theorem:convergence}} \label{appendix:section:adaptation:condition}

As for Theorem \ref{theorem:extension:convergence:spline}, for each $\activeset \subseteq \{1 , \ldots , D\}$, we consider a set of functions in $\mathcal{C}( [0,1]^{|\activeset|} , \mathbb{R} )$ satisfying strict inequalities: $\strictineqset_{\activeset} \subseteq \mathcal{C}_{\activeset}$. We recall that the MaxMod algorithm is initialized with $\activeset_0 \subseteq \{ 1, \ldots , D \}$ and $\subdiv^{(0)} \in \subdivset$.

Conditions \ref{cond:non:empty}, \ref{cond:spans:Rn} and \ref{cond:interior:non:empty} are replaced by the following conditions.

	\begin{condition}[extension of Condition \ref{cond:non:empty}] \label{cond:non:empty:D}
		Consider $\activeset \supseteq \activeset_0 $ and $\subdiv \in \subdivset_{\activeset}$ such that, for $j \in \activeset_0$, we have $\subdiv_{j} \supseteq \subdiv^{(0)}_{j}  $.
		Then the set 
		\[
		\{ \alpha \in A_{\subdiv};
		Y_{\subdiv,\alpha} \in \interpset_{X_{\activeset},y^{(n)}}
		\cap \strictineqset_{\activeset}
		\}
		\]
		is non-empty.
	\end{condition}
	Condition \ref{cond:non:empty:D} means that whenever the set of active variables has been increased compared to $\activeset_0$ and knots have been inserted to the variables in $\activeset_0$ or to the new active variables, then it is possible to find an interpolating function that satisfies the strict inequalities.

		\begin{condition}[extension of Condition \ref{cond:spans:Rn}] \label{cond:spans:Rn:D}
			We have
			\[
			\left\{
			\left(
			Y_{\subdiv^{(0)} , \alpha} \left(x_{\activeset_0}^{(i)}\right)
			\right)_{i=1,\ldots,n}
			;
			\alpha \in \subdiv^{(0)} 
			\right\}
			=
			\mathbb{R}^n.
			\]
		\end{condition}

		For $\activeset \subseteq \{1 , \ldots , D\}$, we let $\Hilb_{\activeset}$ be the RKHS of $k_{\activeset}$ with norm $||\cdot ||_{\Hilb_{\activeset}}$ (recall that $\Hilb_{\activeset}$ is a set of functions from $[0,1]^{|\activeset|}$ to $\mathbb{R}$).
		
			\begin{condition}[extension of Condition \ref{cond:interior:non:empty}] \label{cond:interior:non:empty:D}
				For all $\activeset \subseteq \{1 , \ldots , D\}$, for all $h \in \strictineqset_{\activeset}$, for all $r \in \mathbb{N}$, $U = (u_1,\ldots,u_r) \in ([0,1]^{|\mathcal{I}|})^r$, 
				with $u_1,\ldots,u_r$ two-by-two distinct, 
				letting $v^{(r)} = (h(u_1), \ldots, h(u_r)) \in \R^r$, 
				the set $\mathrm{int}_{||.||_{\Hilb_{\activeset}}}(\Hilb_{\activeset} \cap \ineqset_{\activeset}) \cap \interpset_{U,v^{(r)}}$ is non-empty.
			\end{condition}

			Consider the case where $\ineqset_D$ is given by \eqref{eq:boundedness:D}, \eqref{eq:monotonicity:D} and \eqref{eq:convexity:D} (boundedness, monotonicity and convexity). Then the same discussions of Conditions \ref{cond:non:empty} and \ref{cond:interior:non:empty} of Theorem \ref{theorem:extension:convergence:spline} apply to Conditions \ref{cond:non:empty:D} and \ref{cond:interior:non:empty:D} here. 
			In particular, when $\ineqset_D$ is given by \eqref{eq:boundedness:D}, we choose $\strictineqset_{\activeset}$ as the set of functions in $\mathcal{C}([0,1]^{|\activeset|},\mathbb{R})$ that are strictly between $a$ and $b$.  When $\ineqset_D$ is given by \eqref{eq:monotonicity:D}, we choose $\strictineqset_{\activeset}$ as the set of functions in $\mathcal{C}([0,1]^{|\activeset|},\mathbb{R})$ that are strictly increasing.  When $\ineqset_D$ is given by \eqref{eq:convexity:D}, we choose $\strictineqset_{\activeset}$ as the set of functions in $\mathcal{C}([0,1]^{|\activeset|},\mathbb{R})$ which one-dimensional cuts are strictly convex.
			With these choices, Conditions \ref{cond:non:empty:D} and \ref{cond:interior:non:empty:D} are mild and mean that there are sufficiently many knots and sufficiently many active variables at the beginning of the MaxMod algorithm and that the Hilbert space $\Hilb_{\activeset}$ is sufficiently rich for $\activeset \supseteq \activeset_0$. Finally, Lemma \ref{lem:non-empty:contains} can be straightforwardly adapted.
			
			Condition \ref{cond:piSC:subset:C} is replaced by the following one.
				\begin{condition}[extension of Condition \ref{cond:piSC:subset:C}] \label{cond:piSC:subset:C:D}
					For all $\activeset \subseteq \{1 , \ldots , D\}$,  
					for all subdivision $\subdiv \in \subdivset_{\activeset}$, $\pi_{\subdiv} (\ineqset_{\activeset}) \subseteq \ineqset_{\activeset} $.
				\end{condition}
				As for Theorem \ref{theorem:extension:convergence:spline}, Condition \ref{cond:piSC:subset:C:D} can be shown to hold for boundedness, monotonicity and input-wise convexity constraints.
				Finally, Condition \ref{cond:PfC:subset:C} is replaced by the following one.
					\begin{condition}[extension of Condition \ref{cond:PfC:subset:C}] \label{cond:PfC:subset:C:D}
						For any $\activeset = (a_1,\ldots,a_d) \subseteq \{1 , \ldots , D\}$, for a closed set $\clos = \clos_{a_1} \times \dots \times \clos_{a_d} \subseteq [0,1]^{d}$ with $0,1 \in \clos_{a_j}$ for $j=1,\ldots,d$, let us define $P_{\clos \to \domain}$ as in Proposition \ref{prop:multiaffine:extension}.
						Then, for $f \in \ineqset_{\activeset}$, with $f_{|\clos}$ the restriction of $f$ to $\clos$, we have $P_{\clos \to \domain} (f_{|F}) \in \ineqset_{\activeset}$.  
					\end{condition}
					Remark that Lemma \ref{lem:cond:PfC:subset:C} can be extended straightforwardly.

					\section{Proofs} \label{appendix:section:proofs}
					
					\begin{proof}[{\bf Proof of Lemma \ref{lem:sup:finite}}]
						Since the reward $R_{\activeset_m, \subdiv^{(m)}}(i, t)$ is smaller than $\max(\Delta,\Delta')$,
						it is sufficient to show that, for $i \in \{1 , \ldots , D\}$,
						\begin{equation} \label{eq:max:of:sup}
							\underset{
								\substack{
									t \in [0,1], \\
									\text{s.t. } d(t , \subdiv^{(m)}_{ i })\, \geq \, b_m \\
									\text{if $i \in \activeset_m$}
								}
							}{\sup}
							I_{\activeset_m, S^{(m)}}(i ,t) 
							< \infty.
						\end{equation}
						If $i \not \in \activeset_m$,  $I_{\activeset_m, S^{(m)}}(i ,t)$ in \eqref{eq:max:of:sup} does not depend on $t$ (since the knots for the variable $i$, after this variable has been added, are $\{-1,0,1,2\}$ independently on $t$). Thus the $\sup$ in \eqref{eq:max:of:sup} is finite.
						
						Consider now $i \in \activeset_m$.
						Let $E_{i,m}$ be the set of $t \in [0,1]$ such that $d \big(t, \subdiv_{i}^{(m)} \big) \geq b_m$.
						By continuity of $d \big(., \subdiv_{i}^{(m)} \big)$, notice that $E_{i,m}$ is compact.
						
						Consider the function $\widehat{Y}_{\activeset_m, \, S^{(m)}}$. 
						We can write it in the thinner space obtained by inserting the knot $t$ at coordinate $i$, in the form 
						$Y_{S^{(m)} \addknot{i} t , \, \alpha_{m,t}}$ 
						for some $\alpha_{m,t} \in A_{S^{(m)} \addknot{i} t}$. 
						
						The coefficients $\alpha_{m,t}$ are equal to the values of the 
						function $\widehat{Y}_{\activeset_m,S^{(m)}}$ at the (multidimensional) 
						knots in $S^{(m)} \addknot{i} t$, as remarked just after \eqref{eq:Y:S:alpha}.
						Since $\widehat{Y}_{\activeset_m, S^{(m)}}$ does not depend on $t$ and is continuous on 
						$[0, 1]^{\vert \activeset_m \vert}$, 
						it follows that $\alpha_{m,t}$ is bounded with respect to $t \in E_{i,m}$.\\
						Furthermore, let $\lambda_{\max}(M)$ and $\lambda_{\min}(M)$ be the largest and smallest eigenvalues of a matrix $M$.
						By definition of $\widehat{\alpha}_{\activeset_m,S^{(m)} \addknot{i} t }$, 
						and since $\widehat{Y}_{\activeset_m, S^{(m)}}$ belongs to  $\interpset_{X^{(\activeset_m)}, y^{(n)}} \cap \, \ineqset_{\activeset_m}$, we have:
						\begin{align*}
							& \widehat{\alpha}_{\activeset_m,S^{(m)} \addknot{i} t }^\top
							k_{\activeset_m}( S^{(m)} \addknot{i} t , S^{(m)} \addknot{i} t  )^{-1}
							\widehat{\alpha}_{\activeset_m,S^{(m)} \addknot{i} t }
							\\
							& \leq 
							\alpha_{m,t}^\top
							k_{\activeset_m}( S^{(m)} \addknot{i} t , S^{(m)} \addknot{i} t  )^{-1}
							\alpha_{m,t} \\
							& \leq 
							\frac{||\alpha_{m,t}||^2}{\lambda_{\min}
								\left(
								k_{\activeset_m}( S^{(m)} \addknot{i} t , S^{(m)} \addknot{i} t  )
								\right)}.
						\end{align*}
						Furthermore, 
						\begin{align*}
							& \widehat{\alpha}_{\activeset_m,S^{(m)} \addknot{i} t }^\top
							k_{\activeset_m}( S^{(m)} \addknot{i} t , S^{(m)} \addknot{i} t  )^{-1}
							\widehat{\alpha}_{\activeset_m,S^{(m)} \addknot{i} t }
							\\
							& \geq 
							\frac{||\widehat{\alpha}_{\activeset_m,S^{(m)} \addknot{i} t }||^2}
							{\lambda_{\max}
								\left( k_{\activeset_m}( S^{(m)} \addknot{i} t , S^{(m)} \addknot{i} t  ) \right)
							}.
						\end{align*} 
						Hence, we have that
						\[
						\sup_{t \in E_{i,m}}
						|| \widehat{\alpha}_{\activeset_m,S^{(m)} \addknot{i} t }  ||^2
						\leq 
						\sup_{t \in E_{i,m}}
						\frac{ \lambda_{\max}
							\left(
							k_{\activeset_m}( S^{(m)} \addknot{i} t , S^{(m)} \addknot{i} t  )
							\right) }{\lambda_{\min}
							\left(
							k_{\activeset_m}( S^{(m)} \addknot{i} t , S^{(m)} \addknot{i} t  )
							\right)}
						\sup_{t \in E_{i,m}}
						||\alpha_{m,t}||^2.
						\]
						By assumption, $k_{\activeset_m}$ is continuous and the matrix $k_{\activeset_m}( S^{(m)} \addknot{i} t , S^{(m)} \addknot{i} t  )$ is invertible for all $t$
						in the compact set $E_{i,m}$. Thus, the ratio of eigenvalues above is bounded. 
						Hence $|| \widehat{\alpha}_{\activeset_m,S^{(m)} \addknot{i} t }  ||$ is bounded with respect to $t \in E_{i,m}$ and thus also the supremum of the function $\widehat{Y}_{\activeset_m,S^{(m)} \addknot{i} t}$ is bounded with respect to $t \in E_{i,m}$. Hence the $\sup$ in \eqref{eq:max:of:sup} is indeed finite, which concludes the proof.
					\end{proof}
					
					\begin{proof}[{\bf Proof of Proposition \ref{prop:efficient:computation:new:knot}}]
						
						The current mode function 
						\begin{equation} \label{eq:hatY:step:m:minus:one:bis}
							\widehat{Y}_{\activeset , \, \subdiv}
							=
							\sum_{\widetilde{\multi} \in \mathcal{L}_{\subdiv}} 
							(\widehat{\alpha}_{\activeset,\subdiv})_{\widetilde{\multi}}
							\phi_{\widetilde{\multi}}^{(\subdiv)}
						\end{equation}
						can be expressed of the form
						\[
						\widehat{Y}_{\activeset , \, \subdiv}
						=
						\sum_{\multi \in \mathcal{L}_{\subdiv'}} 
						\gamma_{\multi}
						\phi_{\multi}^{(\subdiv')}.
						\]
						Let us express $\gamma_{\multi}$. Let $\multi = (\ell_{1},\ldots,\ell_{d}) \in \mathcal{L}_{\subdiv'}$. First, if  $\ell_1 \in \{1 , \ldots , \nu \}$, we have
						\[
						\gamma_{\multi}
						=
						\widehat{Y}_{\activeset , \, \subdiv}
						(
						(
						t^{(\subdiv'_{1})}_{(\ell_{1})}
						,
						\ldots
						,
						t^{(\subdiv'_{d})}_{(\ell_{d})}
						) 
						)
						\]
						and because $t^{(\subdiv'_1)}_{(\ell_1)} = t^{(\subdiv_1)}_{(\ell_1)}$, we obtain from \eqref{eq:hatY:step:m:minus:one:bis}, $\widehat{Y}_{\activeset , \, \subdiv}
						(
						(
						t^{(\subdiv'_{1})}_{(\ell_{1})}
						,
						\ldots
						,
						t^{(\subdiv'_{d})}_{(\ell_{d})}
						) 
						)
						=
						(\widehat{\alpha}_{\activeset,\subdiv})_{\multi}
						$
						and thus
						\[
						\gamma_{\multi}
						=
						(\widehat{\alpha}_{\activeset,\subdiv})_{\multi}.
						\]
						
						Second, if  $\ell_1 \in \{\nu+2 , \ldots , m_1 + 1 \}$, we have
						\[
						\gamma_{\multi}
						=
						\widehat{Y}_{\activeset , \, \subdiv}
						(
						(
						t^{(\subdiv'_{1})}_{(\ell_{1})}
						,
						\ldots
						,
						t^{(\subdiv'_{d})}_{(\ell_{d})}
						) 
						)
						\]
						and because $t^{(\subdiv'_1)}_{(\ell_1)}  = t^{(\subdiv_1)}_{(\ell_1-1)}$, we obtain from \eqref{eq:hatY:step:m:minus:one:bis},
						\[
						\widehat{Y}_{\activeset , \, \subdiv}
						(
						(
						t^{(\subdiv'_{1})}_{(\ell_{1})}
						,
						\ldots
						,
						t^{(\subdiv'_{d})}_{(\ell_{d})}
						) 
						)
						=
						(\widehat{\alpha}_{\activeset,\subdiv})_{(\ell_1-1,\ell_2,\ldots,\ell_{d})}.
						\]
						Hence 
						\[
						\gamma_{\multi} = (\widehat{\alpha}_{\activeset,\subdiv})_{(\ell_1-1,\ell_2,\ldots,\ell_{d})}.
						\]
						Finally, if  $\ell_1 = \nu+1$, we have
						\begin{align*}
							\gamma_{\multi}
							&
							=
							\widehat{Y}_{\activeset , \, \subdiv}
							(
							(
							t^{(\subdiv'_{1})}_{(\nu+1)}
							,
							t^{(\subdiv'_{2})}_{(\ell_{2})}
							,
							\ldots
							,
							t^{(\subdiv'_{d})}_{(\ell_{d})}
							) 
							)
							\\
							& =
							(\widehat{\alpha}_{\activeset,\subdiv})_{(\nu,\ell_2,\ldots,\ell_{d})}
							\phi_{t^{(\subdiv_1)}_{(\nu-1)},t^{(\subdiv_1)}_{(\nu)},t^{(\subdiv_1)}_{(\nu+1)}} ( t^{(\subdiv'_1)}_{(\nu+1)} )
							+
							(\widehat{\alpha}_{\activeset,\subdiv})_{(\nu+1,\ell_2,\ldots,\ell_{d})}
							\phi_{t^{(\subdiv_1)}_{(\nu)},t^{(\subdiv_1)}_{(\nu+1)},t^{(\subdiv_1)}_{(\nu+2)}} ( t^{(\subdiv'_1)}_{(\nu+1)} ),
						\end{align*}
						from the definition of the multidimensional hat basis functions and the position of $(
						t^{(\subdiv'_{1})}_{(\nu+1)}
						,
						t^{(\subdiv'_{2})}_{(\ell_2)}
						,
						\linebreak[1]
						\ldots
						,
						t^{(\subdiv'_{d})}_{(\ell_d)}
						) $
						relatively to the knots in the subdivision $\subdiv$. 
						Hence we have, since  $t^{(\subdiv'_1)}_{(\nu+1)} = t$,
						\begin{align*}
							\gamma_{\multi}
							=
							(\widehat{\alpha}_{\activeset,\subdiv})_{(\nu,\ell_2,\ldots,\ell_{d})}
							\frac{ t^{(\subdiv_1)}_{(\nu+1)} - t }{t^{(\subdiv_1)}_{(\nu+1)} - t^{(\subdiv_1)}_{(\nu)}}
							+
							(\widehat{\alpha}_{\activeset,\subdiv})_{(\nu+1,\ell_2,\ldots,\ell_{d})}
							\frac{ t - t^{(\subdiv_1)}_{(\nu)}  }{t^{(\subdiv_1)}_{(\nu+1)} - t^{(\subdiv_1)}_{(\nu)}}.
						\end{align*}
						
						Hence, we have shown that
						\[
						\widehat{Y}_{\activeset , \, \subdiv}
						-
						\widehat{Y}_{\activeset , \, \subdiv'}
						=
						\sum_{\multi \in \mathcal{L}_{\subdiv'}} 
						\beta_{\multi}
						\phi_{\multi}^{(\subdiv')},
						\]
						with $\beta_{\multi}$ as given in the proposition. To conclude the proof, by bilinearity of the square $L^2$ distance, it remains to prove that for $\multi = (\ell_{1},\ldots,\ell_{d}) \in \mathcal{L}_{\subdiv'}$, $\multi' = (\ell'_{1},\ldots,\ell'_{d}) \in \mathcal{L}_{\subdiv'}$,
						\begin{equation} \label{eq:int:prod:phi}
							\int_{[0,1]^d}
							\phi_{\multi}^{(\subdiv')}
							(x)
							\phi_{\multi'}^{(\subdiv')}
							(x)
							dx
							=
							\Psi_{\multi,\multi'},
						\end{equation}
						with $\Psi_{\multi,\multi'}$  as given in the proposition. If there is $j \in \{ 1,\ldots,d \}$ such that $
						| \ell_j - \ell'_j | > 1$, then the supports of $\phi_{\multi}^{(\subdiv')}$
						and
						$\phi_{\multi'}^{(\subdiv')}$ are disjoint and thus \eqref{eq:int:prod:phi} is zero. Consider now that for $j \in \{ 1,\ldots,d \}$, $
						| \ell_j - \ell'_j | \leq 1$. Then \eqref{eq:int:prod:phi} is equal to 
						\[
						\prod_{j=1,\ldots,d}
						\int_{0}^1
						\phi_{t^{(\subdiv'_j)}_{(\ell_j-1)},t^{(\subdiv'_j)}_{(\ell_j)},t^{(\subdiv'_j)}_{(\ell_j+1)}}(x)
						\phi_{t^{(\subdiv'_j)}_{(\ell'_j-1)},t^{(\subdiv'_j)}_{(\ell'_j)},t^{(\subdiv'_j)}_{(\ell'_j+1)}}(x)
						dx.
						\]
						Let now $j \in \{ 1 , \ldots , d\}$ and write $t_{\ell}$ as a short-hand for $ t^{(\subdiv'_j)}_{(\ell)}$, for $\ell \in \{0 , \ldots , m_{\subdiv'_j} +1 \}$. Write also $\phi_{\ell}$ as a short-hand for $\phi_{t^{(\subdiv'_j)}_{(\ell-1)},t^{(\subdiv'_j)}_{(\ell)},t^{(\subdiv'_j)}_{(\ell+1)}}$.

						Consider $\ell , \ell' \in \{1 , \ldots , m_j +1\}$. If $\ell = \ell' = 1$, we have
						\begin{align*}
							\int_{0}^1
							\phi_{\ell}(x)
							\phi_{\ell'}(x)
							=
							\int_{t_1}^{t_2}
							\left(
							\frac{ t_2 - x }{ t_2 - t_1 }
							\right)^2
							dx
							=
							\frac{t_2 - t_1}{3}.
						\end{align*}
						If $\ell = \ell' = m_j+1$, we have
						\begin{align*}
							\int_{0}^1
							\phi_{\ell}(x)
							\phi_{\ell'}(x)
							=
							\int_{t_{m_j}}^{t_{m_j+1}}
							\left(
							\frac{ x - t_{m_j} }{ t_{m_j+1} - t_{m_j} }
							\right)^2
							dx
							=
							\frac{t_{m_j+1} - t_{m_j}}{3}.
						\end{align*}
						
						If $\ell = \ell' \in \{ 2 , \ldots ,  m_j \}$, we have
						\begin{align*}
							\int_{0}^1
							\phi_{\ell}(x)
							\phi_{\ell'}(x)
							=
							\int_{t_{\ell-1}}^{t_{\ell}}
							\left(
							\frac{ x -  t_{\ell-1}  }{ t_{\ell} - t_{\ell-1} }
							\right)^2
							dx
							+
							\int_{t_{\ell}}^{t_{\ell+1}}
							\left(
							\frac{ t_{\ell+1} -  x   }{ t_{\ell+1} - t_{\ell} }
							\right)^2
							dx
							=
							\frac{t_{\ell+1} - t_{\ell-1}}{3}.
						\end{align*}
						If $\ell' = \ell+1 \in \{ 2 , \ldots ,  m_j +1\}$, we have
						\begin{align*}
							\int_{0}^1
							\phi_{\ell}(x)
							\phi_{\ell'}(x)
							=
							\int_{t_{\ell}}^{t_{\ell+1}}
							\left(
							\frac{  t_{\ell+1} - x }{ t_{\ell+1} - t_{\ell} }
							\right)
							\left(
							\frac{  x - t_{\ell} }{ t_{\ell+1} - t_{\ell} }
							\right)
							dx
							=
							\frac{t_{\ell+1} - t_{\ell}}{6}.
						\end{align*}
						Finally, if $\ell' = \ell - 1 \in \{ 1 , \ldots ,  m_j \}$, we have
						\begin{align*}
							\int_{0}^1
							\phi_{\ell}(x)
							\phi_{\ell'}(x)
							=
							\int_{t_{\ell-1}}^{t_{\ell}}
							\left(
							\frac{  t_{\ell} - x }{ t_{\ell} - t_{\ell-1} }
							\right)
							\left(
							\frac{  x - t_{\ell-1} }{ t_{\ell} - t_{\ell-1} }
							\right)
							dx
							=
							\frac{t_{\ell} - t_{\ell-1}}{6}.
						\end{align*}
						Hence, we have shown that for  $j \in \{ 1 , \ldots , d\}$ and $\ell_j, \ell'_j \in \{ 1 , \ldots , m_j +1\}$ with $|\ell_j- \ell'_j| \leq 1$, 
						\[
						\int_{0}^1
						\phi_{t^{(\subdiv'_j)}_{(\ell_j-1)},t^{(\subdiv'_j)}_{(\ell_j)},t^{(\subdiv'_j)}_{(\ell_j+1)}}(x)
						\phi_{t^{(\subdiv'_j)}_{(\ell'_j-1)},t^{(\subdiv'_j)}_{(\ell'_j)},t^{(\subdiv'_j)}_{(\ell'_j+1)}}(x)
						dx
						=
						\psi^{(j)}_{\ell_j,\ell'_j}
						\]
						with the notation of the proposition. This concludes the proof.
					\end{proof}
					
					\begin{proof}[{\bf Proof of Proposition \ref{prop:efficient:computation:new:variable}}]
						
						The current mode function 
						\begin{equation*} 
							\widehat{Y}_{\activeset , \, \subdiv}
							=
							\sum_{\widetilde{\multi} \in \mathcal{L}_{\subdiv}} 
							(\widehat{\alpha}_{\activeset,\subdiv})_{\widetilde{\multi}}
							\phi_{\widetilde{\multi}}^{(\subdiv)}
						\end{equation*}
						can be expressed of the form
						\[
						\widehat{Y}_{\activeset , \, \subdiv}
						=
						\sum_{\widetilde{\multi} \in \mathcal{L}_{\subdiv}} 
						\sum_{\ell_{d+1} =1}^2
						(\widehat{\alpha}_{\activeset,\subdiv})_{\widetilde{\multi}}
						\phi_{\widetilde{\multi}}^{(\subdiv)}
						\phi_{
							t^{(\subdiv^{0})}_{(\ell_{d+1}-1)},
							t^{(\subdiv^{0})}_{(\ell_{d+1})},
							t^{(\subdiv^{0})}_{(\ell_{d+1}+1)} }
						=
						\sum_{\multi \in \mathcal{L}_{\subdiv'}} 
						\gamma_{\multi}
						\phi_{\multi}^{(\subdiv+(d+1))}
						\]
						with $\gamma_{\multi} =  (\widehat{\alpha}_{\activeset,\subdiv})_{\widetilde{\multi}}$ when $\multi = (\widetilde{\multi} , \ell_{d+1})$. Hence we obtain 
						\[
						\widehat{Y}_{\activeset , \, \subdiv}
						-
						\widehat{Y}_{\activeset \cup \{ d+1 \} , \, \subdiv + (d+1)}
						=
						\sum_{\multi \in \mathcal{L}_{\subdiv + (d+1)}} 
						\beta_{\multi}
						\phi_{\multi}^{(\subdiv+(d+1))},
						\]
						with $\beta_{\multi}$ as defined in the proposition. The rest of the proof is then identical to the proof of Proposition \ref{prop:efficient:computation:new:knot}.
					\end{proof}

					\begin{proof}[{\bf Proof of Proposition \ref{prop:multiaffine:extension}}]
						Let us first prove that
						\[
						L_1 \dots L_d f(t)
						= \sum_{\epsilon_1, \dots, \epsilon_d \in \{-, +\}} 
						\Bigg(\prod_{j = 1}^d \omega_{\epsilon_j}(t_j) \Bigg)
						f(t_1^{\epsilon_1}, \dots, t_d^{\epsilon_d}).
						\]
						In dimension $1$, notice that the affine extension can be rewritten as
						\begin{equation} \label{eq:linExtGen}
							L_B f(t) = \omega_-(t) f(t^-) + \omega_+(t) f(t^+)
							= \sum_{\epsilon \in \{-, + \}} \omega_\epsilon(t) f(t^\epsilon),
						\end{equation}
						with the convention chosen for $\omega_-(t), \omega_+(t)$ for $t \in B$. 
						Indeed, in that case $t^- = t^+ = t$, and thus $L_B f (t) = \frac{1}{2}(f(t) + f(t)) =  f(t)$.
						Then, let us show by induction on $i = d, d-1, \dots, 1$ the property
						\begin{align} \label{eq:multiExtFormulaAux}
							& (\mathcal{P}_{i}): \quad  \forall t \in \prodF{i-1} \times [0,1]^{d-i+1},
							\notag
							\\
							&  \quad L_{i} \dots L_d f (t)
							= \sum_{\epsilon_{i}, \dots, \epsilon_d \in \{-, +\}} 
							\Bigg(\prod_{j = i}^d \omega_{\epsilon_j} (t_j) \Bigg) f(t_1, \dots, t_{i-1}, t_i^{\epsilon_i}, \dots, t_d^{\epsilon_d}).
						\end{align} 
						First $\mathcal{P}_d$ is true, by the expression (\ref{eq:linExtGen}) for $L_{\clos_d}$.
						Now, assume that $\mathcal{P}_{i+1}$ is true (for $i \in \{1, \dots, d-1\}$). 
						Then, we have, for $t \in \prodF{i-1} \times [0,1]^{d-i+1}$:
						\begin{align*}
							& L_i \dots L_d f(t)  = L_{\clos_i} (L_{i+1} \dots L_d f(., t_{\sim i}))(t_i) \\
							& =  \sum_{\epsilon_i \in \{-, +\}} \omega_{\epsilon_i}(t_i) (L_{i+1} \dots L_d f)(t_i^{\epsilon_i}, t_{\sim i}) \\
							&=  \sum_{\epsilon_i \in \{-, +\}} \omega_{\epsilon_i}(t_i) 
							\Bigg( \sum_{\epsilon_{i+1}, \dots, \epsilon_d \in \{-, +\}} 
							\Bigg(\prod_{j = i+1}^d \omega_{\epsilon_j}(t_j)\Bigg)
							f(t_1, \dots, t_{i-1},  t_i^{\epsilon_i}, t_{i+1}^{\epsilon_{i+1}}, \dots, t_d^{\epsilon_d}) \Bigg) 
						\end{align*}
						which gives $\mathcal{P}_i$. Finally $\mathcal{P}_1$ gives \eqref{eq:multiExtFormula}. 
						Furthermore, notice that:
						$$ \sum_{\epsilon_1, \dots, \epsilon_d \in \{-, +\}} 
						\Bigg(\prod_{j = 1}^d \omega_{\epsilon_j}(t_j) \Bigg)
						= \prod_{j=1}^d (\omega_-(t_j) + \omega_+(t_j)) = 1.$$
						
						
						Let us now justify the existence and unicity of $\extfun$.\\
						Firstly, if $g$ exists, then necessarily $g = L_1 \dots L_d f$.
						Indeed, by assumption, when $t_{\sim 1} = (t_2, \dots, t_d)$ is fixed in $[0, 1]^{d-1}$, the univariate function 
						$$ g(., t_{\sim 1}): u_1 \mapsto g(u_1, t_{\sim 1})$$
						is continuous and affine on all intervals of $[0, 1] \setminus \clos_1$.
						By property of the affine extension, it is equal to the affine extension of its restriction to $\clos_1$:
						$$ g(., t_{\sim 1}) = L_{\clos_1}(g(., t_{\sim 1})_{\vert \clos_1}) = L_{\clos_1} \left( g_{\vert \clos_1 \times [0, 1]^{d-1}}(., t_{\sim 1}) \right).$$
						This shows that
						$$ g = L_1 \, g_{\vert \clos_1 \times [0, 1]^{d-1}}.$$
						By the same reasoning,
						$$ g_{\vert \clos_1 \times [0, 1]^{d-1}} = L_2 \, g_{\vert \clos_1 \times \clos_2 \times [0, 1]^{d-2}}, $$
						and by an immediate induction, 
						$$g = L_1 \dots L_d \, g_{\vert \clos} = L_1 \dots L_d f.$$
						
						\noindent Now, let us check that the function $g = L_1 \dots L_d f$ verifies the conditions of the proposition.
						\begin{itemize}
							\item It is equal to $f$ on $\clos$, since each $L_i$ leaves the values of its input function unchanged on $\clos$.
							\item Let check that all 1-dimensional cuts of $g$ are affine on all intervals of the complements of the $\clos_i$'s.
							Indeed, consider the explicit formula (\ref{eq:multiExtFormula}). 
							Let fix $i \in \{1, \dots, d\}$ and consider an interval $I_i$ included in $[0, 1] \setminus \clos_i$.
							Without loss of generality, we assume that $I_i  = [a_i, b_i]$ is closed, with $a_i < b_i$.
							Then, for all $t_i \in I_i$, we have $a_i^- = t_i^- < t_i^+ = b_i^+$. 
							Thus, $t_i^-, t_i^+$ do not depend on $t_i$.
							Consequently, $\omega_+(t_i) = \frac{t_i - a_i^-}{b_i^+ - a_i^-}$  
							and $\omega_-(t_i) = 1 - \omega_+(t_i)$ depend linearly on $t_i$. 
							Finally, by (\ref{eq:multiExtFormula}), for all $t_{\sim i} \in [0, 1]^{d-1}$, 
							the cut function $t_i \mapsto L_1 \dots L_d f (t)$ is affine on $I_i$.
							\item To prove continuity, by composition, it is sufficient to prove that for $i \in \{1 , \ldots , d\}$, when $h$ belongs to $\mathcal{C}(\prodF{i}\times [0,1]^{d-i},\mathbb{R})$ then $L_i \, h$ belongs to $\mathcal{C}(\prodF{(i-1)}\times [0,1]^{d-i+1},\mathbb{R})$. Let us thus consider $i$, and $h$ as just described.
							Let $t  = (t_1,\ldots,t_d) \in \prodF{(i-1)}\times [0,1]^{d-i+1}$ and consider a sequence $t_n  = (t_{n,1},\ldots,t_{n,d}) \in \prodF{(i-1)}\times [0,1]^{d-i+1}$ converging to $t$. We will show that $L_i \, h (t_n)$ converges to $L_i \, h (t)$.
							
							Up to extracting subsequences, it suffices to consider the cases (1) $t_{n,i} \in \clos_i$, (2) $t_{n,i} \not \in \clos_i$, $t_{n,i} \leq t_i$, $t_{n,i}$ increasing and (3) $t_{n,i} \not \in \clos_i$, $t_{n,i} \geq t_i$, $t_{n,i}$ decreasing.\\
							In case (1), since $\clos_i$ is closed we have $t_i \in \clos_i$. Thus, as $h$ is continuous, $L_i \, h (t_n) = h(t_n)$ converges to $h(t) = L_i \, h (t)$. \\
							Consider now the case (2). Observe that we have 
							\begin{equation} \label{eq:Lihtn}
								L_i h(t_n) =
								[1 - \omega_+(t_{n,i})] h( t_{n,i}^-, t_{n, \sim i}) + \omega_+(t_{n,i}) h(t_{n,i}^+, t_{n, \sim i}).  
							\end{equation}
							Furthermore, as $t_{n,i}^-$ is increasing and bounded above by 
							$t_i$, it converges to a limit $t_{\infty,i}^- \leq t_i$. 
							Consider the case (2a) where $t_i \in \clos_i$. 
							Then, by definition of $t_{n,i}^+$, we have $t_{n,i} \leq t_{n,i}^+ \leq t_i$. Hence $t_{n,i}^+$ converges to $t_i$. If, first, $t_{\infty,i}^- = t_i$, then from \eqref{eq:Lihtn}, $ L_i h(t_n)$ is a convex combinations of two values of $h$ at two inputs that converge to $t$ so $ L_i h(t_n)$ converges to $h(t) = L_i h(t)$:
							\begin{align*}
								& \vert L_i h(t_n) - h(t) \vert \\
								&  =   \left \vert [1 - \omega_+(t_{n,i})] ( h( t_{n,i}^-, t_{n, \sim i})- h(t) ) 
								+ \omega_+(t_{n,i}) (h(t_{n,i}^+, t_{n, \sim i}) - h(t) ) \right \vert \\
								& \leq  \max(\vert  h( t_{n,i}^-, t_{n, \sim i})- h(t) \vert, \vert h(t_{n,i}^+, t_{n, \sim i}) - h(t) \vert) 
								\underset{n \to +\infty}{\longrightarrow} 0.
							\end{align*}
							If, second, $t_{\infty,i}^- < t_i$, then 
							$\omega_+(t_{n,i}) \to 1$ and by \eqref{eq:Lihtn}, $ L_i h(t_n)$ converges to $h(t) = L_i h(t)$.\\ 
							Consider now the case (2b) where $t_i \not \in \clos_i$. 
							Then, since $[0,1] \backslash \clos_i$ is open in $[0,1]$, for $n$ large enough we have $t_{n,i}^- = t_{i}^-$ and 
							$t_{n,i}^+ = t_{i}^+$. 
							Thus, $\omega_+(t_{n,i}) \to \omega_+(t_i)$, and by \eqref{eq:Lihtn}, $L_i h(t_n) \to L_i h(t)$.
							This concludes the case (2). The case (3) is treated similarly.
						\end{itemize}
					\end{proof}

					\begin{proof}[{\bf Proof of proposition \ref{prop:extension:properties}}]
						{~}
						\begin{enumerate}
							\item First, the linearity of $\extfun$ comes from the linearity of $g \mapsto L_B g$,
							and by composition of affine maps.
							
							Now, let us prove that $\extfun$ is $1$-Lipschitz. By linearity, it is sufficient to show that for $f \in \mathcal{C}(\clos , \mathbb{R}),$
							\begin{equation} \label{eq:to:show:extension:lipschitz}
								\sup_{t \in \domain}
								| \extfun (f)(t) |
								\leq 
								\sup_{t \in \clos}
								|f(t) |.
							\end{equation}
							Notice that, with $g$ a univariate and continuous function on a closed subset $B$ 
							of $[0, 1]$ containing the boundaries $0$ and $1$,
							by construction, the values of the affine extension of $L_B(g)$ lie in the range of $g$ values. This implies that, for $f \in \mathcal{C}(\clos , \mathbb{R})$,
							\[
							\sup_{t \in \clos_1 \times \dots  \times \clos_{d-1} \times [0,1]}
							| L_d \, f(t) |
							\leq 
							\sup_{\substack{ t_{\sim d} \in \clos_1 \times \dots  \times \clos_{d-1} \\  u_d \in \clos_d }}
							|f (u_d , t_{\sim d}) |
							\leq 
							\sup_{t \in \clos}
							|f(t) |.
							\]
							Similarly, we have
							\begin{align*}
								\sup_{t \in \clos_1 \times \dots  \times \clos_{d-2} \times [0,1]^2}
								| L_{d-1} L_d \, f(t) |
								& \leq 
								\sup_{\substack{ t_{\sim (d - 1)} \in \clos_1 \times \dots  \times \clos_{d-2} \times [0,1] \\  u_{d-1} \in \clos_{d-1} }}
								| L_d \,f (u_{d-1} , t_{\sim (d - 1)}) |
								\\
								& \leq 
								\sup_{t \in \clos}
								|f(t) |.
							\end{align*}
							Hence, by iteration we show \eqref{eq:to:show:extension:lipschitz}.
							
							\item 
							Any $f \in E_\subdiv$ can be written as $Y_{\subdiv,\alpha}$ with $\alpha \in A_{\subdiv}$.
							Let us first consider the marginal extension of $Y_{\subdiv, \alpha \vert \clos}$ with respect to coordinate $d$: for $t=(t_1,\ldots,t_d) \in \clos_1 \times \dots \times \clos_{d-1} \times [0,1]$, we have
							$$ L_d  Y_{\subdiv, \alpha \vert \clos} (t) 
							= L_{\clos_d} [u_d \mapsto Y_{\subdiv, \alpha \vert \clos}(t_1, \dots, t_{d-1}, u_d)] (t_d).$$
							Observe that, when $t_1, \dots, t_{d-1}$ are fixed in $\clos_1 \times \dots \times \clos_{d-1}$, the univariate function 
							$$u_d \mapsto Y_{\subdiv, \alpha \vert \clos_1 \times \dots \times \clos_{d-1} \times [0,1]}(t_1, \dots, t_{d-1}, u_d)$$
							is piecewise linear, and all its knots are contained in $\clos_d$, by definition of $\clos$.
							Hence, when restricting it to $\clos_d$ and then reextending it to $[0, 1]$ with $L_{\clos_d}$, 
							one obtains exactly the same function.
							In other words: for $t=(t_1,\ldots,t_d) \in \clos_1 \times \dots \times \clos_{d-1} \times [0,1]$, we have
							$$ L_d  Y_{\subdiv, \alpha \vert \clos} (t) =
							Y_{\subdiv, \alpha \vert \clos_1 \times \dots \times \clos_{d-1} \times [0,1]}(t).$$
							By the same reasoning, we have, for $t=(t_1,\ldots,t_d) \in \clos_1 \times \dots \times \clos_{d-2} \times [0,1]^2$,
							$$ L_{d-1} L_d  Y_{\subdiv, \alpha \vert \clos} (t) =
							Y_{\subdiv, \alpha \vert \clos_1 \times \dots \times \clos_{d-2} \times [0,1]^2}(t).$$
							Hence by an immediate induction:
							$ \extfun  Y_{\subdiv, \alpha \vert \clos} (t) = Y_{\subdiv, \alpha}(t)$ for $t  \in \domain$.
						\end{enumerate}
					\end{proof}
					
					\begin{proof}[{\bf Proof of Corollary \ref{cor:vertice:to:hypercube}}]
						The functions $f$ and $g$ are $d$-affine on $\Delta$, thus they are polynomial functions and can be defined on $\domain$. 
						Consider $\clos = \prod_{j=1}^d \left( [0,1] \backslash (x_j^- , x_j^+) \right)$.
						One can simply show that $f$ satisfies the conditions for $\extfun (f)$ in Proposition \ref{prop:multiaffine:extension},  then by unicity $f = \extfun(f)$. Similarly $g = \extfun (g)$. Finally, from \eqref{eq:multiExtFormula}, $ \extfun (g) = \extfun(f)$ since $f$ and $g$ coincide on the $2^d$ vertices of $\Delta$. 
					\end{proof}
					
					\begin{proof}[{\bf Proof of Lemma \ref{lem:non-empty:contains}}]
						Observe that $\strictineqset = \mathrm{int}_{|| . ||_{\infty}} ( \ineqset )$.
						Let us consider $g \in \Hilb \cap \mathrm{int}_{||.||_{\infty}} (\ineqset)$. 
						Then $g \in \Hilb \cap \ineqset$. Since $k$ is continuous and defined on the compact set $\domain \times \domain$, there exists a constant $C_{\sup}$ such that for all $h \in \mathcal{C}(\domain , \mathbb{R})$, $||h||_{\infty} \leq C_{\sup} ||h ||_{\Hilb}$ (see for instance Lemma 2 in \cite{bay2016generalization}). Let $\epsilon >0$ such that $||g-h||_{\infty} \leq \epsilon$ implies $h \in \ineqset$. Then, for all $h \in \Hilb$ such that $|| g - h ||_{\Hilb} \leq \epsilon / C_{\sup}$, we have $h \in \ineqset$. This means that $g$ is in $\mathrm{int}_{||.||_{\Hilb}} (\Hilb \cap \ineqset)$.
					\end{proof}
					
					\begin{lemma} \label{lem:LB:to:increasing:function}
						Let $B$ be a closed subset of $[0,1]$ containing $0$ and $1$. Let $f_B  \in \mathcal{C}(B,\mathbb{R})$ be non-decreasing. Then $L_B f_{B}$ is non-decreasing from $[0,1] \to \mathbb{R}$.
					\end{lemma}
					\begin{proof}[{\bf Proof of Lemma \ref{lem:LB:to:increasing:function}}]
						Write $L = L_B f_{B}$.
						Let $0 \leq v < w \leq 1$. Since $L$ is affine on $[v^- , v^+]$ and $L(v^-) = f_B(v^-) \leq f_B(v^+) = L(v^+)$, we have $L(v^-) \leq L(v) \leq L(v^+)$. Similarly  $L(w^-) \leq L(w) \leq L(w^+)$. Hence if (1) $v^+ \leq w^-$, then $L(v) \leq L(w)$. If (2) $v^+ > w^-$, then one can show that $[v,w] \cap B = \varnothing$ and thus $v^- = w^-$ and $v^+ = w^+$. Then $L$ is affine on $[v^- , v^+]$ with $L(v^-) \leq L(v^+) $ and  $v , w \in [v^- , v^+]$. Hence $L(v) \leq L(w)$. Hence $L$ is increasing from $[0,1] \to \mathbb{R}$.
					\end{proof}

					In the next lemma, for $B \subseteq \mathbb{R}$, we assume that $f_B : B \to \mathbb{R}$ is convex, i.e. for all $x , y\in B$, $\lambda \in [0,1]$ with $\lambda x + (1 - \lambda) y \in B$, we have
					\begin{equation} \label{eq:convex:on:B}
						f_B( \lambda x + (1 - \lambda) y) 
						\leq 
						\lambda f_B(x) + (1 - \lambda) f_B(y).
					\end{equation}

					\begin{lemma} \label{lem:LB:to:convex:function}
						Let $B$ be a closed subset of $[0,1]$ containing $0$ and $1$. Let $f_B  \in \mathcal{C}(B,\mathbb{R})$ be convex. Then $L_B f_{B}$ is convex from $[0,1] \to \mathbb{R}$.
					\end{lemma}
					
					\begin{proof}[{\bf Proof of Lemma \ref{lem:LB:to:convex:function}}]
						For a function $g$ defined on a domain $D \subseteq \R$, we recall the definition of the epigraph of $f$, 
						denoted $\epi(f)$:
						$$\epi(f) := \{(t,y) \in D \times \R \text{ s.t. } y \geq f(t) \}.$$
						We also recall that the convex hull of a set $A \subseteq \R^2$, denoted $\hull(A)$, is the smallest convex set containing $A$. Equivalently, it is the set of convex combinations of points in $A$ (see e.g. \cite{eggleston_1958}).
						Then, we are going to prove that, \textit{if $f_B$ is convex, then the convex hull of the epigraph of $f_B$ is the epigraph of its affine extension $L_B f_B$}: 
						$$\hull(\epi(f_B)) = \epi(L_B f_B).$$
						The result will then be deduced since it shows that $\epi(L_B f_B)$ is convex, which is equivalent to the convexity of $L_B f_B$ (as $L_B f_B$ is defined on the interval $[0, 1]$). Let us now come back to the proof.
						
						\begin{itemize}
							\item First let us show that (and even if $f_B$ is not convex)  $\epi(L_B f_B) \subseteq \hull(\epi(f_B))$.\\
							Let $(t,y) \in \epi(L_B f_B)$.
							\begin{itemize}
								\item If $t \in B$, then $y \geq L_B f_B(t) = f_B(t)$. Thus $(t, y) \in \epi(f_B) \subseteq \hull(\epi(f_B))$.
								\item If $t \notin B$, then $t^- < t < t^+$. Consider the straight line joigning $(t^-, f_B(t^-))$
								and $(t, y)$. Let $(t^+, y^+)$ the point on that straight line with abscissa $t^+$.
								Notice that $L_B f_B$ is a straight line on $[t^-, t^+]$, joining $(t^-, f_B(t^-))$ and $(t^+, f_B(t^+))$.
								Then, by Thal\`es theorem, the sign of $y^+ - f_B(t^+)$ is the same as $y - L_B f_B(t)$,
								which is positive because $(t,y) \in \epi(L_B f_B)$. Thus $(t^+, y^+) \in \epi(f_B)$. 
								Finally, we have shown that $(t,y)$ belongs to a line segment whose endpoints $(t^-, f_B(t^-)), (t^+, y^+)$ 
								belong to $\epi(f_B)$, which proves that $(t,y) \in \hull(\epi(f_B))$.
							\end{itemize}
							\item Conversely, let us prove that $\hull(\epi(f_B)) \subseteq \epi(L_B f_B)$. Here the convexity of $f_B$ is required.\\
							Let $(t,y) \in \hull(\epi(f_B))$.
							Thus $(t, y)$ belongs to a polygon $P$, which is either a singleton, a segment or a triangle and whose vertices $(t_1, y_1), \dots, (t_m, y_m)$ are in $\epi(f_B)$, with $m \in \{1 ,2 , 3 \}$, from Carath\'eodory's theorem.
							The intersection of $P$ with the band $[t^-, t^+] \times \R$ is a convex polygon containing $(t,y)$.
							Thus, from Krein–Milman theorem, $(t,y)$ is a convex combination of its extremal points $(u_1, z_1), \dots, (u_r, z_r)$.
							By definition of $t^-, t^+$, the points $t_1, \dots, t_m \notin (t^-, t^+)$. Then, one can see that these extremal points have abscissas  in $\{t^- , t^+ \}$ and either belong to $\{ (t_1, y_1), \dots, (t_m, y_m) \}$ or are in the segments with endpoints in $\{ (t_1, y_1), \dots, (t_m, y_m) \}$.
							
							Let us now prove that these extremal points are in $\epi(f_B)$. Consider for instance $(u_1, z_1)$.
							\begin{itemize}
								\item If $(u_1, z_1)$ is in $\{ (t_1, y_1), \dots, (t_m, y_m) \}$, it belongs to $\epi(f_B)$.
								\item Otherwise, assume without loss of generality that $(u_1, z_1)$ is the intersection of the segment joining $(t_1, y_1), (t_2, y_2)$ and $\{t^- \} \times \R$, with $t_1 \leq t^- \leq t_2$.
								Then, there exists $\lambda \in [0,1]$ such that $(u_1,z_1) = \lambda (t_1,y_1) + (1 - \lambda) (t_2,y_2)$. We have, by using that $(t_1,y_2) , (t_2,y_2) \in \epi(f_B)$ and by convexity of $f_B$,
								\begin{align*}
									z_1 & = \lambda y_1 + (1 - \lambda ) y_2
									\\
									& \geq \lambda f_B(t_1) + (1 - \lambda ) f_B(t_2)
									\\
									& \geq f_B( \lambda t_1 + (1 - \lambda ) t_2)
									\\
									& = f_B(u_1).
								\end{align*}
								Hence $(u_1, z_1)$ belongs to $\epi(f_B)$.
							\end{itemize}
							Now, $L_B f_B$ is affine on $[t^-, t^+]$. Recall that $u_1, \dots, u_r$ are elements of $\{t^-, t^+\}$, and thus are in $B$. Hence $L_B f_B$ coincides with $f_B$ at $u_1, \dots, u_m$.
							Writing that $(t,y)$ is a convex combination of $(u_1, z_1), \dots, (u_r, z_r)$, and using that $(u_1, z_1), \dots, (u_r, z_r) \in \epi(f_B)$, we obtain:
							$$y = \sum_{i=1}^r \lambda_i z_i \geq \sum_{i=1}^r \lambda_i f_B(u_i) 
							=  \sum_{i=1}^r \lambda_i L_B f_B(u_i) =
							L_B f_B \left( \sum_{i=1}^r \lambda_i u_i \right)
							=
							L_B f_B(t).$$
							This proves that $(t,y) \in \epi(L_B f_B)$.
						\end{itemize}
						
					\end{proof}

					\begin{proof}[{\bf Proof of Lemma \ref{lem:cond:PfC:subset:C}}]
						{\bf Consider that $\ineqset$ is given by \eqref{eq:boundedness}.} 
						Let $f \in \ineqset$, with $f_{|\clos}$ the restriction of $f$ to $\clos$.
						From \eqref{eq:multiExtFormula}, for $x = (x_1,\ldots,x_d) \in \domain$, the value of $P_{\clos \to \domain} (f_{|F}) (x)$ is a convex combination of the values $(f( x_1^{\epsilon_1} , \ldots , x_d^{\epsilon_d} ))_{\epsilon_1,\ldots,\epsilon_d \in \{-,+\}}$. These $2^d$ values are in $[a,b]$ and thus $P_{\clos \to \domain} (f_{|F}) (x)$ is also in $[a,b]$. Thus $P_{\clos \to \domain} (f_{|F}) \in \ineqset$.
						
						{\bf Consider that $\ineqset$ is given by \eqref{eq:monotonicity}.}
						Let $i \in \{1,\ldots,d\}$. Let $x \in \domain$ and $0 \leq u_i < v_i \leq 1$. From \eqref{eq:multiExtFormula}, we have, with the notation $x_{\sim i}^{\epsilon} = (x_{1}^{\epsilon_1},\ldots,x_{i-1}^{\epsilon_{i-1}},x_{i+1}^{\epsilon_{i+1}},\ldots,x_{d}^{\epsilon_d})$,
						\begin{align}   
							& \extfun (f_{| \clos}) (v_i,x_{\sim i})
							-
							\extfun (f_{| \clos}) (u_i,x_{\sim i})
							\notag \\
							= & \sum_{\epsilon_1, \dots, \epsilon_d \in \{-, +\}} 
							\left(w_{\epsilon_i}( v_i ) \prod_{\substack{j = 1 \\j \neq i }}^d \omega_{\epsilon_j}(x_j)  \right)
							f(v_i^{\epsilon_i},  x_{\sim i}^{\epsilon})
							\notag 
							\\
							& -
							\sum_{\epsilon_1, \dots, \epsilon_d \in \{-, +\}} 
							\left(w_{\epsilon_i}( u_i ) \prod_{\substack{j = 1 \\j \neq i }}^d \omega_{\epsilon_j}(x_j)  \right)
							f(u_i^{\epsilon_i},  x_{\sim i}^{\epsilon})
							\notag \\
							= &
							\sum_{ \epsilon_{\sim i} \in \{-, +\}^{d-1}}
							\left( \prod_{\substack{j = 1 \\j \neq i }}^d \omega_{\epsilon_j}(x_j) 
							\notag  
							\Bigg(
							\sum_{\epsilon_i \in \{-,+\}}
							w_{\epsilon_i}( v_i )
							f(v_i^{\epsilon_i},  x_{\sim i}^{\epsilon})
							-
							\sum_{\epsilon_i \in \{-,+\}}
							w_{\epsilon_i}( u_i )
							f(u_i^{\epsilon_i},  x_{\sim i}^{\epsilon})
							\Bigg)
							\notag \right)  \\
							= & 
							\sum_{\epsilon_{\sim i} \in \{-, +\}^{d-1}}
							\left(\prod_{\substack{j = 1 \\j \neq i }}^d \omega_{\epsilon_j}(x_j) 
							\left(
							L_{\clos_i}  f_{| \clos}(\cdot,  x_{\sim i}^{\epsilon})( v_i )
							-
							L_{\clos_i}  f_{| \clos}(\cdot,  x_{\sim i}^{\epsilon})( u_i )
							\right ) \right). \label{eq:for:monotonicity:weighted:sum:two}
						\end{align}
						In the above display, the function 
						$f_{| \clos}(\cdot, x_{\sim i}^{\epsilon})$
						is continuous non-decreasing from $\clos_i$ to $\mathbb{R}$. Hence, from Lemma \ref{lem:LB:to:increasing:function}, the function
						$$t \in [0,1] \mapsto  L_{\clos_i}  f_{| \clos}(\cdot, x_{\sim i}^{\epsilon}) ( t )$$
						is non-decreasing. Hence, the difference in \eqref{eq:for:monotonicity:weighted:sum:two} is non-negative and the weighted sum 
						is non-negative since the weights are non-negative. This concludes the proof.

						{\bf Consider that $\ineqset$ is given by \eqref{eq:convexity}.} Then the proof is identical to the case where $\ineqset$ is given by \eqref{eq:monotonicity}. Instead of using Lemma \ref{lem:LB:to:increasing:function} we use Lemma \ref{lem:LB:to:convex:function}. 
						For $i \in \{1,\ldots,d\}$, $x \in \domain$ and $0 \leq u_i < v_i < w_i \leq 1$, instead of considering the first order finite difference
						\[
						\extfun (f_{| \clos}) (v_i,x_{\sim i})
						-
						\extfun (f_{| \clos}) (u_i,x_{\sim i}),
						\]
						we consider the second order finite difference
						\begin{align*}
							& 
							\frac{ 
								\extfun (f_{| \clos}) (w_i,x_{\sim i})
								-
								\extfun (f_{| \clos}) (v_i,x_{\sim i})
							}{
								w_i - v_i
							}
							\\
							& -
							\frac{ 
								\extfun (f_{| \clos}) (v_i,x_{\sim i})
								-
								\extfun (f_{| \clos}) (u_i,x_{\sim i})
							}{
								v_i - u_i
							}.
						\end{align*}
						
					\end{proof}
					
					\begin{proof}[{\bf Proof of Theorem \ref{theorem:extension:convergence:spline}}]
						To prove the uniform convergence, we cannot directly apply Theorem \ref{theorem:convergence:spline}
						because the set of d-dimensional knots is not assumed to be dense in $[0, 1]^d$.
						The idea is to match this situation, by considering functions defined on their closure $\clos$ and by using the multiaffine extension $P_{\clos \to \domain}$.
						
						We will thus find functions and sets corresponding to $\widehat{Y}_{\subdiv^{(m)}}$, $k$, $\Hilb$, $\ineqset$ and $\interpset_{X,y^{(n)}}$ in Theorem \ref{theorem:convergence:spline}. We will then show that conditions corresponding to those of Theorem \ref{theorem:convergence:spline} hold. Under these new conditions, the proof of Theorem \ref{theorem:convergence:spline} can be repeated, enabling us to obtain the conclusion of Theorem \ref{theorem:extension:convergence:spline}.
						
						\indent 
						To $\widehat{Y}_{\subdiv^{(m)}}$, we thus associate  $\widehat{Y}_{\subdiv^{(m)} \vert \clos}$, its restriction to $F$. The input space of $\widehat{Y}_{\subdiv^{(m)}}$ is $\domain$  (written $X$ in \cite{bay2017new}) while the input space of $\widehat{Y}_{\subdiv^{(m)} \vert \clos}$ is $\clos$. To $k$ we associate $k_{\clos}$.
						To $\interpset_{X,y^{(n)}}$ (the set of equality constraints, written $I$ in \cite{bay2017new}) we associate $\interpset_{\clos,X,y^{(n)}}$. To $\ineqset$ (the set of inequality constraints) we associate $\ineqset_\clos$. To $\phi_{\multi }^{(\subdiv^{(m)})}$, $\multi \in \mathcal{L}_{\subdiv^{(m)}}$ (the basis functions) we associate their restriction to $\clos$, $\phi_{\multi |\clos}^{(\subdiv^{(m)})}$.
						Then notice that  $\widehat{Y}_{\subdiv^{(m)} \vert \clos}$ indeed corresponds to the mode function $\widehat{Y}_{\subdiv^{(m)}}$ of Theorem \ref{theorem:convergence:spline}, in the sense that we have, for $t \in \clos$,
						\begin{equation*}
							\widehat{Y}_{\subdiv^{(m)} \vert \clos} = \sum_{\multi \in \mathcal{L}_{\subdiv^{(m)}}} (\widehat{\alpha}_{\subdiv^{(m)}})_{\multi} \phi_{\multi | \clos}^{(\subdiv^{(m)})}(t)
						\end{equation*}
						and
						\begin{align} 
							\widehat{\alpha}_{\subdiv^{(m)}}
							&
							\in \underset{ \substack{\alpha \in A_{\subdiv^{(m)}}
									\\
									Y_{\subdiv^{(m)},\alpha} \, \in \, \interpset_{X,y^{(n)}} \, \cap \, \ineqset
							} }{\mathrm{argmin}}
							\alpha^\top
							k(\subdiv^{(m)},\subdiv^{(m)})^{-1}
							\alpha \notag
							\\
							& = \underset{ \substack{\alpha \in A_{\subdiv^{(m)}}
									\\
									Y_{\subdiv^{(m)},\alpha \vert \clos} \, \in \,  \interpset_{\clos,X,y^{(n)}} \, \cap \, \ineqset_\clos
							} }{\mathrm{argmin}}
							\alpha^\top
							k_\clos(\subdiv^{(m)},\subdiv^{(m)})^{-1}
							\alpha. \label{eq:for:apply:bayetal}
						\end{align}
						Above, we have used that $ \extfun Y_{\subdiv^{(m)},\alpha \vert \clos} = Y_{\subdiv^{(m)},\alpha}$, from Proposition \ref{prop:extension:properties}. We finally remark that, by construction, $\clos$ contains all the knots $t^{(\subdiv^{(m)})}_{\multi}$, $\multi \in \mathcal{L}_{\subdiv^{(m)}}$ obtained from $\subdiv^{(m)}$.

						Then, let us prove that the following assumptions, corresponding to those required for Theorem~\ref{theorem:convergence:spline}, are satisfied.
						Notice that $\ineqset_\clos$ is convex, since the map $P_{\clos \to \domain}$ is affine and $\ineqset$ is convex.
						Then, we will show that the following extensions to $(H1)$ and $(H2)$ hold:
						\begin{itemize}
							\item[$(H1,\clos)$] {\it $\mathrm{int}_{||.||_{\Hilb_\clos}}(\Hilb_\clos \cap \ineqset_\clos) \cap \interpset_{\clos,X,y^{(n)}} \neq \emptyset$.}  
							\item[$(H2,\clos)$] {\it $\pi_{\subdiv^{(m)} \vert \clos}(\ineqset_\clos) \subseteq \ineqset_\clos$
							}
						\end{itemize}
						
						Let us first prove $(H2,\clos)$. 
						Let $f_\clos \in \ineqset_\clos$. Thus $f = \extfun(f_\clos) \in \ineqset$.
						Therefore  $\pi_{\subdiv^{(m)}}(f) \in \ineqset$ from Condition \ref{cond:piSC:subset:C}.
						Now $\pi_{\subdiv^{(m)} \vert \clos}(f_\clos)$ is the restriction to $\clos$ of the function $\pi_{\subdiv^{(m)}}(f)$,
						since $f = f_\clos$ on $\clos$ (by definition of $\extfun$).
						Finally, by Proposition~\ref{prop:extension:properties}, 
						$$ \extfun \left(\pi_{\subdiv^{(m)} \vert \clos}(f_\clos) \right) 
						= \extfun \left( \left( \pi_{\subdiv^{(m)}}(f) \right)_{\vert \clos} \right) = \pi_{\subdiv^{(m)}}(f) \in \ineqset.$$
						Hence $\pi_{\subdiv^{(m)} \vert \clos}(f_\clos) \in \ineqset_{\clos}$ which shows $(H2,\clos)$.
						
						Let us now prove $(H1,\clos)$. For $i=1,\ldots,n$ let us write $x^{(i)} = (x^{(i)}_1,\ldots,x^{(i)}_d)$.
						For $j=1,\ldots,d$, write $x^{(i)}_{j,-} = \max \{ u \in \clos_j; u \leq x^{(i)}_{j}  \}$ 
						and $x^{(i)}_{j,+} = \min \{ u \in \clos_j; u \geq x^{(i)}_{j}  \}$. The set 
						\[
						\{  ( x^{(i)}_{1,\epsilon_1} , \ldots , x^{(i)}_{d,\epsilon_d} ) \}_{\epsilon_1,\ldots,\epsilon_d \in \{ - , +\} , i=1,\ldots,n}
						\]
						can be written as $\{ w_1,\ldots,w_p\}$ with $w_1,\ldots,w_p$ two by two distinct.
						
						From Condition \ref{cond:non:empty}, for each $m  \geq 0$
						we have
						\begin{equation}
							\varnothing
							\neq 
							\{ \alpha \in A_{\subdiv^{(m_0)}};
							Y_{\subdiv^{(m_0)},\alpha} \in  \interpset_{X,y^{(n)}} \cap \strictineqset
							\}
							\subseteq 
							\{ \alpha \in A_{\subdiv^{(m)}};
							Y_{\subdiv^{(m)},\alpha} \in \interpset_{X,y^{(n)}} \cap \strictineqset
							\label{eq:in:proof:convergence:the:set}
							\},
						\end{equation}
						
						because the sequence of function spaces $\{ Y_{\subdiv^{(m)},\alpha} ; \alpha \in A_{\subdiv^{(m)}} \}_{m \geq m_0}$ is nested.
						Hence, we can take $\alpha$ in the set in \eqref{eq:in:proof:convergence:the:set}.
						Write, for $i=1,\ldots,p$, $z_i = Y_{\subdiv^{(m)},\alpha}(w_i)$. 
						
						Then from Condition \ref{cond:interior:non:empty}, the set 
						$\mathrm{int}_{||.||_{\Hilb}}(\Hilb \cap \ineqset) \cap \interpset_{W,z^{(n)}}$ is non-empty. 
						For $f$ in this set, let us show that $f_{\vert \clos}$ 
						belongs to $\mathrm{int}_{||.||_{\Hilb_\clos}}(\Hilb_\clos \cap \ineqset_\clos) \cap \interpset_{\clos,X,y^{(n)}}$.
						Thus, this set will be non-empty, what is to prove.
						
						Firstly, let us check that $f_{\vert \clos} \in \interpset_{\clos,X,y^{(n)}}$. 
						Indeed, on each of the $n$ hypercubes 
						\[
						\prod_{j=1,\ldots,d} [ x^{(i)}_{j,-}, x^{(i)}_{j,+} ],
						\]
						$i=1,\ldots,n$, $P_{\clos \to \domain}(f_{\vert \clos})$ coincide with $Y_{\subdiv^{(m)},\alpha} $. Indeed, consider one of these hypercubes. 
						The $2^d$ vertices of this hypercube belong to $\clos$, so on these $2^d$ points, $P_{\clos \to \domain}(f_{\vert \clos})$ coincide with $f_{\vert \clos}$ which coincide with $f$ which coincide with $Y_{\subdiv^{(m)},\alpha}$.
						Furthermore, the two functions $P_{\clos \to \domain}(f_{\vert \clos})$ and $Y_{\subdiv^{(m)},\alpha}$ are $d$-affine on this hypercube, and we have shown that they take the same values on the $2^d$ vertices. Thus they are equal on the hypercube from Corollary \ref{cor:vertice:to:hypercube}.
						Hence in particular $P_{\clos \to \domain}(f_{\vert \clos})\left(x^{(i)}\right) =  Y_{\subdiv^{(m)},\alpha}\left( x^{(i)} \right) = y_i$ and thus $P_{\clos \to \domain}(f_{\vert \clos}) \in \interpset_{X,y^{(n)}}$ and so 
						$f_{\vert \clos} \in \interpset_{\clos,X,y^{(n)}}$. Note that the above argumentation still goes through in the case where there exist $i,j$'s such that $x^{(i)}_{j,-} =  x^{(i)}_{j,+}$.
						
						Secondly, notice that $f \in \Hilb \cap \ineqset$. By Theorem 6 in \cite{berlinet2011reproducing}, $\Hilb_\clos$ is formed by restrictions to $F$ of functions in $\Hilb$, and for all $f_F \in \Hilb_\clos$, we have 
						$\Vert f_F \Vert_{\Hilb_\clos} = \displaystyle
						\inf_{h \in \Hilb, h_{\vert \clos} = f_F} \Vert h \Vert$.
						Thus $f_{\vert \clos} \in \Hilb_\clos$. 
						Moreover, $f_{\vert \clos} \in \ineqset_\clos$ by Condition \ref{cond:PfC:subset:C}. 
						Furthermore, let $\epsilon>0$ be such that $g \in \Hilb \cap \ineqset$ 
						for all $g \in \Hilb$ such that $||g-f||_{\Hilb} \leq 2 \epsilon$ (recall that $f \in \mathrm{int}_{||.||_{\Hilb}}(\Hilb \cap \ineqset)$). 
						Now, let $g_\clos \in \Hilb_\clos$ with $||g_{\clos} - f_{\vert \clos}||_{\Hilb_\clos} \leq \epsilon$.
						By Theorem 6 in \cite{berlinet2011reproducing}, there exists $\psi \in \Hilb$ such that $\psi_{\vert \clos} = g_\clos - f_{| \clos}$,
						and 
						$\Vert \psi \Vert_\Hilb  \leq  \Vert g_\clos - f_{\vert \clos} \Vert_{\Hilb_\clos} + \epsilon \leq 2 \epsilon.$
						This implies that $f + \psi$ belongs to $\Hilb \cap \ineqset$. 
						This in turn implies that $(f + \psi)_{|\clos}$ belongs to $\Hilb_\clos \cap \ineqset_\clos$, 
						with the same arguments as above. Also, we have $(f + \psi)_{|\clos} = f_{|\clos} + g_\clos - f_{| \clos} = g_\clos$, so $ g_\clos \in \Hilb_\clos \cap \ineqset_\clos$.
						Hence, $f_{\vert \clos} \in \mathrm{int}_{||.||_{\Hilb_\clos}}(\Hilb_\clos \cap \ineqset_\clos)$.

						With the conditions $(H1,\clos)$, $(H2,\clos)$, 
						one can check that the proof of Theorem~\ref{theorem:convergence:spline} can be carried out, 
						similarly as in \cite{bay2017new}. There are just a few modifications, that we now explain.
						
						The main modification is that in \cite{bay2017new},  the statement of Lemma 1 and its proof need to be adapted to our context. The adapted lemma is Lemma \ref{lem:adapted:from:bay:et:al}, that we state and prove below.
						
						The second modification is that the reference \cite{bay2017new} considers functions indexed on $[0,1]$ while we consider functions indexed on $\clos$. This only entails straightforward changes, since we prove convergence or uniform convergence of functions on $\clos$, where the set of knots $S^{(m)}_1 \times \dots \times S^{(m)}_d$ is dense in $\clos$.
						
						The last modification is that in \cite{bay2017new}, the set of inequality constraints corresponding to $\interpset_{\clos,X,y^{(n)}}$ is of the form
						\begin{equation} \label{eq:old:form}
							\left \lbrace
							f : \domain \to \R; f\left(x^{(i)}\right) = y_i,i=1,\ldots,n
							\right \rbrace,
						\end{equation}
						while the set $\interpset_{\clos,X,y^{(n)}}$ in our case is of the form
						\begin{equation} \label{eq:new:form}
							\bigg\{
							f_{\vert \clos} : \clos \to \R; 
							\sum_{j=1}^q
							\lambda_{i,j} f(a_j)
							= y_i, i=1,\ldots,n
							\bigg\},
						\end{equation}
						where the fixed coefficients $(\lambda_{i,1} , \ldots ,  \lambda_{i,q})_{i=1,\ldots,n}$ and the points $a_{1}, \ldots , a_q \in \clos$ are explicited in the proof of Lemma \ref{lem:adapted:from:bay:et:al} and come from \eqref{eq:multiExtFormula}. 
						
						The difference between \eqref{eq:old:form} and \eqref{eq:new:form} changes the arguments in the proof of \cite{bay2017new} only in the first item after (13) there, the change being straightforward.
						Thus, from this adaptation of the proof in \cite{bay2017new},
						$\widehat{Y}_{\subdiv^{(m)} \vert \clos}$ converges uniformly on $\clos$ to the function $Y_{\clos,\text{opt}}$
						defined in the text of Theorem~\ref{theorem:extension:convergence:spline}.
						
						By Proposition~\ref{prop:extension:properties}, this implies that
						$\widehat{Y}_{\subdiv^{(m)}}$ converges uniformly on $\domain$ to the function $\extfun(Y_{\clos,\text{opt}})$.
					\end{proof}

					For the next lemmas, notice that the definition of $\pi_{\subdiv^{(m)}}(f)$ in \eqref{eq:piSm} 
					for functions in $\mathcal{C}(\domain , \mathbb{R})$ can be extended for functions
					$f \in \mathcal{C}(\clos , \mathbb{R})$.
					Indeed, it relies only on the sequence of knots $(t^{(\subdiv^{(m)})}_{\multi})_{\multi \in \mathcal{L}_{\subdiv^{(m)}}}$, which are included in $\clos$.

					\begin{lemma} \label{lem:adapted:from:bay:et:al}
						Consider the setting of Theorem \ref{theorem:extension:convergence:spline}.
						Write 
						\[
						F_m = \{ f \in \Hilb_\clos : 
						\pi_{\subdiv^{(m)}}(f) (x^{(i)}) = y_i , i=1,\ldots,n
						\}.
						\]
						Let $g \in \Hilb_\clos \cap \interpset_{\clos,X,y^{(n)}} $. Then for $m$ large enough we can define $g_m$ by
						\[
						g_m =  \mathrm{argmin}_{h \in F_m}
						|| h - g ||_{\Hilb_\clos}.
						\] 
						Furthermore as $m \to \infty$,
						\[
						|| g_m - g ||_{\Hilb_\clos} \to 0.
						\] 
					\end{lemma}
					
					The interpretation of  Lemma \ref{lem:adapted:from:bay:et:al} is that functions in $\Hilb_\clos$ satisfying the $n$ equality constraints are asymptotically well approximated by their projections on $F_m$. The space $F_m$ is the set of functions in $\Hilb_\clos$ that satisfy the $n$ equality constraints, when interpolated through $\pi_{\subdiv^{(m)}}$. Note that this convergence result applies to functions defined on $\clos$ and follows from the density of the knots on $\clos$.
					
					\begin{proof}[{\bf Proof of Lemma \ref{lem:adapted:from:bay:et:al}}]
						We introduce notation for the current subdivision and the current left and right neighbors of $(x^{(i)}_j)_{i=1,\ldots,n,j=1,\ldots,d}$, as in Section \eqref{susubsection:multiaffine:extension}.
						For $i \in \{1 , \ldots , d\}$ and $m \geq m_0$, let us write $x^{(i)}_{m,j,-} = \max \{ u \in \subdiv^{(m)}_j; u \leq x^{(i)}_{j}  \}$ 
						and $x^{(i)}_{m,j,+} = \min \{ u \in \subdiv^{(m)}_j; u \geq x^{(i)}_{j}  \}$.
						Also, if $x^{(i)}_j \not \in \clos_j$,  we write 
						$\omega^m_+(x^{(i)}_j) = \frac{x^{(i)}_j - x^{(i)}_{m,j,-}}{x^{(i)}_{m,j,+} - x^{(i)}_{m,j,-}}$ and $\omega^m_-(x^{(i)}_j) = 1- \omega^m_+(x^{(i)}_j)$. If $x^{(i)}_j  \in \clos_i$,  we write 
						$\omega^m_+(x^{(i)}_j) = \omega^m_-(x^{(i)}_j) = 1/2$.

						We then have, for $i=1,\ldots,n$, for $f \in \mathcal{C}(\clos,\mathbb{R})$,
						\begin{equation} \label{eq:the:sum:with:products}
							\pi_{\subdiv^{(m)}}(f) (x^{(i)})
							=
							\sum_{\epsilon_1, \dots, \epsilon_d \in \{-, +\}} 
							\Bigg(\prod_{j = 1}^d \omega^m_{\epsilon_j}(x^{(i)}_j)  \Bigg)
							f( x^{(i)}_{m,1,\epsilon_1} , \ldots , x^{(i)}_{m,d,\epsilon_d} ),
						\end{equation}
						from \eqref{eq:multiExtFormula} and \eqref{eq:pi:S:f:equal:multiaffine:extension}.
						
						Reindexing by a single index the $n 2^d$ vertices of hypercubes in \eqref{eq:the:sum:with:products},  let us write $q=n 2^d$,
						\[
						\{a_{ m,1 } , \ldots, a_{m,q}\} \subset \clos
						\]
						and, for $i=1,\ldots,n$,
						\[
						\{\lambda_{m,i, 1 } , \ldots, \lambda_{m,i,q}\} \subset \mathbb{R},
						\] 
						such that, for $f \in \mathcal{C}(\clos,\mathbb{R})$,
						\[
						\pi_{\subdiv^{(m)}}(f) (x^{(i)})
						=
						\sum_{j=1}^q
						\lambda_{m,i,j} f(a_{m,j}).
						\] 
						Notice that the $(\lambda_{m,i, j })_{i=1,\ldots,n,j=1,\ldots,q}$ are the products in \eqref{eq:the:sum:with:products} (also reindexed) and are thus between $0$ and $1$.
						Similarly there exist 
						\[
						\{a_{ 1 } , \ldots, a_{q}\} \in \clos
						\]
						and, for $i=1,\ldots,n$,
						\[
						\{\lambda_{i, 1 } , \ldots, \lambda_{i,q}\}
						\] 
						such that, for $f \in \mathcal{C}(\clos,\mathbb{R})$,
						\begin{equation} \label{eq:sum:lambda:f:a}
							\extfun (f) (x^{(i)})
							=
							\sum_{j=1}^q
							\lambda_{i,j} f(a_{j}).
						\end{equation}
						Let us show that, for $m$ large enough, there exist $\gamma_1,\ldots,\gamma_n\in \mathbb{R}$ such that
						\begin{equation} \label{eq:gamma:combination}
							\pi_{S^{(m)}}
							\Bigg(
							\sum_{i'=1}^n 
							\gamma_{i'}
							\Bigg(
							\sum_{j'=1}^q
							\lambda_{m,i',j'}
							k_\clos ( \cdot , a_{m,j'})
							\Bigg)
							\Bigg)
							(x^{(i)})
							=
							y_i
							~ ~,
							i=1,\ldots,n.
						\end{equation}
						This is equivalent to 
						\[
						\sum_{j=1}^q
						\lambda_{m,i,j}
						\Bigg(
						\sum_{i'=1}^n 
						\gamma_{i'}
						\Bigg(
						\sum_{j'=1}^q
						\lambda_{m,i',j'}
						k_\clos ( a_{m,j} , a_{m,j'})
						\Bigg)
						\Bigg)
						=
						y_i
						~ ~,
						i=1,\ldots,n.
						\]
						This is equivalent to 
						\[
						\sum_{i'=1}^n 
						\gamma_{i'}
						\sum_{j=1}^q
						\sum_{j'=1}^q
						\lambda_{m,i,j}
						\lambda_{m,i',j'}
						k_\clos ( a_{m,j} , a_{m,j'})
						=
						y_i
						~ ~,
						i=1,\ldots,n.
						\]
						With $\gamma = (\gamma_1,\ldots,\gamma_n)^\top$, this is equivalent to
						\[
						R_m \gamma = y^{(n)},
						\]
						where $R_m$ is the $n \times n$ matrix with 
						\[
						(R_m)_{i,i'} = \sum_{j=1}^q 
						\sum_{j'=1}^q
						\lambda_{m,i,j}
						\lambda_{m,i',j'}
						k_\clos ( a_{m,j} , a_{m,j'}).
						\]

						Consider a GP $Z$ on $\clos$ with continuous trajectories and covariance function $k_{\clos}$. Note that this exists by taking the restriction to $\clos$ of a GP on $[0,1]^d$ with continuous trajectories and covariance function $k$.
						Then $R_m$ is the covariance matrix of the Gaussian vector
						\begin{equation} \label{eq:vector:Rm}
							\Bigg(
							\sum_{j=1}^q
							\lambda_{m,i,j} 
							Z( a_{m,j} )
							\Bigg)_{i=1,\ldots,n}
							=
							\left(
							\pi_{\subdiv^{(m)}}(Z)
							( x^{(i)} )
							\right)_{i=1,\ldots,n}.
						\end{equation}

						Define $R$ as the covariance matrix of the Gaussian vector
						\begin{equation} \label{eq:vector:R}
							\Bigg(
							\sum_{j=1}^q
							\lambda_{i,j} 
							Z( a_j )
							\Bigg)_{i=1,\ldots,n}
							=
							\left(
							\extfun(Z)
							( x^{(i)} )
							\right)_{i=1,\ldots,n}.
						\end{equation}

						We have that, for all the (continuous) trajectories of $Z$, $\pi_{\subdiv^{(m)}}(Z)$ converges uniformly to $\extfun(Z)$ on $[0,1]^d$ as $m \to \infty$ from Lemma \ref{lem:pi:m:to:ext}. Hence, the Gaussian vector \eqref{eq:vector:Rm} converges almost surely to the Gaussian vector \eqref{eq:vector:R} as $m \to \infty$. Thus, by Gaussianity (see for instance \cite[Lemma 1]{ibragimov1978gaussian}), the covariance matrix of the Gaussian vector \eqref{eq:vector:Rm} converges as $m \to \infty$ to the covariance matrix of the Gaussian vector \eqref{eq:vector:R}. Hence $R_m$ goes to $R$ as $m \to \infty$.

						Let us show that $R$ is invertible.
						Up to a re-arrangement in the sum \eqref{eq:sum:lambda:f:a}, we can assume that $a_1,\ldots,a_q$ are two by two distinct (in this case we have removed duplicates and $q$ can be smaller than $n2^d$).
						From Condition \ref{cond:spans:Rn}, for any $v^{(n)} \in \mathbb{R}^n$, there exists a function $g$ in $E_{\subdiv^{(m_0)}}$ such that 
						\[
						(g(x^{(i)}))_{i=1,\ldots,n}
						=
						v^{(n)}.
						\]
						The function $g$ is equal to $\extfun(g)$ from Proposition \ref{prop:extension:properties}, item 2. Hence, with this function $g$ we have from \eqref{eq:sum:lambda:f:a}
						\[
						\left(
						\extfun(g)
						( x^{(i)} )
						\right)_{i=1,\ldots,n}
						=
						\Bigg(
						\sum_{j=1}^q
						\lambda_{i,j}
						g( a_j )
						\Bigg)_{i=1,\ldots,n}
						=
						v^{(n)}.
						\]
						This means that the application
						\[
						(h_1,\ldots , h_q)
						\mapsto 
						\Bigg(
						\sum_{j=1}^q
						\lambda_{i,j}
						h_j
						\Bigg)_{i=1,\ldots,n}
						\]
						is surjective. 
						In addition, the covariance matrix of $(Z(a_j))_{j=1}^q$ is invertible because $a_1,\ldots,a_q$ are two-by-two distinct. Hence the Gaussian vector 
						\[
						\Bigg(
						\sum_{j=1}^q
						\lambda_{i,j}
						Z(a_j)
						\Bigg)_{i=1,\ldots,n}
						\]
						has an invertible covariance matrix. Its covariance matrix is $R$ which shows that $R$ is invertible.

						Hence, for $m$ large enough, $R_m$ is invertible  and there exist $\gamma_1,\ldots,\gamma_n$ such that \eqref{eq:gamma:combination} holds and thus $F_m$ is non-empty and thus $g_m$ is well-defined (from the classical projection theorem). Let us take $m$ large enough such that $R_m$ is invertible for the rest of the proof.

						We define the spaces
						$G_0^m$ and $G_1^m$, respectively, as 
						\[
						G_0^m = \Bigg\{ f \in \Hilb_\clos : 
						\sum_{j=1}^q
						\lambda_{m,i,j} f(a_{m,j}) = 0 , i=1,\ldots,n
						\Bigg\}
						\]
						and
						\[
						G_1^m
						=
						\mathrm{span}
						\Bigg(
						\sum_{j=1}^q
						\lambda_{m,i,j}
						k_\clos ( \cdot , a_{m,j})
						;
						i=1,\ldots,n
						\Bigg).
						\]

						For arbitrary $f$ in $F_m$, following the proof of Lemma 1 in  \cite{bay2017new}, we have $F_m = f + G_0^m$ and $g_m = f + P_{G_0^m}(g - f)  $, where $P_{G_0^m}$ is the orthogonal projection onto $G_0^m$. Therefore, $g - g_m = g - f - P_{G_0^m}(g - f) \in (G_0^m)^\perp = G_1^m$, where $(G_0^m)^\perp $ is the orthogonal space to $G_1^m$. Then there exist $\beta^m_1 , \ldots , \beta^m_n \in \mathbb{R}$ such that
						\begin{equation} \label{eq:g:minus:gm}
							g - g_m
							=
							\sum_{i'=1}^n
							\beta^m_{i'}
							\sum_{j'=1}^q
							\lambda_{m,i',j'}
							k_\clos ( \cdot , a_{m,j'}).
						\end{equation}
						Hence for $i=1,\ldots,n$ we have
						\begin{align*}
							\Bigg(
							\sum_{j=1}^q
							\lambda_{m,i,j}
							g (a_{m,j})
							\Bigg)
							-
							y_i
							& =
							\sum_{j=1}^q
							\lambda_{m,i,j}
							(g - g_m) (a_{m,j})
							\\
							& =
							\sum_{j=1}^n
							\lambda_{m,i,j}
							\sum_{i'=1}^n
							\beta^m_{i'}
							\sum_{j'=1}^q
							\lambda_{m,i',j'}
							k_\clos ( a_{m,j} , a_{m,j'}).
						\end{align*}
						Hence $\beta^m = (\beta^n_1,\ldots,\beta_n^m)^\top$ is solution of the system
						\[
						R_m \beta_m
						=
						z_m,
						\]
						where 
						\[
						z_m
						=
						\Bigg(
						\sum_{j=1}^q
						\lambda_{m,i,j}
						g(a_{m,j})
						-
						y_i
						\Bigg)_{i=1,\ldots,n}.
						\]
						
						The matrix $R^m$ converges to an invertible matrix as seen before. Furthermore, from Lemma \ref{lem:pi:m:to:ext}, 
						\[
						\sum_{j=1}^q
						\lambda_{m,i,j}
						g(a_{m,j}) = \pi_{\subdiv^{(m)}}(g)(x^{(i)}) \to_{m \to \infty} \extfun(g) (x^{(i)}) = y_i
						\]
						since $g \in \interpset_{\clos,X,y^{(n)}}$. Hence the vector $z_m$ goes to zero. Hence $\beta^m_1 , \ldots , \beta^m_n$ go to zero as $m \to \infty$.   Furthermore, the $\lambda_{m,i,j}$'s are non-negative and bounded by $1$ and $q$ is fixed, hence, from \eqref{eq:g:minus:gm} and the triangle inequality,
						\[
						||g - g_m||_{\Hilb_{\clos}}
						\leq 
						n  q
						\max_{i'=1,\ldots,n} 
						\max_{j'=1,\ldots,q} 
						|\beta^m_{i'}|
						|\lambda_{m,i',j'}|
						\max_{x \in \clos}
						\sqrt{k(x,x)}
						\to_{m \to \infty} 0,
						\]
						which concludes the proof.
						
					\end{proof}

					\begin{lemma} \label{lem:pi:m:to:ext}
						Consider the setting of Theorem \ref{theorem:extension:convergence:spline}.
						Let $f \in \mathcal{C}(\clos , \mathbb{R})$. Then as $m \to \infty$, $\pi_{\subdiv^{(m)}}(f) - \extfun (f) \to 0$, uniformly on $\domain$. 
					\end{lemma}

					\begin{remark} \label{rem:continuity:F:to:ext:F}
						From Remark \ref{remark:multi:affine:extension}, we know that 
						\[
						\pi_{\subdiv^{(m)}}(f) = 
						P_{ \clos_{\subdiv^{(m)}} \to [0,1]^d} (f),
						\]
						where $\clos_{\subdiv^{(m)}} =  \prod_{j=1}^d ( \subdiv_j^{(m)} \cap [0,1] )$. The set  $\clos_{\subdiv^{(m)}}$ is contained in $\clos$ and converges to $\clos$, for the Hausdorff distance, as $m \to \infty$. Hence, Lemma \ref{lem:pi:m:to:ext} is a continuity property of the application $\clos' \mapsto P_{ \clos' \to [0,1]^d}(f)  $, for fixed $f$, with respect to $\clos'$ and for the Hausdorff distance.
					\end{remark}

					\begin{proof}[{\bf Proof of Lemma \ref{lem:pi:m:to:ext}}]
						Let $f \in \mathcal{C}(\clos , \mathbb{R})$. Assume that as $m \to \infty$, $\pi_{\subdiv^{(m)}}(f) - \extfun (f)$ does not go to zero. Then, up to extracting subsequences and by compacity of $[0,1]^d$, there exist $\epsilon >0$ and a sequence $(x^{(m)})_{m \geq m_0}$ converging to $x^{(\infty)} \in \domain$ such that
						\begin{equation} \label{eq:to:be:contradicted:pi:minus:ext}
							\left| \pi_{\subdiv^{(m)}}(f)(x^{(m)}) - \extfun (f)(x^{(m)}) \right| \geq \epsilon.
						\end{equation}
						Let us now contradict \eqref{eq:to:be:contradicted:pi:minus:ext}.
						Similarly to the proof of Lemma \ref{lem:adapted:from:bay:et:al}, we introduce notation for the current subdivision and the current left and right neighbors of $x^{(m)}_1 , \ldots , x^{(m)}_d$.
						For $j \in \{1 , \ldots , d\}$ and $m \geq m_0$, let us write $x^{(m)}_{m,j,-} = \max \{ u \in \subdiv^{(m)}_j; u \leq x^{(m)}_{j}  \}$ 
						and $x^{(m)}_{m,j,+} = \min \{ u \in \subdiv^{(m)}_j; u \geq x^{(m)}_{j}  \}$.
						Also, if $x^{(m)}_j \not \in \clos_j$,  we write 
						$\omega^m_+(x^{(m)}_j) = \frac{x^{(m)}_j - x^{(m)}_{m,j,-}}{x^{(m)}_{m,j,+} - x^{(m)}_{m,j,-}}$ and $\omega^m_-(x^{(m)}_j) = 1- \omega^m_+(x^{(m)}_j)$. If $x^{(m)}_j  \in \clos_j$,  we write 
						$\omega^m_+(x^{(m)}_j) = \omega^m_-(x^{(m)}_j) = 1/2$.

						We then have
						\[
						\pi_{\subdiv^{(m)}}(f) (x^{(m)})
						=
						\sum_{\epsilon_1, \dots, \epsilon_d \in \{-, +\}} 
						\Bigg(\prod_{j = 1}^d \omega^m_{\epsilon_j}(x^{(m)}_j)  \Bigg)
						f( x^{(m)}_{m,1,\epsilon_1} , \ldots , x^{(m)}_{m,d,\epsilon_d} ),
						\]
						from \eqref{eq:multiExtFormula} and \eqref{eq:pi:S:f:equal:multiaffine:extension}.
						Up to extracting subsequences, we can partition $\{1,\ldots,d\}$ as $J_1 \cup J_2 \cup J_3 \cup J_4$ (where some of the four sets are possibly empty) where 
						\begin{itemize}
							\item for $j \in J_1$ we have $x^{(\infty)}_j \not \in \clos_j$,
							\item  for $j \in J_2$ we have $x^{(\infty)}_j \in \clos_j$, $x^{(m)}_{m,j,-} \to x^{(\infty)}_{j}$ and $x^{(m)}_{m,j,+} \to x^{(\infty)}_{j}$, 
							\item for $j \in J_3$  we have $x^{(\infty)}_j \in \clos_j$, $x^{(m)}_{m,j,-} \to x^{(\infty)}_{j}$ and $\liminf | x^{(m)}_{m,j,+} - x^{(\infty)}_{j} | >0$,
							\item and for $j \in J_4$  we have $x^{(\infty)}_j \in \clos_j$,  $\liminf | x^{(m)}_{m,j,-} - x^{(\infty)}_{j} | >0$ and $x^{(m)}_{m,j,+}  \to x^{(\infty)}_{j}$.
						\end{itemize}
						Note that we can indeed find this partition, after extraction of subsequences, because for $x^{(\infty)}_j \in \clos_j$, we have $\liminf |x^{(m)}_{m,j,-} - x^{(\infty)}_{j} | = 0$ or $\liminf |x^{(m)}_{m,j,+} - x^{(\infty)}_{j} | = 0$.
						Up to re-indexing and without loss of generality, we assume that $J_1 = \{1,\ldots,j_1\}$, $J_2 = \{j_1+1,\ldots,j_2\}$, $J_3 = \{j_2+1,\ldots,j_3\}$ and $J_4 = \{j_3+1,\ldots,d\}$.

						We have
						\begin{align} \label{eq:two:sums:piSm:minus:ext}
							& \left|
							\pi_{\subdiv^{(m)}}(f) (x^{(m)})
							-
							\extfun (f)(x^{(m)})
							\right|
							\\
							& = 
							\Bigg|
							\sum_{\epsilon_1, \dots, \epsilon_d \in \{-, +\}} 
							\Bigg(\prod_{j = 1}^d \omega^m_{\epsilon_j}(x^{(m)}_j)  \Bigg)
							f( x^{(m)}_{m,1,\epsilon_1} , \ldots , x^{(m)}_{m,d,\epsilon_d} )
							\notag
							\\
							& ~ ~ -
							\sum_{\epsilon_1, \dots, \epsilon_d \in \{-, +\}} 
							\Bigg(\prod_{j = 1}^d \omega_{\epsilon_j}(x^{(m)}_j)  \Bigg)
							f( x^{(m)}_{1,\epsilon_1} , \ldots , x^{(m)}_{d,\epsilon_d} )
							\Bigg|. \notag
						\end{align}

						In the above display, for $j \in J_{3}$, then $ \liminf | x^{(m)}_{m,j,+} -  x^{(\infty)}_{j} | >0$ and $ x^{(m)}_{m,j,-} \to  x^{(\infty)}_{j}$. 
						Hence one can see that if $\epsilon_j = +$, $w^m_{\epsilon_j} (x_j^{(m)}) \to 0$. Similarly, for $j \in J_4$, if $\epsilon_j = -$, then $w^m_{\epsilon_j} (x_j^{(m)}) \to 0$.

						Similarly, for $j \in J_3$, if $\epsilon_j = +$, then $w_{\epsilon_j} (x_j^{(m)}) \to 0$ and for $j \in J_4$, if $\epsilon_j = -$, then $w_{\epsilon_j} (x_j^{(m)}) \to 0$.

						Hence, in the two sums in \eqref{eq:two:sums:piSm:minus:ext}, if $(\epsilon_1,\ldots,\epsilon_d)$ is such that $\epsilon_j = +$ for $j \in J_3$, or $\epsilon_j = -$ for $j \in J_4$, then the corresponding summand goes to zero.
						As a consequence we have
						\begin{align*}
							& \left|
							\pi_{\subdiv^{(m)}}(f) (x^{(m)})
							-
							\extfun (f)(x^{(m)})
							\right|
							\\
							& = 
							o(1)
							+
							\\
							&
							\Bigg|
							\sum_{\epsilon_1, \dots, \epsilon_{j_2} \in \{-, +\}} 
							\Bigg(\prod_{j = 1}^{j_2} \omega^m_{\epsilon_j}(x^{(m)}_j)  \Bigg)
							f( x^{(m)}_{m,1,\epsilon_1} , \ldots , x^{(m)}_{m,j_2,\epsilon_{j_2}},
							\\
							&
							\hspace{6cm}
							x^{(m)}_{m,j_2+1,-},\ldots,x^{(m)}_{m,j_3,-} ,  x^{(m)}_{m,j_3+1,+},\ldots,x^{(m)}_{m,j_4,+})
							\Bigg.
							\\
							&
							\Bigg.
							-
							\sum_{\epsilon_1, \dots, \epsilon_{j_2} \in \{-, +\}} 
							\Bigg(\prod_{j = 1}^{j_2} \omega_{\epsilon_j}(x^{(m)}_j)  \Bigg)
							f( x^{(m)}_{1,\epsilon_1} , \ldots , x^{(m)}_{j_2,\epsilon_{j_2}},x^{(m)}_{j_2+1,-},\ldots,x^{(m)}_{j_3,-} ,  x^{(m)}_{j_3+1,+},\ldots,x^{(m)}_{j_4,+})
							\Bigg|.
						\end{align*}

						For $j \in \{j_1+1 , \ldots , j_2\}$, we have $x^{(\infty)}_j \in \clos_j$, $x^{(m)}_{m,j,-} \to x^{(\infty)}_{j}$ and $x^{(m)}_{m,j,+} \to x^{(\infty)}_{j}$. As a consequence, because $ x^{(m)}_{m,j,-} \leq x^{(m)}_{j,-} \leq x^{(m)}_{j} \leq x^{(m)}_{j,+} \leq x^{(m)}_{m,j,+}$, also $x^{(m)}_{j,-} \to x^{(\infty)}_{j}$ and $x^{(m)}_{j,+} \to x^{(\infty)}_{j}$.
						Similarly, for $j \in \{j_2+1 , \ldots , j_3\}$, we have $x^{(m)}_{m,j,-} \to x^{(\infty)}_{j}$ and so also $x^{(m)}_{j,-} \to x^{(\infty)}_{j}$. Similarly, for $j \in \{j_3+1 , \ldots , d\}$, we have $x^{(m)}_{m,j,+} \to x^{(\infty)}_{j}$ and so also $x^{(m)}_{j,+} \to x^{(\infty)}_{j}$.

						This yields, by continuity of $f$,
						\begin{align*}
							& \left|
							\pi_{\subdiv^{(m)}}(f) (x^{(m)})
							-
							\extfun (f)(x^{(m)})
							\right|
							\\
							& = 
							o(1)
							+
							\Bigg|
							\sum_{\epsilon_1, \dots, \epsilon_{j_2} \in \{-, +\}} 
							\Bigg(\prod_{j = 1}^{j_2} \omega^m_{\epsilon_j}(x^{(m)}_j)  \Bigg)
							\\
							&
							f( x^{(m)}_{m,1,\epsilon_1} , \ldots , x^{(m)}_{m,j_1,\epsilon_{j_1}},
							x^{(\infty)}_{j_1+1} , \ldots , x^{(\infty)}_{j_2},
							x^{(\infty)}_{j_2+1},\ldots,x^{(\infty)}_{j_3} ,  x^{(\infty)}_{j_3+1},\ldots,x^{(\infty)}_{j_4})
							\\
							&
							-
							\sum_{\epsilon_1, \dots, \epsilon_{j_2} \in \{-, +\}} 
							\Bigg(\prod_{j = 1}^{j_2} \omega_{\epsilon_j}(x^{(m)}_j)  \Bigg)
							\\
							&
							f( x^{(m)}_{1,\epsilon_1} , \ldots , x^{(m)}_{j_1,\epsilon_{j_1}},
							x^{(\infty)}_{j_1+1} , \ldots , x^{(\infty)}_{j_2},
							x^{(\infty)}_{j_2+1},\ldots,x^{(\infty)}_{j_3} ,  x^{(\infty)}_{j_3+1},\ldots,x^{(\infty)}_{j_4})
							\Bigg|.
						\end{align*}

						In the above display, for the first sum, we can isolate 
						$$
						\sum_{\epsilon_{j_1+1}, \dots, \epsilon_{j_2} \in \{-, +\}} 
						(\prod_{j = j_1+1}^{j_2} \omega^m_{\epsilon_j}(x^{(m)}_j) ) = 1,
						$$
						because the argument of $f$ does not depend on $\epsilon_{j_1+1}, \dots, \epsilon_{j_2}$. We can proceed similarly with the second sum.
						This yields
						\begin{align*}
							& \left|
							\pi_{\subdiv^{(m)}}(f) (x^{(m)})
							-
							\extfun (f)(x^{(m)})
							\right|
							\\
							& = 
							o(1)
							+
							\Bigg|
							\sum_{\epsilon_1, \dots, \epsilon_{j_1} \in \{-, +\}} 
							\Bigg(\prod_{j = 1}^{j_1} \omega^m_{\epsilon_j}(x^{(m)}_j)  \Bigg)
							\\
							&
							f( x^{(m)}_{m,1,\epsilon_1} , \ldots , x^{(m)}_{m,j_1,\epsilon_{j_1}},
							x^{(\infty)}_{j_1+1} , \ldots , x^{(\infty)}_{j_2},
							x^{(\infty)}_{j_2+1},\ldots,x^{(\infty)}_{j_3} ,  x^{(\infty)}_{j_3+1},\ldots,x^{(\infty)}_{j_4})
							\\
							&
							-
							\sum_{\epsilon_1, \dots, \epsilon_{j_1} \in \{-, +\}} 
							\Bigg(\prod_{j = 1}^{j_1} \omega_{\epsilon_j}(x^{(m)}_j)  \Bigg)
							\\
							&
							f( x^{(m)}_{1,\epsilon_1} , \ldots , x^{(m)}_{j_1,\epsilon_{j_1}},
							x^{(\infty)}_{j_1+1} , \ldots , x^{(\infty)}_{j_2},
							x^{(\infty)}_{j_2+1},\ldots,x^{(\infty)}_{j_3} ,  x^{(\infty)}_{j_3+1},\ldots,x^{(\infty)}_{j_4})
							\Bigg|.
						\end{align*}

						Now for $j=1,\ldots,j_1$, we have $x^{(\infty)}_{j,-} < x^{(\infty)}_j < x^{(\infty)}_{j,+}$. Hence by continuity, we have
						\begin{align*}
							& \left|
							\pi_{\subdiv^{(m)}}(f) (x^{(m)})
							-
							\extfun (f)(x^{(m)})
							\right|
							\\
							& = 
							o(1)
							+
							\Bigg|
							\sum_{\epsilon_1, \dots, \epsilon_{j_1} \in \{-, +\}} 
							\Bigg(\prod_{j = 1}^{j_1} \omega_{\epsilon_j}(x^{(\infty)}_j)  \Bigg)
							\\
							&
							f( x^{(\infty)}_{1,\epsilon_1} , \ldots , x^{(\infty)}_{j_1,\epsilon_{j_1}},
							x^{(\infty)}_{j_1+1} , \ldots , x^{(\infty)}_{j_2},
							x^{(\infty)}_{j_2+1},\ldots,x^{(\infty)}_{j_3} ,  x^{(\infty)}_{j_3+1},\ldots,x^{(\infty)}_{j_4})
							\\
							&
							-
							\sum_{\epsilon_1, \dots, \epsilon_{j_1} \in \{-, +\}} 
							\Bigg(\prod_{j = 1}^{j_1} \omega_{\epsilon_j}(x^{(\infty)}_j)  \Bigg)
							\\
							&
							f( x^{(\infty)}_{1,\epsilon_1} , \ldots , x^{(\infty)}_{j_1,\epsilon_{j_1}},
							x^{(\infty)}_{j_1+1} , \ldots , x^{(\infty)}_{j_2},
							x^{(\infty)}_{j_2+1},\ldots,x^{(\infty)}_{j_3} ,  x^{(\infty)}_{j_3+1},\ldots,x^{(\infty)}_{j_4})
							\Bigg|.
						\end{align*}
						
						Notice that the two sums above are the same. Hence 
						\[
						\left|
						\pi_{\subdiv^{(m)}}(f) (x^{(m)})
						-
						\extfun (f)(x^{(m)})
						\right|
						=
						o(1).
						\]
						This is in contradiction with \eqref{eq:to:be:contradicted:pi:minus:ext} which concludes the proof.
					\end{proof}
					
					\begin{proof}[{\bf Proof of Theorem \ref{theorem:convergence}}]
						
						Since the sequence of sets $\activeset_m$ is nested, it is equal to a set $\activeset_{\infty}$ for $m$ larger than some $m_0 \in \mathbb{N}$. Let us consider $m \geq m_0$ for the rest of the proof. Let $d = |\activeset_{\infty}|$. Then, with the set of variables $\activeset_{\infty}$, and the sequence of subdivisions $\left(\subdiv^{(m)}\right)_{m \geq m_0}$ in $\subdivset_{\activeset_{\infty}}$, we can check that the conditions of Theorem \ref{theorem:convergence} imply the conditions of Theorem \ref{theorem:extension:convergence:spline}.
						
						Hence, from Theorem \ref{theorem:extension:convergence:spline}, as $m \to \infty$, 
						the sequence of functions $\widehat{Y}_{\activeset_m,\subdiv^{(m)}}$ 
						converges uniformly to a limit function. Hence, as $m \to \infty$, we have, 
						\[
						I_{\activeset_m, S^{(m)}}(i^\star_{m+1},t^\star_{m+1})
						\to 0.
						\]
						
						Assume now that there exists $i \in \activeset_{\infty}$ such that $S^{(m)}_i$ is not dense in $[0,1]$. 
						As for $t_i \in [0,1]$, the sequence $d ( t_i , S^{(m)}_i )$ is decreasing (with respect to $m$) and thus has a limit, 
						this implies that there exists $\epsilon >0$ and $t_i \in [0,1]$ 
						such that $d ( t_i , S^{(m)}_i ) \geq \epsilon$ for all $m \in \mathbb{N}$. 
						This means that for $m$ large enough such that $b_m \leq \epsilon$,
						\[
						\underset{
							\substack{
								i \in \activeset_{m}, \, t \in [0,1], \\
								d(t , \subdiv^{(m)}_i) \geq b_m
							}
						}{\sup}
						\left(
						I_{\activeset_m, S^{(m)}}(i ,t) + \Delta d(t, \subdiv^{(m)}_{i})
						\right)
						\geq \Delta \epsilon.
						\]
						As $m \geq m_0$, then $i^\star_{m+1} \in \activeset_m = \activeset_\infty$. Thus, by definition of algorithm \ref{alg:iterative},
						\[
						I_{\activeset_m, S^{(m)}}(i^\star_{m+1} ,t^\star_{m+1}) + \Delta \, d\big(t^\star_{m+1}, \subdiv^{(m)}_{i^\star_{m+1}} \big) + a_m
						\geq \Delta \epsilon.
						\]
						Hence, we will reach a contradiction if we show that $
						\liminf_{m \to \infty} d(t^\star_{m+1}, \subdiv^{(m)}_{i^\star_{m+1}}) = 0$. 
						There is at least one coordinate  $i \in \activeset_m$ chosen an infinite number of times by MaxMod algorithm.
						Let $(m_\ell)_{\ell \in \mathbb{N}}$ be the corresponding subsequence of values of $m$, i.e. for which $i^\star_{m+1} = i$. 
						Then we have $\{ t^\star_{m_1+1} , \ldots , t^{\star}_{m_{\ell}+1} \} \subseteq S^{(m_{\ell}+1)}_{i} $ 
						and thus 
						$d \big( t^\star_{m_\ell+1},  S^{(m_{\ell}+1)}_{i}  \big) \leq 
						d \big( t^\star_{m_\ell+1} , \{t^{\star}_{m_{1}} , \ldots , t^{\star}_{m_{\ell}} \} \big)$. 
						The limit inferior of this last quantity is zero, 
						by considering a convergent subsequence of $(t^{\star}_{m_{\ell}})_\ell$ in the compact interval $[0, 1]$.
						Hence, we have reached a contradiction and, eventually,  for all $i \in \{1 , \ldots , d\}$, $S^{(m)}_i$ is dense in $[0,1]$.
						
						Let us now assume that $d < D$. Let us then consider $j_{\infty} \in \{1 , \ldots , D\} \backslash \activeset$. We then have by definition of the MaxMod algorithm, for $m$ larger than $m_0$,
						\[
						I_{\activeset_m,S^{(m)}}(i^\star_{m+1} ,t^\star_{m+1}) + \Delta \, d\big(t^\star_{m+1}, \subdiv^{(m)}_{i^\star_{m+1}} \big) + a_m
						\geq 
						\Delta',
						\] 
						by considering $i = j_\infty \not \in \activeset_{\infty}$ in \eqref{eq:max:mod}.
						The left hand side of the above display goes to zero, because for $i \in \activeset_{\infty}$, $S_i^{(m)}$ is dense in $[0,1]$, as we have shown and also because 
						\[
						I_{\activeset_m, S^{(m)}}(i^\star_{m+1} ,t^\star_{m+1}) \to 0 
						\]
						as we have shown.
						This yields a contradiction. Hence $\activeset_m = \{1 , \ldots , D\}$ for $m \geq m_0$.
						
						From the fact that all the variables are eventually active and from the density of $S_i^{(m)}$ for $i \in \{1 , \ldots , D\}$, we conclude from Theorem
						\ref{theorem:convergence:spline} (Theorem 3.2 in \cite{bay2017new}).
					\end{proof}

\bibliographystyle{abbrv}
\bibliography{Biblio}

\end{document}